\documentclass[a4paper,10pt,reqno]{smfart}
\usepackage[T1]{fontenc}
\usepackage[utf8]{inputenc}
\usepackage{mdwlist}
\usepackage{url}
\usepackage{amsmath} 
\usepackage{amsfonts}
\usepackage{amscd} 
\usepackage{amssymb} 
\usepackage{enumerate}
\usepackage{mathrsfs}
\usepackage[francais]{babel}
\usepackage[all]{xy}
\newtheorem{thm}{Th\'eor\`eme}[section]
\newtheorem{prop}[thm]{Proposition}
\newtheorem{lemme}[thm]{Lemme}
\newtheorem{cor}[thm]{Corollaire}
\newtheorem{question}[thm]{Question}
\newtheorem{hyp}[thm]{Hypoth\`ese}

\newtheorem{hyps}[thm]{Hypoth\`eses}

\theoremstyle{definition}
\newtheorem{defi}[thm]{D\'efinition}
\newtheorem{notas}[thm]{Notations}
\newtheorem{nota}[thm]{Notation}
\theoremstyle{remark}

\newtheorem{rem}[thm]{Remarque}

\makeatletter
\@addtoreset{equation}{subsection}
\renewcommand{\theequation}{\arabic{section}.\arabic{subsection}.\arabic{equation}}
\makeatother
\newcommand{\termin}[1]{{\em #1}}
\newcommand{\bA}{\mathbf{A}}\newcommand{\bC}{\mathbf{C}}

\newcommand{\bL}{\mathbf{L}}
\newcommand{\bN}{\mathbf{N}}\newcommand{\bP}{\mathbf{P}}
\newcommand{\bQ}{\mathbf{Q}}

\newcommand{\bZ}{\mathbf{Z}}

\newcommand{\cC}{{\mathcal C}}\newcommand{\cD}{{\mathcal D}}
\newcommand{\cH}{{\mathcal H}}
\newcommand{\cI}{{\mathcal I}}\newcommand{\cJ}{{\mathcal J}}
\newcommand{\cM}{{\mathcal M}}

\DeclareMathAlphabet{\eulercal}{U}{eus}{m}{n}

\newcommand{\ecC}{{\eulercal C}}
\newcommand{\ecF}{{\eulercal F}}
\newcommand{\ecG}{{\eulercal G}}\newcommand{\ecH}{{\eulercal H}}

\newcommand{\ecO}{{\eulercal O}}\newcommand{\ecP}{{\eulercal P}}

\DeclareMathAlphabet{\beulercal}{U}{eus}{b}{n}

\newcommand{\mfp}{\mathfrak{p}}

\newcommand{\scD}{\mathscr{D}}

\newcommand{\scX}{\mathscr{X}}

\newcommand{\imply}{\Rightarrow}
\newcommand{\longto}{\longrightarrow}
\newcommand{\longlto}{\longleftarrow}
\newcommand{\isom}{\overset{\sim}{\to}}

\DeclareMathOperator{\End}{End}
\DeclareMathOperator{\Vect}{Vect}
\DeclareMathOperator{\Pic}{Pic}

\DeclareMathOperator{\Br}{Br}

\DeclareMathOperator{\Aut}{Aut}
\DeclareMathOperator{\Sym}{Sym}
\DeclareMathOperator{\Alt}{Alt}
\DeclareMathOperator{\rg}{rg}

\DeclareMathOperator{\Ker}{Ker}
\DeclareMathOperator{\im}{Im}
\DeclareMathOperator{\Hom}{Hom}

\DeclareMathOperator{\Sup}{Sup}

\DeclareMathOperator{\Max}{Max}

\DeclareMathOperator{\Spec}{Spec}

\DeclareMathOperator{\Gal}{Gal}

\DeclareMathOperator{\Tr}{Tr}

\DeclareMathOperator{\Id}{Id}

\DeclareMathOperator{\Irr}{Irr}
\DeclareMathOperator{\Ind}{Ind}

\let\leq\leqslant
\let\geq\geqslant
\newcommand{\sumu}[1]{\underset{#1}{\sum}}
\newcommand{\produ}[1]{\underset{#1}{\prod}}

\newcommand{\oplusu}[1]{\underset{#1}{\oplus}}

\newcommand{\wedgeo}[1]{\overset{#1}{\wedge}}

\newcommand{\disju}[1]{\underset{#1}{\bigsqcup}}
\newcommand{\veeu}[1]{\underset{#1}{\vee}}

\newcommand{\eps}{\varepsilon}
\newcommand{\vide}{\varnothing}

\newcommand{\eqdef}{\overset{\text{{\tiny{d\'ef}}}}{=}}

\newcommand{\ind}{\mathbf{1}}

\newcommand{\wt}{\widetilde}

\newcommand{\wh}{\widehat}

\newcommand{\sym}{\mathfrak{S}}

\newcommand{\abs}[1]{\left| #1 \right|} 
\newcommand{\card}[1]{\left[#1\right]} 
 
\newcommand{\map}[4]{
\begin{array}{rcl}
#1 & \longto & #2 \\
#3 & \longmapsto & #4 
\end{array} 
}
\newcommand{\bigo}[2]{\ecO_{#1}\left(#2\right)} 
\newcommand{\inv}{\times}

\DeclareMathOperator{\triv}{triv}
\DeclareMathOperator{\ChMot}{CHM}

\DeclareMathOperator{\MA}{MA}
\DeclareMathOperator{\Var}{Var}
\DeclareMathOperator{\Fr}{Fr}
\DeclareMathOperator{\Poinc}{Poinc}
\DeclareMathOperator{\Poincabs}{Poinc_{abs}}
\DeclareMathOperator{\Conju}{Conj}
\DeclareMathOperator{\bConj}{\textbf{Conj}}

\DeclareMathOperator{\Vol}{\mathscr{V}}
\DeclareMathOperator{\Volm}{\mathscr{V}_{\text{{\tiny mot}}}}

\DeclareMathOperator{\Trp}{\Tr_{\mfp}}
\DeclareMathOperator{\Ar}{Ar}
\DeclareMathOperator{\GrVect}{GrVect}
\newcommand{\classe}[1]{\left[#1\right]}
\newcommand{\str}[1]{\mathscr{O}_{#1}}
\newcommand{\rhop}{\rho^{\text{op}}}
\newcommand{\rhoNS}{\rho_{_{\text{NS}}}}
\newcommand{\rhoNSp}{\rho_{_{\text{NS},\mfp}}}
\newcommand{\Gop}{G^{\text{op}}}
\newcommand{\sep}[1]{#1^{\text{s}}}
\newcommand{\clo}[1]{\overline{#1}}
\renewcommand{\card}[1]{\left|#1\right|} 
\newcommand{\Poincl}{\Poinc_{\ell}}
\newcommand{\Hl}{H_{\ell}}
\newcommand{\polcar}{\mathscr{P}}

\newcommand{\mk}{\kochmk{\bQ}}
\newcommand{\mkq}{\kochmkq{\bQ}}
\newcommand{\hml}{\wh{[\mkq]}^{\Poincl}}
\newcommand{\hmabs}{\wh{\mkq^{\Poincabs}}}
\newcommand{\Frp}{\Fr_{\mfp}}

\newcommand{\fil}{\ecF_{\varphi}}
\newcommand{\kpff}{K_0(\text{PFF}_k)}
\newcommand{\Ql}{\bQ_{\ell}}
\newcommand{\Zl}{\bZ_{\ell}}
\newcommand{\galabs}[1]{\ecG_{#1}}
\newcommand{\gkql}{\galabs{k}\text{-}\Ql}

\newcommand{\GQvect}{G\text{-}\bQ\text{-}\Vect}
\newcommand{\kogkql}{K_0(\gkql)}
\newcommand{\kogkqlq}{K_0(\gkql)\otimes{\bQ}}
\newcommand{\courbe}{\mathscr{C}}
\newcommand{\irrq}[1]{\Irr_{\bQ}(#1)}
\newcommand{\irrE}[1]{\Irr_{E}(#1)}
\newcommand{\Conj}[1]{\Conju(#1)}
\newcommand{\Conjc}[1]{\Conju_\text{c}(#1)}
\newcommand{\bconj}[1]{\bConj(#1)}
\newcommand{\bconjc}[1]{\bConj_\text{c}(#1)}
\newcommand{\chm}[2]{\ChMot(#1)_{#2}}
\newcommand{\chmk}[1]{\chm{k}{#1}}
\newcommand{\AMk}[1]{\MA(k)_{#1}}
\newcommand{\kochm}[2]{K_0\left(\chm{#1}{#2}\right)}
\newcommand{\kochmk}[1]{\kochm{k}{#1}}
\newcommand{\kochmkq}[1]{K_0\left(\chmk{#1}\right)\otimes {\bQ}}

\newcommand{\koAMk}[1]{K_0\left(\AMk{#1}\right)}
\newcommand{\kovchmk}[1]{K_0^{\textnormal{\scriptsize var}}\!\left(\chmk{#1}\right)}
\newcommand{\kovchmkq}[1]{K_0^{\textnormal{\scriptsize var}}\!\left(\chmk{#1}\right)_{\bQ}}
\newcommand{\symb}[1]{\left[#1\right]}
\newcommand{\symbv}[1]{\left[#1\right]_{\textnormal{\scriptsize var}}}
\newcommand{\frev}[1]{\varphi_{{}_{#1}}}
\newcommand{\etak}[1]{\eta_{{}_{#1}}}
\newcommand{\psin}[1]{\psi_{{}_{#1}}}
\newcommand{\chil}{\chi_{\ell}}
\newcommand{\chif}{\chi_{{}_{\text{\textnormal{\scriptsize form}}}}}
\newcommand{\chiv}{\chi_{{}_{\text{\textnormal{\scriptsize var}}}}}
\newcommand{\chieq}{\chi_{{}_{\text{\textnormal{\scriptsize eq}}}}}
\newcommand{\vark}{\Var_k}
\newcommand{\varC}{\Var_{\bC}}
\newcommand{\kovark}{K_0(\vark)}
\newcommand{\kovarC}{K_0(\varC)}
\newcommand{\gvark}{G\text{-}\Var_k}
\newcommand{\kogvark}{K_0(\gvark)}
\newcommand{\gk}[1]{\galabs{k}\text{-}#1}
\newcommand{\kogk}[1]{K_0(\gk{#1})}
\newcommand{\Zv}{Z_{\text{\textnormal{\scriptsize{var}}}}}
\newcommand{\Zm}{Z_{\text{\textnormal{\scriptsize{mot}}}}}
\newcommand{\ZHW}{Z_{\text{HW}}}
\newcommand{\ZHm}{Z_{\text{H},\text{\textnormal{\scriptsize{mot}}}}}
\newcommand{\ZHv}{Z_{\text{H},\text{\textnormal{\scriptsize{var}}}}}
\newcommand{\ZH}{Z_{\text{H}}}
\newcommand{\ZHp}{Z_{\text{H},\mfp}}
\newcommand{\Lm}{L_{\text{\textnormal{\scriptsize{mot}}}}}

\newcommand{\Lar}{L_{\text{Ar}}}
\newcommand{\LmDM}{L^{\text{DM}}_{\textnormal{\scriptsize{mot}}}}
\newcommand{\LmDMnr}{L^{\text{DM,nr}}_{\textnormal{\scriptsize{mot}}}}

\newcommand{\phinv}{\Phi_{n,\text{\textnormal{\scriptsize{var}}}}}
\newcommand{\phinvar}[1]{\phinv(#1)}
\newcommand{\psinvar}[1]{\psi_{n,\text{\textnormal{\scriptsize{var}}}}(#1)}
\newcommand{\psidvar}[1]{\psi_{d,\text{\textnormal{\scriptsize{var}}}}(#1)}
\newcommand{\fICJD}{f^{\cI,\cC}_{\cJ,\cD}}
\newcommand{\dICJD}{d^{\cI,\cC}_{\cJ,\cD}}
\newcommand{\psinx}[1]{\psi_n(#1)}
\newcommand{\psidx}[1]{\psi_d(#1)}
\title[Nombre de Tamagawa motivique]{
Fonctions $L$ d'Artin et nombre de Tamagawa motiviques
}
\alttitle{
Motivic Artin $L$-functions and a motivic Tamagawa number
}
\author{David Bourqui}
\address{I.R.M.A.R\\Campus de Beaulieu\\
35042 Rennes cedex \\ France}
\email{david.bourqui@univ-rennes1.fr}
\date{}
\begin{document}
\frontmatter
\begin{abstract}
Dans la premi\`ere partie de ce texte, nous d\'efinissons des fonctions
$L$ d'Artin motivique \`a l'aide d'un produit eulerien motivique, et
montrons qu'elles co\"\i ncident avec les fonctions introduites
par Dhillon et Minac dans \cite{DhMi:motivic_chebotarev}. Dans la seconde partie, nous d\'efinissons, sous
certaines conditions,  un nombre de Tamagawa motivique et montrons qu'il se sp\'ecialise sur le nombre de
Tamagawa usuel d\'efini par Peyre dans le cadre des conjectures de Manin
sur le nombre de points de hauteur born\'ee des vari\'et\'es de Fano.
\end{abstract}
\begin{altabstract}
In the first part of this text,  
we define motivic Artin $L$-fonctions 
via a motivic Euler product,
and show that they coincide with
the functions introduced by Dhillon and Minac dans \cite{DhMi:motivic_chebotarev}. In the second part,
we define under some assumptions a motivic Tamagawa number and show
that it specializes to the Tamagawa number introduced by Peyre in the context
of Manin's conjectures about rational points of bounded height on Fano varieties.
\end{altabstract} 
\subjclass{
14G10 14C35  (11M41 12E30 14J45) 
}
\keywords{Fonction $L$ d'Artin motivique, nombre de Tamagawa, nombre de Tamagawa
  motivique, produit eulerien motivique, fonction zeta des hauteurs}
\altkeywords{Motivic Artin $L$-function, Tamagawa number, motivic Tamagawa number, 
 motivic Euler product, height zeta function}
\thanks{Je remercie Florian Ivorra pour de tr\`es utiles discussions}
\maketitle
\mainmatter
\section{Introduction}
Comme l'ont illustr\'e Denef et Loeser dans \cite{DeLo:rat_gen_series},
les propri\'et\'es de nombre de s\'eries rationnelles issues de la g\'eom\'etrie arithm\'etique
sont de nature motivique : elles s'obtiennent naturellement par sp\'ecialisation de
s\'eries \`a coefficients dans un anneau de Grothendieck de motifs et
leur propri\'et\'es se lisent d\'ej\`a (au moins conjecturalement) sur ces
s\'eries motiviques. Dans la m\^eme veine, on peut se demander si
les propri\'et\'es des fonctions z\^eta des hauteurs, \'etudi\'ees dans le cadre
des conjectures de Manin sur les points de hauteur born\'ee (cf. par
exemple \cite{Pey:bordeaux} et \cite{Pey:bki}) sont de nature
motivique. Il est \`a noter qu'en g\'en\'eral on ne s'attend pas \`a ce que  de telles s\'eries soient 
rationnelles (cf. \cite[{\em in fine}]{BaTs:anis}).

Dans ce texte, nous montrons que l'on peut, dans certains cas, 
donner une version motivique naturelle du nombre de Tamagawa d\'efini
par Peyre
qui appara\^\i t
conjecturalement dans la partie principale de la fonction z\^eta des hauteurs.
Dans le cas classique, le volume ad\'elique d\'efinissant ce nombre de
Tamagawa peut s'exprimer comme un
produit eulerien. 
L'analogue motivique que nous proposons s'exprime comme un 
 \og produit eulerien motivique \fg~(notion qui appara\^\i t dans un
pr\'ec\'edent travail \cite{Bou:prod:eul:mot} consacr\'e aux fonctions z\^eta
des hauteurs motiviques des vari\'et\'es toriques),
dont on montre la convergence dans une certaine compl\'etion de l'anneau
de Grothendieck des motifs (th\'eor\`eme \ref{thm:princ}).
Cette compl\'etion est bas\'ee sur la filtration par le degr\'e du polyn\^ome de Poincar\'e virtuel
$\ell$-adique (i.e. par le poids). 
Un de ses int\'er\^ets est que la r\'ealisation \og comptage des points \fg~s'\'etend \`a certains \'el\'ements de la compl\'etion.
Nous remarquons qu'une approche similaire est utilis\'ee dans
\cite{BehDhi:motivic} et \cite{Eke:class}. 
Dans le cas d'un corps global, nous montrons que le nombre de Tamagawa motivique se
sp\'ecialise en presque toute place sur le nombre de Tamagawa classique
(th\'eor\`eme \ref{thm:princ:spe}).
Dans le cas d'un corps fini, nous montrons que le nombre de Tamagawa motivique se
sp\'ecialise  sur le nombre de Tamagawa classique (modulo une hypoth\`ese malhereusement
peu naturelle cf. th\'eor\`eme \ref{thm:princ:spe:fin}  et remarque \ref{rem:princ:spe:fin}).
Enfin dans le cas d'une surface, utilisant un r\'esultat de Kahn, Murre
et Pedrini nous donnons une version purement
motivique du nombre de Tamagawa motivique, c'est-\`a-dire que sa
convergence est d\'efinie \`a
l'aide d'un polyn\^ome de Poincar\'e virtuel absolu et non pas
$\ell$-adique (th\'eor\`eme \ref{thm:princ:mot}).

La d\'efinition de Peyre fait intervenir des facteurs de convergence 
qui sont les facteurs locaux de la fonction $L$ d'Artin associ\'ee au module
de Neron-Severi de $X$. Nous avons besoin d'un analogue motivique de
ces facteurs locaux. Une version motivique des fonctions $L$ d'Artin a
\'et\'e propos\'ee par Dhillon et Minac dans
\cite{DhMi:motivic_chebotarev}. Leur construction, quoique compacte et
\'el\'egante,
pr\'esente vis-\`a-vis de notre objectif le d\'efaut de  ne  justement pas faire
intervenir de facteurs locaux. C'est pourquoi nous donnons, dans la
premi\`ere partie de ce texte, une d\'efinition alternative des  fonctions
$L$ motivique via un produit
eulerien motivique.  Nous rappelons et pr\'ecisons les propri\'et\'es de la fonction $L$ de Dhillon et Minac \`a
la section \ref{sec:Lmot:DM}. Dans la section \ref{sec:L:prodeulmot}, 
nous d\'efinissons notre fonction $L$. Nous montrons qu'elle co\"\i ncide avec la fonction $L$ de
Dhillon et Minac et dans le cas d'un corps de nombres se sp\'ecialise en
presque toute place sur la fonction $L$ usuelle. Il est \`a noter que, stricto sensu, les
r\'esultats de la premi\`ere partie ne sont pas utilis\'es dans la seconde
(pour la plupart, ils ne sont d'ailleurs  valables a priori qu'en
caract\'eristique z\'ero, \`a cause notamment de l'utilisation du r\'esultat
de Denef et Loeser permettant d'associer de mani\`ere canonique un motif virtuel \`a une
telle formule, cf. th\'eor\`eme \ref{thm:denefloeser}).
Cependant : 1) ils justifient moralement le fait que les facteurs locaux utilis\'es dans la
d\'efinition du nombre de Tamagawa motivique sont les facteurs \og
naturels \fg~; 2) ils donnent une interpr\'etation arithm\'etique de la
fonction $L$ d'Artin motivique (pour un corps de caract\'eristique z\'ero quelconque)
et 3) ils permettent de d\'ecrire pr\'ecis\'ement les \og p\^oles \fg~de la
fonction $L$ motivique, ce qui est utile pour une formulation
d'une version motivique de la conjecture de Manin
(cf. les remarques \ref{rem:norm:bis} et la section \ref{subsec:lien:conj}).

Pour conclure cette introduction, il faut remarquer que la d\'efinition
propos\'ee du nombre de Tamagawa n'est
pas enti\`erement satisfaisante conceptuellement  : une \og bonne \fg~d\'efinition
devrait certainement utiliser une (hypoth\'etique) version globale de l'int\'egration
motivique (comme le remarquent les auteurs de \cite{BehDhi:motivic}
\`a propos d'une version motivique du nombre de Tamagawa d'un groupe alg\'ebrique).

\tableofcontents

\section{Quelques rappels et notations}
\subsection{Anneaux de Grothendieck de vari\'et\'es et de motifs}
Dans tout ce texte, les actions de groupes sont
des actions \`a gauche.
Si  $G$ est un groupe, on note $\Gop$ le groupe oppos\'e.
Soit $k$ un corps.
On note $\vark$ (respectivement $\gvark$) la cat\'egorie des
vari\'et\'es alg\'ebriques quasi-projectives d\'efinies sur $k$
(respectivement munie d'une action alg\'ebrique d'un groupe fini $G$) 
et $\kovark$ (respectivement
$\kogvark$) son anneau de Grothendieck (cf. \cite[13.1.1]{And:mot}).
Si $F$ est un anneau, on note $\chmk{F}$ la cat\'egorie des motifs de
Chow d\'efinis sur $k$ \`a coefficients dans $F$ (cf. \cite[Chapitre
4]{And:mot}) et $\kochmk{F}$
son anneau de Grothendieck (cf. \cite[13.2.1]{And:mot}).
La classe du motif de Lefschetz $\ind(-1)$  dans $\kochmk{F}$ est
not\'ee $\bL$.  Pour $d\in \bZ$, on note $M(-d)\eqdef M \otimes
\ind(-1)^{\otimes d}$ la $d$-\`eme torsion de Tate de $M$.
\begin{thm}[Gillet-Soul\'e,Guillen-Navarro-Aznar,Bittner]\label{thm:getal}
Soit $k$ un corps de  caract\'eristique z\'ero.
Il existe un unique morphisme d'anneaux
\begin{equation}
\chiv\,:\,\kovark\longto \kochmk{F}
\end{equation}
qui envoie la classe d'une vari\'et\'e projective et lisse $X$
sur la classe de son motif de Chow $h(X)$.
\end{thm}
L'image de $\kovark$ par $\chiv$ sera not\'ee $\kovchmk{F}$.

Notons $\ecC(G,\bQ)$ le $\bQ$-espace vectoriel des fonctions
$\bQ$-centrales de $G$ dans $\bQ$ (i.e les fonctions $\alpha\,:\,G\to
\bQ$ qui v\'erifient $\alpha(x)=\alpha(y)$ d\`es que les sous-groupes 
$\langle x\rangle$ et $\langle y\rangle$ sont conjugu\'es.
On rappelle \`a pr\'esent
un cas particulier d'une version \'equivariante du th\'eor\`eme \ref{thm:getal},
due \`a Denef, Loeser, del Ba\~no et Navarro-Aznar (cf. \cite[theorem 6.1]{dBaNa:quotient}).
\begin{thm}
\label{thm:chiveq}
Soit $k$ un corps de caract\'eristique z\'ero et $G$ un groupe fini.
Il existe une unique famille de morphismes d'anneaux
\begin{equation}
\chieq(-,\alpha)\,:\,
\kogvark\to \kochmk{\bQ}\otimes \bQ
\end{equation}
index\'ee par $\alpha\in \ecC(G,\bQ)$
ayant les propri\'et\'es suivantes :
\begin{enumerate}
\item\label{item:thm:chiveq:projlisse}
si $X$ est une $k$-$G$-vari\'et\'e projective et lisse, $\rho$ une $\bQ$-repr\'esentation
lin\'eaire de dimension finie irr\'eductible de $G$ 
et
$
p_{\rho}\eqdef
\frac{1}{\card{G}}
\sum_{g\in G} \rho(g^{-1})\otimes [g]
$
l'idempotent de $V_{\rho}\otimes h(X)$ associ\'e,
alors on a 
\begin{equation}\label{eq:chieqXchirho}
\chieq(X,\chi_{\rho})=\symb{\im(p_{\rho})} ;
\end{equation}
\item
l'application $\alpha\mapsto \chieq(X,\alpha)$ est un morphisme de groupe.
\end{enumerate}
\end{thm}
\begin{defi}\label{rem:car:non:nulle}\label{def:car:non:nulle}
Si $k$ est un corps de caract\'eristique non nulle, $G$ un groupe fini
et $X$ une $k$-$G$-vari\'et\'e projective et  lisse, on d\'efinit  $\chieq(X,\chi_{\rho})$ 
via la relation \eqref{eq:chieqXchirho}
puis par lin\'earit\'e $\chieq(X,\alpha)$ pour tout \'el\'ement $\alpha$ de $\ecC(G,\bQ)$.
\end{defi}
\begin{thm}[\cite{dBaNa:quotient}]\label{thm:dbna}
Soit $k$ un corps de caract\'eristique z\'ero, $G$ un groupe fini
et $X$ une $k$-$G$-vari\'et\'e projective et lisse.
Alors on a
\begin{equation}
\chiv(X/G)=\symb{h(X)^G}.
\end{equation}
\end{thm}

\subsection{Caract\'eristique d'Euler-Poincar\'e $\ell$-adique et nombre de
  points modulo $\mfp$} 
Pour tout corps $k$, on note $\sep{k}$ une cl\^oture s\'eparable
de $k$ et $\galabs{k}=\Gal(\sep{k}/k)$ le groupe de Galois absolu
de $k$. Pour tout nombre premier $\ell$,
on note $\kogkql$ l'anneau de Grothendieck de la cat\'egorie
des $\Ql$-espaces vectoriels de dimension finie munis d'une action 
continue de $\galabs{k}$. 
On supposera toujours $\ell$ distinct de la caract\'eristique de $k$,
et on fixera un plongement $\Ql\hookrightarrow \bC$.
La
caract\'eristique d'Euler-Poincar\'e $\ell$-adique est le morphisme d'anneaux
\begin{equation}
\chil\,:\,\kovark\longto \kogkql
\end{equation} 
d\'efini par $\chil(\symb{X})=\sum_i (-1)^i \symb{H^i_c(\sep{X},\Ql)}$, o\`u
$\sep{X}\eqdef X\times_k \sep{k}$. Si $k$ est de caract\'eristique z\'ero,
$\chil$ se factorise par $\chiv$.

On suppose \`a pr\'esent que $k$ est un corps global. 
Soit $\mfp$ une place finie de $k$. On note $\kappa_{\mfp}$ son corps r\'esiduel,
$I_{\mfp}\subset \galabs{k}$ un groupe d'inertie en $\mfp$  et
$\Fr_{\mfp}$ un Frobenius en $\mfp$.
Le \termin{nombre de points modulo $\mfp$} d'un \'el\'ement $V$ de $\kogkql$ 
est $\Tr(\Fr_{\mfp}|V^{I_{\mfp}})$. On le notera $\Trp(V)$.
Si $X$ est une $k$-vari\'et\'e, pour presque tout $\mfp$ on a 
\begin{equation}
\Trp(\chil(X))=\card{X(\kappa_{\mfp})},
\end{equation}
o\`u $X(\kappa_{\mfp})$ d\'esigne (abusivement) l'ensemble des
$\kappa_{\mfp}$-points d'un mod\`ele de $X$ (ainsi
$\card{X(\kappa_{\mfp})}$ est bien d\'efini \og modulo un nombre fini de $\mfp$\fg~).

\subsection{Objets de dimension finie et rationnalit\'e}
Pour tout anneau $A$, on note $1+A[[t]]^+$ le sous-groupe de
$A[[t]]^{\inv}$ form\'e des \'el\'ements de terme constant \'egal \`a $1$ et
$1+A[t]^+$ le sous-monoïde des polyn\^omes de
$1+A[[t]]^+$. 
On dit qu'un \'el\'ement $f$ de $1+A[[t]]^+$ est \termin{rationnel} s'il
existe $g\in 1+A[t]^+$ tel que $g\,f\in 1+A[t]^+$.

Soit $\mathscr{A}$ une cat\'egorie tensorielle pseudo-ab\'elienne
$F$-lin\'eaire, o\`u $F$ est une $\bQ$-alg\`ebre.
Soit $G$ un groupe fini, $M$ un objet de $\mathscr{A}$ muni d'une
action de $G$ et $\rho$ une $F$-repr\'esentation lin\'eaire de dimension
finie de $G$. On note
$(M\otimes V_{\rho})^G$ l'image dans $M\otimes V_{\rho}$ du projecteur
$\frac{1}{\card{G}}\sum_{g\in G} g\otimes \rho(g)$.
Dans le cas particulier de l'action de $\sym_n$ sur $M^{\otimes n}$ et
$\rho$ est la repr\'esentation triviale (respectivement la signature),
cette image est not\'ee $\Sym^n M$ (respectivement $\Alt^n M$).
Suivant la terminologie de  \cite{And:bki}, 
un objet  $M$ de $\mathscr{A}$
est dit \termin{pair} (respectivement
\termin{impair}) s'il v\'erifie $\Alt^n M=0$ pour $n>\!\!>0$
(respectivement $\Sym^n M=0$ pour $n>\!\!>0 $. Un objet $M$ de
$\mathscr{A}$ est dit 
\termin{de
dimension finie} s'il s'\'ecrit comme somme directe d'un objet pair et d'un objet impair.
Pour tout objet $M$, on pose
\begin{equation}
Z_{\mathscr{A}}(M,t)\eqdef \sum_{n\geq 0} \symb{\Sym^n M}t^n\in 1+K_0(\mathscr{A})[[t]]^+.
\end{equation}
On a dans $K_0(\mathscr{A})[[t]]$ 
la formule (cf. eg \cite[Lemma 4.1]{Hei:func_eq})
\begin{equation}\label{eq:symnaltn}
Z_{\mathscr{A}}(M,t)\,
\left(\sumu{n\geq 0} \symb{\Alt^n M}(-1)^n\,t^n\right)=1
\end{equation} 
d'o\`u d\'ecoule la proposition suivante.
\begin{prop}[Andr\'e]\label{prop:zm:dimf}
Soit $M$ un objet de $\mathscr{A}$. Si $M$ est pair (respectivement
impair) alors $Z_{\mathscr{A}}(M,t)\in 1+\mathscr{A}[t]^+$
(respectivement $Z_{\mathscr{A}}(M,t)^{-1}\in 1+\mathscr{A}[t]^+$). En particulier,
pour tout objet $M$ de dimension finie, $Z_{\mathscr{A}}(M,t)$ est rationnelle.
\end{prop}

\subsection{Fonctions z\^eta de Hasse-Weil g\'eom\'etrique et motivique}

Soit $k$ un corps et $X$ une $k$-vari\'et\'e quasi-projective. 
On d\'efinit, suivant Kapranov, la fonction z\^eta de Hasse-Weil
g\'eom\'etrique de $X$
\begin{equation}
\Zv(X,t)\eqdef \sum_{n\geq 0} \symb{\Sym^n X}\,t^n\,\in 1+\kovark[[t]]^+.
\end{equation}
Il existe 
un unique morphisme de groupes
\begin{equation}
\Zv(\,.\,,t)\,:\,\kovark\longto 1+\kovark[[t]]^+
\end{equation}
qui envoie la classe d'une vari\'et\'e quasi-projective $X$ sur $\Zv(X,t)$.

Soit $F$ un corps de caract\'eristique z\'ero.
Pour tout objet $M$ de $\chmk{F}$ on d\'efinit, suivant Andr\'e, la fonction z\^eta de
Hasse-Weil motivique de $M$
\begin{equation}
\Zm(M,t)\eqdef Z_{\chmk{F}}(M,t)=\sum_{n\geq 0} \symb{\Sym^n(M)}\,t^n\,\in\,1+\kochmk{F}[[t]]^+.
\end{equation}
On a en particulier, pour tout entier $d$,
\begin{equation}\label{eq:zmd}
\Zm(M(-d),t)=\Zm(M,\bL^d\,t).
\end{equation}
Il existe un unique morphisme de groupes
\begin{equation}
\Zm(\,.\,,t)\,:\,\kochmk{F}\longto 1+\kochmk{F}[[t]]^+
\end{equation}
qui envoie la classe d'un motif $M$ sur $\Zm(M,t)$.

Si $X$ est une vari\'et\'e projective et lisse, on pose
$
\Zm(X,t)\eqdef \Zm(h(X),t).
$
Si $k$ est de caract\'eristique z\'ero, on a d'apr\`es
le th\'eor\`eme \ref{thm:dbna}
\begin{equation}\label{eq:chivzv=zmchiv}
\chiv \circ \Zv(\,.\,,t)=\Zm(\chiv(\,.\,),t).
\end{equation}
Dans ce cas, il existe un unique morphisme de groupes
\begin{equation}
\Zm\,:\,\kovark
\longto 
1+\kochmk{F}[[t]]^+
\end{equation}
qui envoie la classe d'une vari\'et\'e projective et lisse $X$ sur
$\Zm(X,t)$.

\begin{defi}\label{def:phinv}
Soit $M$ un \'el\'ement de $\kochmk{F}$. On d\'efinit la famille de motifs
virtuels $(\Phi_n(M))_{n\geq 1}$ par la relation 
\begin{equation}\label{eq:defphin}
\sum_{n\geq 1} \Phi_n(M)\,\frac{t^n}{n}=t\,\frac{d\log}{dt} \Zm(M,t). 
\end{equation}
Si $X$ est une $k$-vari\'et\'e projective et lisse, on pose 
$\Phi_n(X)\eqdef \Phi_n(h(X))$.
Si $X$ est un \'el\'ement de $\kovark$, on d\'efinit la 
famille de vari\'et\'es virtuelles  $(\phinvar{X})_{n\geq 1}$
par la relation
\begin{equation}\label{eq:defphinvar}
\sum_{n\geq 1} \phinvar{X}\,\frac{t^n}{n}=t\,\frac{d\log}{dt} \Zv(X,t). 
\end{equation}
\end{defi}
\begin{rem}\label{rem:phinv}
D'apr\`es \eqref{eq:chivzv=zmchiv}, si $k$ est de caract\'eristique z\'ero, on a 
\begin{equation}
\chiv \circ \phinv =\Phi_n \circ \chiv.
\end{equation}
Par ailleurs, si $k$ est un corps fini et $X$ une $k$-vari\'et\'e
quasi-projective, 
le morphisme \og nombre de $k$-points\fg~$\kovark\to\bZ$ 
 envoie  $\Zv(X,t)$ sur la fonction z\^eta de Hasse-Weil classique $\ZHW(X)$.
D'apr\`es \eqref{eq:defphinvar}, le nombre de $k$-points de
$\phinvar{X}$ est donc \'egal au nombre de points de $X$ \`a valeurs dans
$k_n$, o\`u $k_n$ est une extension de degr\'e $n$ de $k$. Une remarque similaire
vaut pour $\Phi_n(X)$ si $X$ est projective et lisse.

Comme on a $\card{X\times Y (k_n)}=\card{X(k_n)}.\card{Y(k_n)}$
pour tout $n$, on peut se demander plus g\'en\'eralement (sur un corps $k$
quelconque) si les morphismes de groupes $\Phi_n$ (respectivement $\phinv$)
ne sont pas en fait des morphismes d'anneaux.

Ceci vaut pour $\Phi_n$. 
Je tiens \`a remercier Evgeny Gorsky
qui m'a indiqu\'e l'argument qui suit\footnote{Dans
\cite{Bou:prod:eul:mot},
nous montrons que  $\Phi_n\circ \chiv$ est un morphisme d'anneaux par
une preuve \og arithm\'etique\fg~utilisant le th\'eor\`eme de Denef et
Loeser \ref{thm:denefloeser}.}.
Dans le langage de la th\'eorie des $\lambda$-anneaux, les $\Phi_n$
(respectivement les $\phinv$) sont les op\'erations de Adams 
associ\'ees \`a la structure oppos\'ee \`a la $\lambda$ structure d\'efinie 
par le morphisme $\Zm(\,.\,,t)$ (respectivement $\Zv(\,.\,t)$).
Par ailleurs, Heinloth montre dans \cite{Hei:func_eq} que la structure oppos\'ee  
\`a la $\lambda$-structure d\'efinie par $\Zm$ est sp\'eciale.
D'apr\`es \cite[Proposition 5.1]{AtiTal}, ceci entra\^\i ne que les $\Phi_n$
sont des morphismes de $\lambda$-anneaux, donc en particulier d'anneaux.

Le m\^eme type d'argument permet de montrer, au moins si le corps de
base est $\bC$, que $\phinv$ ne
peut pas toujours \^etre un morphisme d'anneaux. Ceci est implicitement
contenu dans la remarque du d\'ebut la section 8 de \cite{LaLu:rationality_criteria}.
Indiquons les arguments. Soit $\courbe$ une courbe projective, lisse
et connexe de genre sup\'erieur \`a $1$. 
Les auteurs de \cite{LaLu:rationality_criteria} construisent 
un corps $\cH$ de caract\'eristique z\'ero
et un morphisme d'anneaux $\mu\,:\,\kovarC\to\cH$ tel que  
$\mu(\Zv(\courbe\times \courbe,t))$ n'est pas rationnelle (cf. \cite[Section 3]{LaLu:motivic_birational}).
Supposons alors que l'on ait
\begin{equation}
\forall n\geq 1,\quad \phinvar{\courbe\times \courbe}=\phinvar{\courbe}^2.
\end{equation} 
Comme $\cH$ est sans torsion, ceci entra\^\i ne (cf. \cite[Theorem, p. 49]{Knu:lambdaring}) que le morphisme
\begin{equation}
\mu\circ \Zv^{-1}(\,.\,,-t)\,:\,\kovarC\longto 1+\cH[[t]]^+
\end{equation}
envoie $\courbe\times \courbe$ sur le carr\'e de l'image de $\courbe$. 
Rappelons la structure d'anneau 
mise en jeu sur $1+\cH[[t]]^+$ : la loi de groupe additif sur $1+\cH[[t]]^+$
est induite par la multiplication dans $\cH[[t]]$ et la multiplication 
est alors enti\`erement d\'etermin\'ee par la r\`egle 
\begin{equation}
\forall a,b\in \cH,\quad (1+a\,t)\bullet (1+b\,t)=1+a\,b\,t.
\end{equation}
En particulier si $A$ et $B$ sont deux \'el\'ements de $1+\cH[[t]]^+$
qui sont rationnelles, alors $A\bullet B$ l'est encore.
Or, d'apr\`es un r\'esultat de Kapranov (cf. \cite[proposition
13.3.1.2]{And:mot}), 
$
\Zv(\courbe,t)
$
est rationnelle.
Ainsi 
$\mu(\Zv(\courbe\times \courbe,t))=\mu(\Zv(\courbe,t))\bullet \mu(\Zv(\courbe,t))$ est rationnelle, d'o\`u une contradiction.
\end{rem}

\subsection{Motifs d'Artin}\label{subsec:AM}
On note $\AMk{F}$ la cat\'egorie des motifs d'Artin, i.e. la
sous-cat\'egorie de $\chmk{F}$ engendr\'ee par les motifs des $k$-vari\'et\'es de
dimension z\'ero. Rappelons que le foncteur qui au spectre d'une $k$-alg\`ebre \'etale
$K$ associe le $\galabs{k}$-module discret $F^{\Hom_k(K,\sep{k})}$
induit une \'equivalence de cat\'egories
\begin{equation}\label{eq:cat:artin}
\AMk{F}\isom \gk{F}
\end{equation}
o\`u $\gk{F}$
est la cat\'egories des $\galabs{k}$-repr\'esentations discr\`etes \`a valeurs dans des
$F$-espaces vectoriels de dimension finie. On a donc un isomorphisme d'anneaux
canonique $\koAMk{F}\isom \kogk{F}$ au moyen duquel
nous identifierons d\'esormais ces deux anneaux de Grothendieck.

\subsection{Formule de MacDonald motivique}

Soit $F$ un anneau, $K$ un corps contenant $F$,
$\GrVect_K$ la cat\'egorie des $K$-espaces vectoriels gradu\'es
de dimension finie
et 
$
H\,:\,\chmk{F}\longto \GrVect_{K}
$
 une r\'ealisation cohomologique de Weil
(avec \'eventuellement des structures suppl\'ementaires sur 
les objets $\GrVect_{K}$ par exemple l'action du groupe de Galois absolu
dans le cas de la r\'ealisation $\ell$-adique) (cf. \cite[\S 4.2.5 et 7.1.1]{And:mot}).
L'application
$
\Poinc_H\,:\,\kochmk{F}\longto K_0(\Vect_K)[u,u^{-1}]
$
qui \`a $M$ associe $\sumu{i\in \bZ} \symb{H^i(M)}\,u^i$
est alors un morphisme d'anneaux,
que l'on appelle  polyn\^ome de Poincar\'e virtuel (associ\'e \`a la
r\'ealisation cohomologique $H$).

Dans la suite, on ne consid\'erera que des r\'ealisations cohomologiques
\termin{classiques},
au sens de \cite[\S 3.4]{And:mot}. 
Pour $i\in \bZ$, on notera $b_i(M)$
le $i$-\`eme nombre de Betti de $M$, i.e. la dimension du $K$-espace
vectoriel $H^i(M)$ (qui ne d\'epend pas du choix de la cohomologie
classique $H$ d'apr\`es \cite[Th\'eor\`eme 4.2.5.2]{And:mot}).

Le r\'esultat suivant, d\^u \`a del Ba\~no, g\'en\'eralise la formule de
MacDonald calculant les nombres de Betti d'un produit sym\'etrique (\cite{McDon}).
\begin{thm}[del Ba\~no]\label{thm:macdonald:motivique}
Pour tout objet $M$ de $\chmk{F}$, on a 
\begin{equation}
\Poinc_H(\Zm(M,t))
=
\frac
{\produ{i\in \bZ,\,i\text{ impair}}\,\,\,\sumu{n\geq 0} \symb{\wedgeo{n} H^i(M)}\,u^{\,i\,n}\,t^{n}}
{\produ{i\in \bZ,\,i\text{ pair}}\,\,\,\sumu{n\geq 0} \symb{\wedgeo{n} H^i(M)}\,(-1)^n\,u^{\,i\,n}\,t^{n}}
\end{equation}
\end{thm}
\begin{proof}
Ceci d\'ecoule
de la proposition 3.8 de \cite{dBa:moduli},
compte tenu de la formule \eqref{eq:symnaltn}.
\end{proof}
\begin{cor}\label{cor:delbano}
Supposons que $k$ soit un corps global. Soit $X$ une $k$-vari\'et\'e,
suppos\'ee en outre projective et lisse si $k$ est de caract\'eristique
non nulle.
Pour presque tout $\mfp$, 
on a  $\Trp(\chil(\Zm(X,t)))=\ZHW(X_{\mfp},t)$, o\`u $\ZHW$ est la fonction
z\^eta de Hasse-Weil classique de la $\kappa_{\mfp}$-vari\'et\'e $X_{\mfp}$.
\end{cor}
Le th\'eor\`eme \ref{thm:macdonald:motivique} 
va nous permettre de donner une formule 
pour $\Poinc_H(\Phi_d(M))$, qui nous
sera utile pour montrer la convergence du volume de Tamagawa motivique
(cf. th\'eor\`eme \ref{thm:princ}).
\begin{nota}
Pour $n\geq 1$, 
soit $(P_{n,m})_{m\geq 1}$ la famille d'\'el\'ements de
$\bZ[T_1,\dots,T_n]$
d\'efinie par la relation
\begin{equation}\label{eq:rem:prop:pnm}
t\,\frac{d\log}{dt}\left(1+\sum_{1\leq i \leq n} T_i\,t^i\right)=\sum_{m\geq 1} P_{n,m}(T_1,\dots,T_n)\,t^m.
\end{equation}
\end{nota}
\begin{rem}\label{rem:trfpdimvm:trfn}
Si $V$ est un $K$-espace vectoriel de dimension finie 
et $f\in \End(V)$
on a donc l'\'egalit\'e
\begin{equation}
\Tr(f|P_{\dim(V),m}(
\wedgeo{j} V
)_{j=1,\dots,\dim(V)})
=
\Tr(f^n|V)
\end{equation}
\end{rem}
Des  relations \eqref{eq:defphin} et \eqref{eq:rem:prop:pnm}
et du th\'eor\`eme \ref{thm:macdonald:motivique}
on d\'eduit aussit\^ot la proposition suivante.
\begin{prop}\label{prop:poinchPhi_d:mot}\label{prop:poinchPhi_d}
Soit $n\geq 1$. Pour tout objet $M$ de $\chmk{F}$, on a 
\begin{equation}
\Poinc_H(\Phi_n(M))
=\sum_{i\in \bZ}
P_{b_i(M),n}\left(\symb{\wedgeo{j} H^i(M)}\right)_{1\leq j\leq b_i(M)} (-1)^{(n+1)\,i}\,u^{n\,i}
\end{equation}
En particulier, pour toute $k$-vari\'et\'e projective et lisse $X$ 
on a
\begin{equation}\label{eq:poinch:phin}
\Poinc_H(\Phi_n(X))
=\sum_{i=0}^{2\,\dim(X)} 
P_{b_i(X),n}\left(\left[\wedgeo{j} H^i(X)\right]\right)_{1\leq j\leq b_i(X)} (-1)^{(n+1)\,i}\,u^{\,n\,i}
\end{equation}
\end{prop}
\begin{rem}
Si $k$ est un corps fini et $H$ est la r\'ealisation
$\ell$-adique,
en prenant la trace du Frobenius et en faisant $u=-1$ dans la relation 
\eqref{eq:poinch:phin}, on obtient, d'apr\`es la remarque
\ref{rem:trfpdimvm:trfn}, la formule liant le nombre de points de $X$
\`a valeurs dans une extension de degr\'e $n$ de $k$ et la somme altern\'ee
des traces de la puissance $n$-\`eme du Frobenius agissant sur les groupes de
cohomologie $\ell$-adique.
La formule \eqref{eq:poinch:phin} peut donc \^etre vue comme une
g\'en\'eralisation de cette formule de trace.
\end{rem}

\section{La fonction $L$ d'Artin motivique de Dhillon et Minac}\label{sec:Lmot:DM}

\subsection{Une remarque sur les actions de groupes sur les motifs}

Afin de pr\'eciser les r\'esultats de rationalit\'e de \cite{DhMi:motivic_chebotarev}, nous aurons
besoin de la proposition \ref{prop:act:g} ci-dessous, qui est certainement bien connue
des sp\'ecialistes, mais pour laquelle nous n'avons pas trouv\'e de r\'ef\'erence.
Soit $M$ un objet d'une cat\'egorie pseudo-ab\'elienne, $G$ un groupe
agissant sur $M$ et $p$ un idempotent de $M$.
On dit que l'action de
$G$ est \termin{compatible \`a $p$} 
si la relation $p\,g\,p\,h\,p=p\,g\,h\,p$
vaut pour tous $g,h$ de $G$. Dans ce cas l'action de $G$ sur $M$
induit naturellement une action de $G$ sur $\im(p)$, donn\'ee par le
morphisme $g\mapsto p\,g\,p$.
Les deux lemmes ci-dessous sont \'el\'ementaires.
\begin{lemme}\label{lm:act:g:pseudo}
Soit $M$ et $N$ des objets d'une cat\'egorie pseudo-ab\'elienne, $M$
\'etant muni de l'action d'un  groupe $G$.
\begin{enumerate}
\item
On suppose qu'il existe $i\in \Hom(N,M)$ et $r\in \Hom(M,N)$ tels que
$r\,i=\Id_N$. Soit $N'$ le facteur direct de $M$ d\'efinit par la
r\'etraction $r$, i.e. l'image du projecteur $i\,r$.
On suppose que l'application $\psi\,:\,G\to \End(N)$ qui \`a $g$ associe
$r\,g\,i$ v\'erifie $\psi(g\,h)=\psi(g)\,\psi(h)$, ce qui induit une
action de $G$ sur $N$. Alors l'action de $G$ sur $M$ est compatible \`a $i\,r$ et 
l'isomorphisme naturel $i\,:\,N\isom N'$ est $G$-\'equivariant. 
\item
On suppose qu'il existe $p\in \Hom(M,N)$ et $s\in \Hom(N,M)$ tels que
$p\,s=\Id_N$. Soit $N'$ le facteur direct de $M$ d\'efinit par la
section $s$, i.e. l'image du projecteur $s\,p$.
On suppose que l'application $\psi\,:\,G\to \End(N)$ qui \`a $g$ associe
$p\,g\,s$ v\'erifie $\psi(g\,h)=\psi(g)\,\psi(h)$, ce qui induit une
action de $G$ sur $N$.
Alors l'action de $G$ sur $M$ est compatible \`a $s\,p$ et 
l'isomorphisme naturel $s\,:\,N\isom N'$ est $G$-\'equivariant. 
\end{enumerate}
\end{lemme}
\begin{lemme}\label{lm:MG}
Soit $M$ un objet d'une cat\'egorie tensorielle pseudo-ab\'elienne
$F$-lin\'eaire, o\`u $F$ est une $\bQ$-alg\`ebre. Soit $G$ un groupe fini
agissant sur $M$. Soit $p_1,\dots,p_r$ un syst\`eme complet d'idempotents othogonaux de
$M$. Pour tout $i$, on suppose que l'action de $G$ est
compatible \`a $p_i$. On a alors un isomorphisme canonique
\begin{equation}
M^G\isom \oplusu{1\leq i\leq r} \im(p_i)^G
\end{equation}
\end{lemme}

\begin{notas}\label{notas:px}
Soit $X$ une $k$-vari\'et\'e projective, lisse et int\`egre de dimension
$d$. Le \termin{corps des constantes} de $X$ est $k'\eqdef
H^0(X,\ecO_X)$. 
Soit $\pi\,:\,X\to \Spec(k')$ le morphisme naturel.
Soit $i\,:\,k'\to k''$ une extension finie s\'eparable de degr\'e $n$ telle que $X(k'')$
soit non vide, et  $x\,:\Spec(k'')\to X$. On a alors (\cite[\S 1.11]{Sch:mot})
$
i_{\ast}x^{\ast}\pi^{\ast}=n.
$
et
$\pi_{\ast}x_{\ast}i^{\ast}=n$.
Ainsi $\frac{1}{n}i_{\ast}x^{\ast}$
est une r\'etraction de $\pi^{\ast}\,:\,h(\Spec(k'))\to h(X)$, et induit donc un isomorphisme
$\iota_x$ de $h(\Spec(k'))$ sur un facteur direct de $X$, \`a savoir
l'image du projecteur 
$
p_x\eqdef \frac{1}{n}\pi^{\ast}i_{\ast}x^{\ast}
$
qui sera not\'ee $h^0(X)$.
De m\^eme $\frac{1}{n}x_{\ast}i^{\ast}(-d)$ est une section 
de $\pi_{\ast}\,:\,h(X)\to h(\Spec(k'))(-d)$ et induit donc un isomorphisme
$\iota'_x$ de
$h(\Spec(k'))(-d)$ sur un facteur direct de $X$, \`a savoir l'image du projecteur
$
p'_x\eqdef \frac{1}{n}x^{\ast}i^{\ast}(-d)\pi_{\ast},
$
qui sera not\'ee $h^{2\,d}(X)$.
\end{notas}

Supposons \`a pr\'esent qu'un groupe  $G$ agisse sur $X$ par
$k$-automorphismes. Cette action induit par composition \`a gauche une
action de $G$ sur $\Hom(X,\Spec(k'))=\Aut_k(\Spec(k'))$, i.e.
un morphisme de groupe $\psi\,:\,G\to \Aut_k(\Spec(k'))$ d'o\`u
par fonctorialit\'e une action de $\Gop$ sur $h(\Spec(k'))$ et $h(\Spec(k'))(-d)$.

\begin{prop}\label{prop:act:g}
L'action de $\Gop$ sur $h(X)$ d\'eduite par fonctorialit\'e de l'action
de $G$ sur $X$ 
est compatible aux projecteurs $p_x$ et $p'_x$.
L'action induite de $\Gop$ sur $h^0(X)$
(respectivement $h^{2\,d}(X)$) 
co\"\i ncide modulo identification naturelle avec l'action de $\Gop$ sur $h(\Spec(k'))$ 
(respectivment $h(\Spec(k'))(-d)$) induite par $\psi$.

En particulier, si $X$ est
g\'eom\'etriquement int\`egre l'action induite de $\Gop$ sur $h^0(X)$ et
$h^{2\,d}(X)$ est triviale. Si $X$ n'est pas g\'eom\'etriquement int\`egre, l'action de $\Gop$ sur
$h^0(X)$ et $h^{2\,d}(X)$ n'est pas
n\'ecessairement triviale.
\end{prop}
\begin{proof}
Compte tenu du lemme \ref{lm:act:g:pseudo}, 
il suffit de montrer pour tout $g\in G$ les relations
\begin{equation}\label{eq:rel:1}
\frac{1}{n} i_{\ast}\,x^{\ast}\,g^{\ast}\,\pi^{\ast}=\psi(g)^{\ast}
\end{equation}
et
\begin{equation}\label{eq:rel:2}
\frac{1}{n} \pi_{\ast}\,g^{\ast}\,x_{\ast}\,i^{\ast}=\psi(g)^{\ast}.
\end{equation}
La relation \eqref{eq:rel:1} est imm\'ediate compte tenu des \'egalit\'es
$\pi\,g=\psi(g)\,\pi$ et $i_{\ast}\,x^{\ast}\,\pi^{\ast}=n$.
Pour montrer la relation \eqref{eq:rel:2}, on utilise les \'egalit\'es 
$g_{\ast}\,g^{\ast}=\Id_{h(X)}$ et $\psi(g)_{\ast}\,\psi(g)^{\ast}=\Id_{h(\Spec(k'))}$
(\cite[\S 1.10]{Sch:mot}) d'o\`u $g^{\ast}=(g^{-1})_{\ast}$ et
$\psi(g^{-1})_{\ast}=\psi(g)^{\ast}$.
Compte tenu de $\pi\,g^{-1}=\psi(g^{-1})\,\pi$,
il s'ensuit
\begin{equation}
\frac{1}{n} \pi_{\ast}\,g^{\ast}\,x_{\ast}\,i^{\ast}=
\frac{1}{n} \psi(g^{-1})_{\ast}\,\pi_{\ast}\,x_{\ast}\,i^{\ast}=
\psi(g)^{\ast}.
\end{equation}

Si $X$ est g\'eom\'etriquement int\`egre, on a $k'=k$ et $\psi$ est  trivial,
d'o\`u la seconde assertion. Pour montrer la derni\`ere assertion, il suffit de consid\'erer la vari\'et\'e
$X\times_{\Spec(k)} \Spec(k')$, o\`u $X$ est projective, lisse et int\`egre, $k'/k$ est
une extension finie telle que $\Aut_k(k')$ est non trivial et
$G=\Aut_k(k')^{\text{op}}$ agit sur $X\times_{\Spec(k)} \Spec(k')$ via l'action naturelle sur le
deuxi\`eme facteur.
\end{proof}

\subsection{D\'efinition et propri\'et\'es de la fonction $L$ motivique}

\begin{nota}\label{nota:repg}
Si $G$ est un groupe fini et $F$ un corps, on appellera
\termin{$F$-repr\'esentation de $G$} toute repr\'esentation lin\'eaire de dimension
finie de $G$ d\'efinie sur $F$. Si $\rho$ est une $F$-repr\'esentation on
note $V_{\rho}$ son espace de repr\'esentation et $\chi_{\rho}$ son caract\`ere.
\end{nota}

Soit $G$ un groupe fini et $\rho$ une $F$-repr\'esentation de $G$.
Les auteurs de \cite{DhMi:motivic_chebotarev} associent alors \`a 
tout objet $M$ de $\chmk{F}$
muni d'une action de $G$ 
une \termin{fonction $L$ d'Artin motivique} 
\begin{equation}
\LmDM(M,G,\rho,t)\eqdef \Zm((M\otimes V_{\rho})^G,t)\,\in\,1+\kochmk{F}[[t]]^+
\end{equation}
et \`a toute  $k$-$G$-vari\'et\'e projective et lisse $X$ la fonction 
\begin{equation}
\LmDM(X,G,\rho,t)\eqdef\LmDM(h(X),G^{\text{\text{op}}},\rho^{\text{\text{op}}},t)
\end{equation}
o\`u $\rhop$ est la repr\'esentation oppos\'ee de $\rho$.
Notons que pour $d\in \bZ$, on a un isomorphisme $(M(-d) \otimes
V)^G\isom (M\otimes V)^G(-d)$, d'o\`u, d'apr\`es \eqref{eq:zmd},
\begin{equation}\label{eq:lmdmd}
\LmDM(M(-d),G,\rho,t)=\LmDM(M,G,\rho,\bL^d\,t).
\end{equation}
Si $k$ est de caract\'eristique z\'ero, 
la proposition 2.7 de \cite{DhMi:motivic_chebotarev} et le lemme
7.1 de \cite{Bit:univ_eul_car}
montrent 
qu'il existe un unique morphisme de groupes
\begin{equation}
\LmDM(\,.\,,G,\rho,t)\,:\,\kogvark\to 1+\kochmk{F}[[t]]^+
\end{equation}
qui envoie la classe d'une $G$-vari\'et\'e projective et lisse $X$ sur $\LmDM(X,G,\rho,t)$.
\begin{rem}
Les auteurs de \cite{DhMi:motivic_chebotarev} supposent dans tout leur
article que le corps $F$ des coefficients des motifs contient toutes
les racines de l'unit\'e, 
mais cette hypoth\`ese est inutile pour la d\'efinition de $\LmDM$ et les r\'esultats de
\cite{DhMi:motivic_chebotarev} utilis\'es dans la pr\'esente section.
Ils n'utilisent cette hypoth\`ese qu'\`a partir
de la section 5 de leur article.
\end{rem}
\begin{rem}\label{rem:compat:scal}
D'apr\`es \cite[\S 4.2.2.]{And:mot}, si $E\to F$ est une extension,
il existe un morphisme d'anneaux naturel $\kochmk{E}\to \kochmk{F}$
et la formation  de $\LmDM$ est compatible \`a ce changement de
coefficients, i.e. si $\rho$ est une $F$-repr\'esentation, 
l'image de $\LmDM(X,G,\rho,t)$ par ce morphisme co\"\i ncide avec
$\LmDM(X,G,\rho\otimes_E F,t)$.
\end{rem}
\begin{lemme}\label{lm:lmdmtriv}
Si $k$ est de caract\'eristique z\'ero et si $\rho=\triv$ est la
repr\'esentation triviale, on a pour toute $G$-$k$-vari\'et\'e quasi-projective $X$
\begin{equation}
\LmDM(X,G,\triv,t)=\Zm(X/G,t).
\end{equation}
\end{lemme}
\begin{proof}
Il suffit de le montrer pour $X$ projective et lisse. 
On a alors, par d\'efinition,
\begin{equation}
\LmDM(X,G,\triv,t)=\Zm(h(X)^G,t)
\end{equation}
D'apr\`es le th\'eor\`eme \ref{thm:dbna}, on a $\Zm(h(X)^G,t)=\Zm(X/G,t)$.
\end{proof}
\begin{prop}\label{prop:LmDMrat}
Si $M$ est pair (respectivement impair), $(M\otimes
V)^G$ est pair (respectivement impair).
Si $M$ est de dimension finie, $(M\otimes V)^G$ est de dimension
finie~; 
en particulier $\LmDM(M,G,\rho,t)$ est rationnelle.
\end{prop}
\begin{proof}
On a des isomorphismes $\Sym^n(M\otimes V)\isom \Sym^n(M)\otimes \Sym^n(V)$ et
$\Alt^n(M\otimes V)\isom \Alt^n(M) \otimes \Alt^n(V)$. Ainsi si $M$ est
pair (repectivement impair) il en est de m\^eme pour
$M\otimes V$. En particulier si $M$ est de dimension finie, $M\otimes
V$ est de dimension finie. Or $(M\otimes V)^G$ est un facteur direct
de $M\otimes V$, et un facteur direct d'un objet de dimension finie
est de dimension finie.
\end{proof}
\begin{prop}\label{prop:lmdmdimzero}
Si $M$ est un motif d'Artin muni d'une action de $G$ et $\rho$ une
$F$-repr\'esentation de $G$, alors $(M\otimes V_{\rho})^G$ est encore un motif d'Artin.
En particulier $\LmDM(M,G,\rho,t)^{-1}$ est un \'el\'ement de $1+\koAMk{F}[t]^+$.
\end{prop}
\begin{proof}
La premi\`ere assertion est imm\'ediate. 
Compte tenu du fait que les motifs d'Artin sont pairs et de la formule \eqref{eq:symnaltn},
la deuxi\`eme en d\'ecoule.
\end{proof}
\begin{lemme}\label{lm:well:known}
Soit $G$ un groupe fini et $\rho$ une $F$-repr\'esentation de $G$.
On consid\`ere la structure de $G$-module sur $W\eqdef F[G]\otimes V_{\rho}$
donn\'ee par la r\'eguli\`ere gauche sur $F[G]$ et l'action triviale sur $V_{\rho}$.
Alors l'endomorphisme 
$\pi_{G,\rho}\eqdef \frac{1}{\card{G}}\sum_{g\in G}
\rho_d(g)\otimes \rho(g)$ 
($\rho_d$ d\'esignant la r\'eguli\`ere
droite) est un projecteur $G$-\'equivariant de $W$, d'image
$G$-isomorphe \`a $V_{\rho}$.
\end{lemme}
\begin{proof}
On v\'erifie que l'application
\begin{equation}
\map{V_{\rho}}{\im(\pi_{G,\rho})}{v}{\sumu{g\in G} g\otimes \rho(g)v}
\end{equation}
est un isomorphisme $G$-\'equivariant.
\end{proof}
La proposition suivante pr\'ecise 
les propositions 13.3.1.2 de \cite{And:mot} et  4.5 de \cite{DhMi:motivic_chebotarev}.
Remarquons que dans ces deux derniers \'enonc\'es, il est n\'ecessaire de supposer
la courbe g\'eom\'etriquement int\`egre.

\begin{prop}\label{prop:lmdmcourbe}
Soit $\courbe$ une $k$-courbe projective et  lisse.
\begin{enumerate}
\item\label{item:0:prop:lmdmcourbe}
$\Zm(\courbe,t)$ est rationnelle. Plus pr\'ecis\'ement, si $\courbe$ est irr\'eductible et $k'$ est le corps des constantes
de $\courbe$, la s\'erie formelle\footnote{La proposition
  \ref{prop:lmdmdimzero} montre que $\Zm(k',t)^{-1}$ 
est dans $1+\kochmk{F}[t]^+$.}
\begin{equation}
\Zm(k',t)^{-1}\,\Zm(k',\bL\,t)^{-1}\,\Zm(\courbe,t)
\end{equation} 
est un \'el\'ement de $1+\kochmk{F}[t]^+$. En particulier, si $\courbe$ est g\'eom\'etriquement int\`egre 
on a 
\begin{equation}
(1-t)\,(1-\bL\,t)\,\Zm(\courbe,t)\in 1+\kochmk{F}[t]^+
\end{equation} 
\suspend{enumerate}
On suppose \`a pr\'esent que $\courbe$ est irr\'eductible et qu'un goupe fini $G$ agit sur $\courbe$.
Rappelons (cf. proposition \ref{prop:act:g}) que $\Gop$ agit alors naturellement sur $h^0(\courbe)$ et $h^2(\courbe)$.
Soit $\rho$ une $F$-repr\'esentation de $G$.
\resume{enumerate}
\item\label{item:new:prop:lmdmcourbe}
$\LmDM(\courbe,G,\rho,t)$ est rationnelle. Plus pr\'ecis\'ement\footnote{
La proposition  \ref{prop:lmdmdimzero} montre que $\Zm((h^0(\courbe)\otimes V_{\rho{^{\text{op}}}})^{G^{\text{op}}},t)^{-1}$
est dans $1+\kochmk{F}[t]^+$.}, on a 
\begin{multline}
\Zm((h^0(\courbe)\otimes V_{\rhop})^{\Gop},t)^{-1}\,
\Zm((h^0(\courbe)\otimes V_{\rhop})^{\Gop},\bL\,t)^{-1}\,
\LmDM(\courbe,G,\rho,t)
\\
\in 1+\kochmk{F}[t]^+
\end{multline}
\item\label{item:2:prop:lmdmcourbe}
On suppose que l'action de $\Gop$ sur $h^0(\courbe)$ et $h^2(\courbe)$ est triviale.
\begin{enumerate}
\item\label{item:2a:prop:lmdmcourbe}
Pour toute $F$-repr\'esentation $\rho$ irr\'eductible non triviale, $\LmDM(\courbe,G,\rho,t)$ est un polyn\^ome.
\item\label{item:2b:prop:lmdmcourbe}
Pour toute $F$-repr\'esentation $\rho$, on a 
\begin{equation}
\left[\Zm(k',t)^{-1}\,\Zm(k',\bL\,t)^{-1}\right]^{\rg(V_{\rho}^G)}\,\LmDM(\courbe,G,\rho,t)\in 1+\kochmk{F}[t]^+.
\end{equation} 
En particulier, si $\courbe$ est g\'eom\'etriquement int\`egre on a
\begin{equation}
\left[(1-t)\,(1-\bL\,t)\right]^{\rg(V_{\rho}^G)}\,\LmDM(\courbe,G,\rho,t)\in 1+\kochmk{F}[t]^+.
\end{equation}
\end{enumerate}
\item\label{item:3:prop:lmdmcourbe}
On suppose que $\courbe=Y\times_k k'$,
o\`u $Y$ est g\'eom\'etriquement int\`egre,
$\Gop=\Gal(k'/k)$ et $G$ agit sur $Y\times_kk'$ via l'action naturelle
sur le deuxi\`eme facteur.
Pour toute $F$-repr\'esentation $\rho$ de $G$, on a alors
\begin{equation}
\left[
\Zm(V_{\rhop},t)^{-1}
\,
\Zm(V_{\rhop},\bL\,t)^{-1}
\right]\,
\LmDM(\courbe,G,\rho,t)\in 1+\kochmk{F}[t]^+.
\end{equation}
\end{enumerate}
\end{prop}
\begin{proof}
On reprend les notations \ref{notas:px}. Soit $h^1(\courbe)$ l'image du
projecteur $\Id-p_x-p'_x$. Alors $h^1(\courbe)$ est impair
(\cite[Theorem 4.2]{Kim:finite_dim}),
donc (proposition \ref{prop:zm:dimf}) $\Zm(h^1(\courbe),t)$ est dans $1+\kochmk{\bQ}[t]^+$.
La d\'ecomposition
\begin{equation}
h(\courbe)=h^0(\courbe)\oplus h^1(\courbe) \oplus h^2(\courbe)
\end{equation}
induit la d\'ecomposition
\begin{equation}
\Zm(\courbe,t)=\Zm(h^0(\courbe),t)\,\Zm(h^1(\courbe),t)\,\Zm(h^2(\courbe),t).
\end{equation}
Des isomorphismes $h^0(\courbe)\isom h(\Spec(k'))$ et 
$h^2(\courbe)\isom h(\Spec(k'))(-1)$ on d\'eduit le point \ref{item:0:prop:lmdmcourbe}.

Comme $p_x$ et $p'_x$ sont compatibles \`a l'action de $\Gop$
(proposition \ref{prop:act:g}), $\Id-p_x-p'_x$ l'est \'egalement  et
on a  (lemme \ref{lm:MG})
\begin{equation}
\left(h(\courbe)\otimes V_{\rho}\right)^{\Gop }
=
\left(h^0(\courbe)\otimes V_{\rho}\right)^{\Gop } 
\oplus
\left(h^1(\courbe)\otimes V_{\rho}\right)^{\Gop } 
\oplus 
\left(h^0(\courbe)\otimes V_{\rho}\right)^{\Gop }
\end{equation}
d'o\`u une d\'ecomposition
\begin{multline}
\LmDM(\courbe,G,\rho,t)
\\
=
\Zm((h^0(\courbe)\otimes V_{\rho})^{\Gop },t)\,
\Zm((h^1(\courbe)\otimes V_{\rho})^{\Gop },t)
\Zm((h^2(\courbe)\otimes V_{\rho})^{\Gop },t).
\end{multline}
Comme $h^1(\courbe)$ est impair, $(h^1(\courbe)\otimes V_{\rho})^{G^{\text{op}}}$
l'est \'egalement (proposition \ref{prop:LmDMrat}) et 
$\Zm((h^1(\courbe)\otimes V_{\rho})^{G^{\text{op}}},t)$ 
est un \'el\'ement de $1+\kochmk{F}[t]^+$. 
Par ailleurs, on a
\begin{align}
\Zm((h^2(\courbe)\otimes V_{\rho})^{\Gop },t)
&
=
\Zm((h^0(\courbe)(-1)\otimes V_{\rho})^{\Gop },t)
\\
&
=
\Zm((h^0(\courbe)\otimes V_{\rho})^{\Gop }(-1),t)
\\
&
=
\Zm((h^0(\courbe)\otimes V_{\rho})^{\Gop },\bL\,t).
\end{align}
Ceci montre le point \ref{item:new:prop:lmdmcourbe}.

Pour montrer le point \ref{item:2:prop:lmdmcourbe}, on remarque que si l'action
de $\Gop$ sur $h^0(\courbe)\isom h(\Spec(k'))$ est triviale, et si $\rho$ est
irr\'eductible, on a $\left(h^0(\courbe)\otimes V_{\rhop}\right)^{\Gop}=h^0(\courbe)\otimes V_{\rhop}^{\Gop}=0$
Si $\rho$ est quelconque, on obtient donc en
d\'ecomposant $V_{\rho}$ en somme de $G$-repr\'esentations irr\'eductibles un
isomorphisme de $F$-$\galabs{k}$-repr\'esentations discr\`etes 
\begin{equation}
\left(h^0(\courbe)\otimes V_{\rhop}\right)^{\Gop}
\isom
h^0(\courbe)^{\dim(V_{\rho}^{G})}
\isom 
h(\Spec(k'))^{\dim(V_{\rho}^{G})}
\end{equation}
et on applique le point \ref{item:new:prop:lmdmcourbe}.

Montrons a pr\'esent le point \ref{item:3:prop:lmdmcourbe}.
Via l'\'equivalence de cat\'egories \eqref{eq:cat:artin}, le motif d'Artin $h(\Spec(k'))$ s'identifie \`a la repr\'esentation
r\'eguli\`ere gauche de $\Gal(k'/k)$. D'apr\`es la proposition \ref{prop:act:g}, 
l'action de $\Gal(k'/k)$ sur $h^0(\courbe)\isom h(\Spec(k'))\isom F[\Gal(k'/k)]$ induite par l'action de $\Gal(k'/k)^{\text{op}}$ sur
$Y\times_k k'$ est la repr\'esentation r\'eguli\`ere droite.
On applique alors le point \ref{item:new:prop:lmdmcourbe} et le lemme \ref{lm:well:known}.
\end{proof}

\section{La fonction $L$ d'Artin motivique d\'efinie comme produit eulerien motivique}\label{sec:L:prodeulmot}
Dans cette section, nous allons donner, pour un corps $k$ de
caract\'eristique z\'ero, 
une autre d\'efinition de la fonction $L$ d'Artin
motivique attach\'ee \`a une $k$-$G$-vari\'et\'e  quasi-projective et une
$\bQ$-repr\'esentation  de $G$,
sous forme d'un \og produit eul\'erien motivique\fg, et 
montrer que la fonction obtenue
co\"\i ncide avec celle de Dhillon et Minac. 

\subsection{Motif virtuel associ\'e \`a une formule}

Concernant les rappels que contient cette sous-section, on renvoie \`a \cite{DeLo:def_sets_motives},
\cite{DeLo:grot_pff} ou \cite{Nic:rel:motive} pour plus de d\'etails. 
Dans toute la suite, on appelle \termin{$k$-formule} une formule logique
du premier ordre dans le langage des anneaux \`a coefficients dans un
corps $k$.
Pour toute formule $\varphi$  en $n$ variables libres et 
toute extension $K$ de $k$ on notera $\varphi(K)$ le sous-ensemble
de $K^n$ constitu\'e des \'el\'ements de $K^n$ satisfaisant $\varphi$. 
Si $X$ est une $k$-vari\'et\'e quasi-affine, on appellera \termin{formule sur $X$}
toute formule  en $n$ variables libres de la forme $\varphi \wedge \varphi_X$
o\`u $\varphi$ est une formule en $n$ variables libres et $\varphi_X$ la formule d\'efinissant les \'equations
d'un plongement de $X$ dans l'espace affine $\bA^n$.

Un \termin{corps pseudo-fini} est un corps parfait $K$ 
v\'erifiant les propri\'et\'es suivantes : 
\begin{enumerate}
\item
toute vari\'et\'e g\'eom\'etriquement irr\'eductible sur $K$ a un point rationnel dans $K$ ;
\item
une cl\^oture alg\'ebrique de $K$ \'etant fix\'ee,  pour tout $n\geq 1$,
$K$ admet  une unique extension de degr\'e $n$ dans cette cl\^oture alg\'ebrique.
\end{enumerate}
Tout corps admet une extension qui est un corps pseudo-fini.

Soient $\varphi$ et $\psi$ des formules \`a coefficients dans $k$ 
en les variables libres $(x_1,\dots,x_m)$
et $(y_1,\dots,y_n)$, respectivement.
On dit que $\varphi$ est un \termin{$d$-rev\^etement de $\psi$} 
s'il existe une formule 
$\theta$ en les variables libres $(x_1,\dots,x_m,y_1,\dots,y_n)$ tel
que pour tout corps pseudo-fini $K$ contenant $k$, l'ensemble
$\theta(K)\subset K^n\times K^m$
est le graphe d'une application $d$ pour $1$ de $\varphi(K)$
vers $\psi(K)$. Les formules $\varphi$ et $\psi$ sont dites
\termin{logiquement \'equivalentes} si $\varphi$ est un $1$-rev\^etement de $\psi$.

L'anneau de Grothendieck de la th\'eorie des corps pseudo-finis sur $k$, 
not\'e $\kpff$,
est  engendr\'e comme groupe par les symboles $[\varphi]$, o\`u $\varphi$ est une $k$-formule, et les relations $[\varphi]=[\psi]$ si $\varphi$
et $\psi$ sont logiquement \'equivalentes, ainsi que 
$
[\varphi\vee\psi]+[\varphi\wedge \psi]
=
[\varphi]+[\psi]
$
si $\varphi$ et $\psi$ ont les m\^emes variables libres.
Le produit est d\'efini par
$
[\varphi].[\psi]\eqdef\symb{\varphi\wedge\psi}.
$

Si $k$ est de caract\'eristique z\'ero, 
Denef et Loeser ont montr\'e dans \cite{DeLo:def_sets_motives} et 
 \cite{DeLo:grot_pff} comment associer de mani\`ere canonique \`a toute
 $k$-formule un motif de Chow virtuel
\og avec d\'enominateur \fg, 
i.e. un \'el\'ement de $\kochmkq{\bQ}$.
Dans \cite{Nic:rel:motive}
cette construction est \'etendue \`a un cadre relatif.
Pour d\'emontrer ces r\'esultats, une utilisation cruciale est faite de la th\'eorie de l'\'elimination des
quantificateurs dans les corps pseudo-finis en termes de formules
galoisiennes,  
d\^ue \`a Fried, Jarden et Sacerdote.
Rappelons ce qu'est une formule galoisienne : soit $X$ une vari\'et\'e
affine normale munie d'une action libre d'un groupe fini $G$.
Si $\cC$ est une classe de conjugaison de sous-groupes cycliques de
$G$, on note $\frev{X,G,\cC}$ une formule sur $X/G$
telle que, pour toute $k$-extension  pseudo-finie $K$, 
$\frev{X,G,\cC}(K)$
s'identifie \`a l'ensemble des \'el\'ements de $(X/G)(K)$ qui admette
les \'el\'ements de $\cC$ comme  groupes de d\'ecomposition dans $X\to X/G$.
Si $C$ est un \'el\'ement de $\cC$, ce dernier ensemble co\"\i ncide avec
l'ensemble des \'el\'ements de $(X/G)(K)$ qui se rel\`event \`a un \'el\'ement
de $(X/C)(K)$ mais pas \`a un \'el\'ement de $(X/D)(K)$ pour tout sous-groupe strict $D$ de $C$,
ce qui montre l'existence d'une telle formule d'anneau.

On a alors le r\'esultat suivant (cf. \cite[Theorem 2.1]{DeLo:grot_pff}
et \cite[Lemma 8.5]{Nic:rel:motive}).
\begin{thm}[Denef-Loeser,Nicaise]
\label{thm:denefloeser}
Soit $k$ un corps de caract\'eristique z\'ero.
Il existe un unique morphisme d'anneaux
\begin{equation}
\chif\,:\,\kpff\longto \kovchmkq{\bQ}
\end{equation}
qui envoie la classe d'une formule qui est une conjonction d'\'equations polyn\^omiale
sur la classe de la vari\'et\'e affine d\'efinie par ces \'equations 
et qui satisfait, pour toutes formules $\varphi$ et $\psi$ telle que $\varphi$
est un $d$-rev\^etement de $\psi$,
\begin{equation}
\chif\left(\varphi\right)=d\,\chif\left(\psi\right).
\end{equation}

Ce morphisme poss\`ede en outre les propri\'et\'es suivantes :
\begin{enumerate}
\item
si $X$ est une $k$-vari\'et\'e affine normale munie d'une action libre
de $G$ et $\cC$ est une classe de
conjugaison de sous-groupes cycliques de $G$, on a (cf. th\'eor\`eme \ref{thm:chiveq})
\begin{equation}\label{eq:chif:frev:chieq}
\chif\left(\frev{X,G,\cC}\right)=\chieq(X,\theta_\cC)
\end{equation}
o\`u $\theta_{\cC}$ est la fonction
\begin{equation}
\theta_{\cC}\,:\,g\mapsto
\left\{
\begin{array}{cl}
1&\text{si le groupe engendr\'e par $g$ est dans }\cC\\
0&\text{sinon.}
\end{array}
\right.
\end{equation}
\item\label{item:nbre:pts:mod:p}
Supposons que $k$ soit un corps de nombres. Soit $\ell$
un nombre premier. Soit $\varphi\in \kpff$. 
Alors pour presque tout id\'eal premier non nul
$\mfp$ de $k$ 
on a $\Trp(\chil(\chif(\varphi)))=\card{\varphi(\kappa_{\mfp})}$.
\end{enumerate}
\end{thm}
\begin{defi}\label{def:car:non:nulle:bis}
Soit $k$ un corps quelconque, $G$ un groupe fini
et $X$ une $k$-vari\'et\'e quasi-projective  munie d'une action libre
de $G$.
Si $X$ est projective et lisse, on d\'efinit $\chif(\frev{X,G,\cC})\in \kochmkq{\bQ}$ par la relation
\eqref{eq:chif:frev:chieq} (cf. la d\'efinition \ref{def:car:non:nulle}).
Dans le cas g\'en\'eral, on d\'efinit
$\symbv{\frev{X,G,\cC}}\in\kovark\otimes\bQ$
comme suit~: 
si $G$ est cyclique, on d\'efinit
$\symbv{\frev{X,G,G}}$
par r\'ecurrence sur $\card{G}$ en utilisant la relation 
\begin{equation}
\card{G}\,\symb{X/G}=\sum_{C<G} \card{C}\,\symbv{\frev{X,C,C}}\,;
\end{equation}
pour $G$ quelconque et $\cC$ une classe de conjugaison de sous-groupes
cycliques de $G$, on pose 
$
\symbv{\frev{X,G,\cC}}=\frac{\card{C}}{\card{N_G(C)}}\,\symbv{\frev{X,C,C}}
$
o\`u $C$ est un \'el\'ement de $\cC$.
\end{defi}
\begin{rem}\label{rem:car:non:nulle:bis}
Supposons $X$ projective et lisse.
Si $k$ est un corps global,
en raisonnant comme dans la preuve du lemme 3.3.2 de 
\cite{DeLo:def_sets_motives}, on voit 
que pour presque toute place
finie $\mfp$ de $k$ la quantit\'e
$\Trp(\chil(\chif(\frev{X,G,\cC})))$
est \'egale au cardinal de l'ensemble des \'el\'ements de $(X/G)(\kappa_{\mfp})$ ayant
d\'ecomposition $\cC$ dans le rev\^etement $X\to X/G$.

Supposons \`a pr\'esent $k$ fini. 
Soit $F_k\in \galabs{k}$ le Frobenius g\'eom\'etrique. Pour $r\geq 1$, on
note $k_r$ l'extension de degr\'e $r$ de $k$.
Toujours en raisonnant comme dans la preuve du lemme 3.3.2 de 
\cite{DeLo:def_sets_motives}, on voit que, pour tout entier $r\geq 1$,
$
\Tr(F_k^r|\chil(\chif(\frev{X,G,\cC})))
$
est \'egal au cardinal de l'ensemble des \'el\'ements de $(X/G)(k_r)$ ayant
d\'ecomposition $\cC$ dans le rev\^etement $X\to X/G$.

Par ailleurs, si $k$ est de caract\'eristique z\'ero et
$X$ est affine normale, le
th\'eor\`eme pr\'ec\'edent montre que, par construction, on a 
\begin{equation}\label{eq:chiv:frev}
\chiv(\symbv{\frev{X,G,\cC}})=\chif(\frev{X,G,\cC}).
\end{equation}
\end{rem}

\subsection{Le motif virtuel  des points ferm\'es de degr\'e $n$}

Nous rappelons dans cette sous-section la construction du paragraphe
2.5 de \cite{Bou:prod:eul:mot}.
Soit $k$ un corps de caract\'eristique z\'ero et $X$ une $k$-vari\'et\'e
quasi-projective.
Pour tout entier $n\geq 1$, on note $\Sym^n(X)_0$ l'ouvert
de $\Sym^n(X)$ constitu\'e des ensembles de $n$ points deux \`a deux distincts.
Le morphisme $\pi_n\,:\,X^n\to \Sym^n(X)$ induit donc un $\sym_n$-rev\^etement \'etale
\begin{equation}
\pi_n\,:\,
\pi_n^{-1}\left(\Sym^n(X)_0\right)
\longto
\Sym^n(X)_0
\end{equation}
Supposons \`a pr\'esent $X$ affine. Soit $\cC_n$ est la classe des sous-groupes de $\sym_n$ engendr\'e par
un $n$-cycle.
On note $\psinx{X}$ 
la formule galoisienne $\frev{\pi_n^{-1}\left(\Sym^n(X)_0\right),\sym_n,\cC_n}$.
\begin{rem}\label{rem:psink}
Pour toute $k$-extension pseudo-finie $K$, l'application qui a un
\'el\'ement 
de $\{x_1,\dots,x_n\}$ de $\psinx{X}(K)$ associe le z\'ero-cycle 
$K$-rationnel $\sum x_i$ induit donc une bijection de $\psinx{X}(K)$
sur l'ensemble des points ferm\'es de degr\'e $n$ de $X_K$.
\end{rem}
La classe de $\psinx{X}$ 
dans $\kpff$ ne d\'epend pas du choix du
plongement affine de $X$, on la note encore $\psinx{X}$.
\begin{rem}\label{rem:psin:dim0}
Si $L$ est une $k$-extension finie s\'eparable, on a
$\psinx{\Spec(L)}=0$
si $n>[L:k]$.
\end{rem}
Comme $\kovark$ est engendr\'e par les classes de vari\'et\'es affines et
que pour tout ouvert affine $U$ d'une vari\'et\'e affine $X$ on a la
relation $\psinx{X}=\psinx{U}+\psinx{X\setminus U}$ 
(cf. \cite[Lemme 3.7]{Bou:prod:eul:mot})
il existe un
unique morphisme de groupes 
\begin{equation}
\kovark\to\kpff
\end{equation}
qui envoie la classe d'une vari\'et\'e affine $X$ sur $\psinx{X}$.
On note encore $\psinx{\,.\,}$ le morphisme de groupes $\kovark\to
\kochmkq{\bQ}$ obtenu par composition avec le morphisme
$\chif$ du th\'eor\`eme \ref{thm:denefloeser}. Rappelons l'\'enonc\'e de la proposition 2.17 de \cite{Bou:prod:eul:mot}.
\begin{prop}
Soit $k$ un corps de caract\'eristique z\'ero.
Pour toute $k$-vari\'et\'e $X$, on a la relation
\begin{equation}\label{eq:relpsidzm}
\sum_{n\geq 1} \left(\sum_{d|n} d\,\psidx{X}\right)\frac{t^n}{n} 
=t\,\frac{d\log}{dt} \Zm(X,t)
\end{equation}
soit de mani\`ere \'equivalente
\begin{equation}\label{eq:relpsinphin}
\forall n\geq 1,\quad \Phi_n(X)=\sum_{d|n} d\,\psidx{X}.
\end{equation}
\end{prop}
\begin{defi}\label{def:eq:relpsinphin}
Si $k$ est de caract\'eristique non nulle et $X$ est une $k$-vari\'et\'e projective et
lisse, on d\'efinit $\psinx{X}\in \kochmk{\bQ}_{\bQ}$ par la relation \eqref{eq:relpsinphin}.
De m\^eme si $k$ est quelconque et $X$ est une
$k$-vari\'et\'e quasi-projective, on d\'efinit $\psinvar{X}\in
\kovark\otimes \bQ$ par la relation
\begin{equation}\label{eq:relpsinvphinv}
\forall n\geq 1,\quad \phinvar{X}=\sum_{d|n} d\,\psidvar{X}
\end{equation}
(de sorte que si $k$ est de caract\'eristique nulle, on a $\chiv(\psinvar{X})=\psinx{X}$).
\end{defi}
\begin{rem}\label{rm:eq:relpsinphin}
Une autre possibilit\'e pour d\'efinir $\psinvar{X}$ serait bien s\^ur de poser
(cf. d\'efinition \ref{def:car:non:nulle:bis})
$
\psinvar{X}\eqdef\symbv{\frev{\pi_n^{-1}\left(\Sym^n(X)_0\right),\sym_n,\cC_n}}
$.
D'apr\`es \eqref{eq:chiv:frev}, si $k$ est de
caract\'eristique $0$ on a encore $\chiv(\psinvar{X})=\psinx{X}$.
Mais il n'est pas clair que la relation \eqref{eq:relpsinvphinv} soit
alors  v\'erifi\'ee.
\end{rem}
\begin{cor}\label{cor:prop:spec:psin}
Supposons que $k$ soit un corps de nombres. 
Soit $X$ une $k$-vari\'et\'e.
Pour presque toute place finie $\mfp$ de $k$, on a 
la propri\'et\'e suivante :
pour tout $n\geq 1$, 
le nombre de points de $\Phi_n(X)$ (respectivement $\psinx{X}$) modulo $\mfp$
est \'egal au nombre de points de $X_{\mfp}$ \`a valeurs dans une
extension de degr\'e $n$ de $\kappa_{\mfp}$ 
(respectivement au nombre de points ferm\'es de degr\'e $n$ de $X_{\mfp}$).
\end{cor}
\begin{proof}
Ceci d\'ecoule du corollaire \ref{cor:delbano}, 
de la relation \eqref{eq:relpsidzm}
et des relations classiques liant fonction z\^eta de Hasse-Weil,
nombre de points ferm\'es de degr\'e $n$ et nombre de points \`a valeurs
dans une extension de degr\'e $n$.
\end{proof}
\subsection{Motif virtuel associ\'e \`a un symbole d'Artin}
Utilisant toujours la construction de Denef et Loeser,
nous d\'efinissons dans cette sous-section des motifs virtuels
associ\'e naturellement aux symboles d'Artin (cf. la
proposition \ref{prop:spec:phixgcin}).
Nous verrons \`a la section \ref{sec:symb:artin} qu'en un certain sens,
ils co\"\i ncident avec les motifs virtuels analogues d\'efinis par
Dhillon et Minac dans \cite{DhMi:motivic_chebotarev}.
Le lemme \'el\'ementaire suivant nous sera utile.
\begin{lemme}\label{lm:sansnom}
Soit $\pi\,:\,X\longto Y$ un morphisme de $k$-vari\'et\'es affines. 
Soit $K$ une $k$-extension  pseudo-finie de $k$, dont on fixe
une cl\^oture alg\'ebrique. 
Soit $n\geq 1$ un
entier, $\Sym^n \pi\,:\,\Sym^n X\to \Sym^n Y$ le morphisme induit par
$\pi$ et $K_n$ l'unique extension de degr\'e $n$ de $K$.
Soit $y\in \psi_n(K)\subset (\Sym^n Y)(K)$, soit $\wt{y}$ le point ferm\'e
de degr\'e $n$ de $X_K$ associ\'e, 
et soit $\clo{y}$ un \'el\'ement de  $Y(K_n)$
ayant pour image $\wt{y}$. Alors $y$ se rel\`eve via $\Sym^n \pi$ en un
point de $(\Sym^n X)(K)$ 
si et seulement si $\clo{y}$ se rel\`eve via $\pi$ en un point de $X(K_n)$.
\end{lemme}

\begin{nota}
Soit  $G$ un groupe fini. 
On note 
$\irrq{G}$ l'ensemble des classes d'isomorphismes de $\bQ$-repr\'esentations
irr\'eductibles de $G$ et $\bconjc{G}$ l'ensemble des classes de conjugaison de sous-groupes cycliques
de $G$.

Soit $\rho$ une $\bQ$-repr\'esentation de $G$, $\cC$ un \'el\'ement de $\bconjc{G}$ et $g$ un g\'en\'erateur d'un
\'el\'ement de $\cC$. La valeur de $\chi_{\rho}(g)$ ne d\'epend pas
du choix d'un tel $g$ : on la note $\chi_{\rho}(\cC)$.
\end{nota}

Soit $G$ un groupe fini, $X$ une $k$-$G$-vari\'et\'e affine sur $k$ et
$Y\eqdef X/G$.
Soit $\cI$ une classe de conjugaison de sous-groupes de $G$. On note $Y_\cI$
le sous-ensemble contructible de $Y$ form\'e des points ayant un groupe
d'inertie dans $\cI$.
Soit $\cD$ une classe de conjugaison de sous-groupes de $G$ v\'erifiant
la propri\'et\'e suivante~: il existe des
\'el\'ements $I$ de $\cI$ et $D$ de $\cD$ tels que $I\subset D \subset
N_G(I)$ et $D/I$ est cyclique. 

Soit $n\geq 1$ un entier. On note $\Sym^n(Y_\cI)_0$ le sous-ensemble
constructible de $\Sym^n(Y)$ form\'e de l'intersection de $\Sym(Y)^n_0$
avec l'image de $Y_\cI^n$ dans $\Sym^n(Y)$.
On note $\psin{X,G,\cI,\cD,n}$ une formule dont l'interpr\'etation dans
toute $k$-extension pseudo-finie
$K$ est l'ensemble des \'el\'ements de $\Sym^n(Y_\cI)_0(K)$, satisfaisant $\psinx{Y}$, 
qui se rel\`event en un point $K$-rationnel de $\Sym^n(X/D)$ mais pas \`a un point $K$-rationnel 
de $\Sym^n(X/D')$ pour tout sous-groupe strict $D'$ de $D$. 
\begin{rem}\label{rem:phixgidn}
D'apr\`es le lemme \ref{lm:sansnom}, la bijection naturelle entre $\psi_n(K)$
et les points ferm\'es de degr\'e $n$ de $Y_K$
(cf. la remarque \ref{rem:psink}) met en correspondance
les \'el\'ements de $\psin{X,G,\cI,\cD,n}(K)$ et les points ferm\'es de degr\'e $n$ de $Y_K$
ayant dans le $G$-rev\^etement $X_K\to Y_K$ un groupe d'inertie dans $\cI$
et un groupe de d\'ecomposition dans  $\cD$.
\end{rem}
On suppose \`a pr\'esent $k$ de caract\'eristique z\'ero.
L'image de $\psin{X,G,\cI,\cD,n}$ par le morphisme $\chif$ du th\'eor\`eme
\ref{thm:denefloeser} ne d\'epend pas du choix du plongement affine de $X$, on la note
encore $\psin{X,G,\cI,\cD,n}$. On v\'erifie que si $U$ est un ouvert
affine $G$-stable de $X$, on a la relation 
\begin{equation}
\psin{X,G,\cI,\cD,n}=\psin{U,G,\cI,\cD,n}+\psin{X\setminus U,G,\cI,\cD,n}
\end{equation}
ce qui permet de d\'efinir, \`a
$G$, $\cI$, $\cC$ et $n$ fix\'es, un
morphisme de groupes
\begin{equation}
\psin{\,.\,,G,\cI,\cD,n}\,:\,\kogvark\longto \kovchmkq{\bQ}.
\end{equation}
\begin{rem}\label{rem:ex:dim0}
Soit $G$ un groupe fini, $L$ une extension finie de $k$ 
et $\pi\,:\,G\to \Aut_k(\Spec(L))$ un morphisme.
Pour tout $n\geq 1$, on a $\psin{\Spec(L),G,\cI,\cD,n}=0$ si 
$\cI\neq
\{\Ker(\pi)\}.
$
Par ailleurs, d'apr\`es la remarque \ref{rem:psin:dim0}, on a 
$\psin{\Spec(L),G,\{\ker(\pi)\},\cD,n}=0$ 
si $n>[L^G:k]$.
\end{rem}
\begin{prop}\label{prop:spec:phixgcin}
Supposons que $k$ soit un corps de nombres. 
Soit $n\geq 1$. 
Pour presque toute place finie 
$\mfp$, 
le nombre de points modulo $\mfp$ de $\psin{X,G,\cI,\cD,n}$
est \'egal au nombre de points ferm\'es de degr\'e $n$ de $(X/G)_{\mfp}$
ayant
un groupe d'inertie dans $\cI$
et un groupe de d\'ecomposition dans $\cD$.
\end{prop}
\begin{proof}
On peut supposer $X$ affine. Pour presque tout $\mfp$, la $G$-vari\'et\'e
$X$ 
a bonne r\'eduction en $\mfp$.
D'apr\`es le lemme \ref{lm:sansnom} et le point 2. du th\'eor\`eme \ref{thm:denefloeser}
il y a alors une bijection entre les \'el\'ements de
$\psin{X,G,\cI,\cD,n}(\kappa_{\mfp})$ 
et les points ferm\'es de degr\'e $n$ de $Y_{\kappa_{\mfp}}$
ayant dans le $G$-rev\^etement $X_{\kappa_{\mfp}}\to Y_{\kappa_{\mfp}}$ un groupe d'inertie dans $\cI$
et un groupe de d\'ecomposition dans $\cD$. 
On conclut gr\^ace au point \ref{item:nbre:pts:mod:p} du th\'eor\`eme
\ref{thm:denefloeser}.
\end{proof}
\begin{rem}\label{rem:prop:spec:phixgcin}
Nous ignorons s'il est possible de trouver un ensemble fini de places
$S$ de $k$ tel que pour tout $n\geq 1$ et tout $\mfp\notin S$ le nombre de points modulo $\mfp$ de
$\psin{X,G,\cI,\cD,n}$
est \'egal au nombre de points ferm\'es de degr\'e $n$ de $(X/G)_{\mfp}$
ayant
un groupe d'inertie dans $\cI$
et un groupe de d\'ecomposition dans $\cD$
(i.e. s'il existe une
version uniforme en $n$ de la proposition \ref{prop:spec:phixgcin}).
C'est en tout cas vrai si $X$ est de dimension z\'ero (d'apr\`es la
remarque \ref{rem:ex:dim0}) ou de dimension un (cf. la
remarque \ref{rem:nonram:dimun} ci-dessous).
\end{rem}
\begin{nota}\label{nota:bconjcgI}
Si $G$ est un groupe fini, 
on note $\bconj{G}$ l'ensemble des classes de
conjugaison de sous-groupes de $G$. 
Si $\cI$ est un \'el\'ement de $\bconj{G}$,
on note $\bconjc{G,\cI}$ l'ensemble des classes de
conjugaison $\cC$ de $G$ v\'erifiant la propri\'et\'e suivante : il existe des
\'el\'ements $I$ de $\cI$ et $C$ de $\cC$ tels que $I\subset C \subset
N_G(I)$ et $C/I$ est cyclique.
\end{nota}
\begin{rem}
Soit $X$ une $G$-vari\'et\'e affine.
Pour $n\geq 1$ donn\'e, 
d'apr\`es les remarques \ref{rem:psink}
et \ref{rem:phixgidn}
il est imm\'ediat que les
formules 
\begin{equation}
\left(\psin{X,G,\cI,\cD,n}\right)_{\cI\in \bconj{G},\cD\in \bconjc{G,\cI}}
\end{equation}
forment une partition de la formule $\psi_n(X/G)$, i.e. on a pour
toute $k$-extension pseudo-finie $K$
\begin{equation}
\psi_n(X/G)(K)=\disju{
\substack{
\cI\in \bconj{G},\\
\cD\in \bconjc{G,\cI}}
}
\psin{X,G,\cI,\cD,n}(K).
\end{equation}
En particulier, pour toute $G$-vari\'et\'e quasi-projective $X$, 
on a dans 
$\kpff$
la relation
\begin{equation}\label{eq:rel:psin:phixgicn}
\psinx{X/G}
=
\sum_{
\substack{
\cI\in \bconj{G},\\
\cD\in \bconjc{G,\cI}}
}
\psin{X,G,\cI,\cD,n}.
\end{equation}
\end{rem}
Nous introduisons \`a pr\'esent une construction qui nous sera utile pour la d\'emonstration de la
compatibilit\'e \`a l'induction de la fonction $L$ d'Artin motivique
(cf. lemme \ref{lm:compat:ind}).
\begin{nota}\label{nota:fICJD}
Soit $H$ un sous-groupe d'un groupe fini $G$.
Soit $\cJ\in \bconj{H}$, $\cD\in \bconjc{H,\cJ}$,
$\cI\in \bconj{G}$ et $\cC\in \bconjc{G,\cI}$.
On \'ecrit $(\cJ,\cD)\subset (\cI,\cC)\cap H$ si on a la propri\'et\'e : il
existe un \'el\'ement $(I,C)$ de $\cI\times \cC$ tel que $I$ est distingu\'e
dans $C$, $C/I$ est cyclique, $C\cap H\in \cD$ et $I\cap H \in \cJ$.
On note alors
\begin{equation}
\fICJD\eqdef \frac{\card{C}\,\card{I\cap H}}{\card{C\cap H}\card{I}}
\end{equation}
et 
\begin{equation}
\dICJD
\eqdef
\card{\{g\in C\backslash G/H,\quad g\,C\,g^{-1}\cap H\in \cD\text{ et }g\,I\,g^{-1}\cap H\in \cJ\}}.
\end{equation}
\end{nota}
\begin{rem}
Soit $X$ une $k$-$G$-vari\'et\'e quasi-projective.
Soit $K$ une $k$-extension pseudo-finie et $y$ un point ferm\'e de
$(X/H)_K$, 
ayant
d\'ecomposition $\cD$ et inertie $\cJ$. On note $\cC$ (respectivement
$\cI$) les groupes de d\'ecomposition (respectivement d'inertie) de
l'image $x$ de $y$ dans $(X/G)_K$. Alors le degr\'e de l'extension de corps
r\'esiduel $\kappa_x\to \kappa_y$ est $\fICJD$, et il y a exactement
$\dICJD$ points ferm\'es de $(X/H)_K$ ayant d\'ecomposition $\cD$, inertie
$\cJ$, et image $x$.
\end{rem}
Soit $n\geq 1$ et $d$ un diviseur de $n$. Pour toute $k$-vari\'et\'e
quasi-projective $X$, on note $\pi_{(d,n)}\,:\,\Sym^d X \to \Sym^n X$ le morphisme qui envoie $\{x_1,\dots,x_d\}$
sur $\{\underbrace{x_1,\dots,x_1}_{\frac{n}{d}\text{ fois}},x_2,\dots,x_2,\dots,x_d,\dots,x_d\}$.
Consid\'erons \`a pr\'esent une $k$-$G$-vari\'et\'e affine $X$ et $H$ un
sous-groupe de $G$. Soit $p$ le morphisme naturel $X/H\to X/G$.
Soit $\cJ\in \bconj{H}$, $\cD\in \bconjc{H,\cJ}$,
$\cI\in \bconj{G}$ et $\cC\in \bconjc{G,\cI}$
tels que $(\cJ,\cD)\subset (\cI,\cC)\cap H$.
Pour $n\geq 1$,
on note $\psi^{\cI,\cC}_{_{X,H,\cJ,\cD,n}}$ une formule sur $\Sym^n(X/H)$
dont l'interpr\'etation dans toute $k$-extension pseudo-finie $K$ est
l'ensemble des \'el\'ements de $\Sym^n(X/H)(K)$  
\begin{enumerate}
\item
qui sont dans $\psin{X,H,\cJ,\cD,n}(K)$ ;
\item
dont l'image par $\Sym^n p$ dans $\Sym^n(X/G)(K)$
est l'image par 
$\pi_{n/\fICJD,n}$ 
d'un \'el\'ement de 
$\Sym^{n/\fICJD}(X/G)(K)$
qui est dans 
$\psin{X,G,\cI,\cC,n/\fICJD}(K)$.
\end{enumerate}
L'identification canonique de $\psin{X,H,\cJ,\cD,n}(K)$ 
\`a $[(X/H)_K]^{(0)}_{_{H,\cJ,\cD,n}}$ met donc en bijection 
$\psi^{\cI,\cC}_{_{X,H,\{e\},\cD,n}}(K)$ avec l'ensemble des \'el\'ements 
de $[(X/H)_K]^{(0)}_{_{H,\cJ,\cD,n}}$  dont l'image par $p$ est dans  
$[(X/G)_K]^{(0)}_{_{H,\cI,\cC,d}}$.
En particulier, les ensembles $\psi^{\cI,\cC}_{_{X,H,\cJ,\cD,n}}(K)$, o\`u
$(\cC,\cI)$ d\'ecrit l'ensemble des \'el\'ements de $\bconj{G}^2$ tels que
$\cC\in \bconjc{G,\cI}$ et $(\cJ,\cD)\subset (\cI,\cC)\cap H$, forment une partition de
$\psin{X,H,\cJ,\cD,n}(K)$.
On en d\'eduit \'egalement que $p$ induit une application $\dICJD$ pour $1$
de $\psin{X,H,\cJ,\cD,n}(K)$ sur son image par $\Sym^np$, laquelle est en
bijection avec $\psin{X,G,\cI,\cC,\frac{n}{\fICJD}}(K)$.

Supposons \`a pr\'esent $k$ de caract\'eristique z\'ero.
L'image de $\psi^{\cI,\cC}_{_{X,H,\cJ,\cD,n}}$ par le morphisme $\chif$ du th\'eor\`eme
\ref{thm:denefloeser} ne d\'epend pas du choix du plongement affine de $X$, on la note
encore $\psi^{\cI,\cC}_{_{X,H,\cJ,\cD,n}}$. En outre
$\psi^{\cI,\cC}_{_{X,H,\cJ,\cD,n}}$ 
est compatible \`a la
d\'ecomposition d'une $G$-vari\'et\'e affine en une sous-$G$-vari\'et\'e affine
ouverte et son compl\'ementaire. 
On d\'efinit ainsi, \`a 
$G$, $H$, $\cJ$, $\cD$, $\cI$, $\cC$ et $n$ fix\'es, un
morphisme de groupes
\begin{equation}
\psi^{\cI,\cC}_{_{\,.\,,H,\cJ,\cD,n}}\,:\,\kogvark\longto \kovchmk{\bQ}_{\bQ}
\end{equation}
v\'erfiant, d'apr\`es ce qui pr\'ec\`ede, les deux relations suivantes.
\begin{lemme}\label{lm:psi:eq:sum:psi:c}
\begin{equation}
\psin{X,H,\cJ,\cD,n}
=
\sum_{
\substack{
(\cI,\cC)\in \bconj{G}^2 
\\
\cC\in \bconjc{G,I}
\\
(\cJ,\cD)\subset (\cI,\cC)\cap H
}
}
\psi^{\cI,\cC}_{_{X,H,\cJ,\cD,n}}
\end{equation}
\end{lemme}
\begin{lemme}\label{lm:psiC} 
\begin{equation}
\psi^{\cI,\cC}_{_{X,H,\cJ,\cD,n}}=\dICJD\,\,\psin{X,G,\cI,\cC,n/\fICJD}.
\end{equation}
\end{lemme}

\subsection{D\'efinition via  le produit eul\'erien motivique}

\begin{nota}
Soit $G$ un groupe fini et $\rho$ une $\bQ$-repr\'esentation de $G$.
Soit $\cI\in \bconj{G}$ et $\cD\in \bconjc{G,\cI}$.
Soient $I$ et $D$ des \'el\'ements de $\cI$ et $\cD$ respectivement tels
qu'on ait $I\subset D \subset N_G(I)$ et $D/I$ cyclique. 
Soit $g$ un \'el\'ement de $D$ dont l'image engendre $D/I$.
L'\'el\'ement de $\bQ[t]$
$
\det(\Id-t\,\rho(g)\,| V^I)\in \bQ[t]
$
ne d\'epend pas des choix de $I$, $D$ et (comme $\rho$ est d\'efinie sur $\bQ$) $g$~;
on le note
$
\polcar_{\rho,\cI,\cD}(t).
$
\end{nota}

Rappelons la d\'efinition de la fonction $L$ d'Artin
classique comme produit eulerien. C'est sur cette d\'efinition qu'est
model\'ee notre d\'efinition 
de la fonction $L$ d'Artin
motivique comme produit eulerien motivique.

Soit $k$ un corps fini, $X$ une $k$-$G$-vari\'et\'e quasi-projective, 
$E$ un corps et $\rho$ une $E$-repr\'esentation de $G$.
Soit $(X/G)^{(0)}$ l'ensemble des points ferm\'es de $X/G$. Pour $y\in
(X/G)^{(0)}$,  soit $D_y$ un groupe de d\'ecomposition au-dessus de $y$,
$I_y$ le groupe d'inertie correspondant, et $\Fr_y$ une
pr\'eimage dans $C_y$ du Frobenius correspondant.
Le facteur local de la fonction
$L$
d'Artin en $y$ s'\'ecrit alors
\begin{equation}\label{eq:ly}
L_y(X,G,\rho,t)\eqdef\det(\Id-t^{\deg(y)}\,\rho(\Fr_y)\,|\,V^{I_y})^{-1}
\end{equation}
La fonction $L$ d'Artin associ\'ee aux donn\'ees pr\'ec\'edentes 
est l'\'el\'ement de $1+E[[t]]^+$ d\'efini par 
\begin{equation}
\Lar(X,G,\rho,t)\eqdef \prod_{y\in (X/G)^{(0)}}L_y(X,G,\rho,t).
\end{equation}

Supposons \`a pr\'esent que $E=\bQ$.
Pour $y\in (X/G)^{(0)}$,
on note $\cI_y$ (respectivement $\cD_y$) l'ensemble des groupes
d'inertie (respectivement de d\'ecomposition)
au-dessus de $y$. La d\'efinition \eqref{eq:ly} se r\'e\'ecrit alors
\begin{equation}
L_y(X,G,\rho,t)=\polcar_{\rho,\cI_y,\cD_y}(t^{\deg(y)})^{-1}.
\end{equation}
\begin{nota}
Soit $X$ une $k$-$G$-vari\'et\'e quasi-projective. Pour tout $n\geq 1$, tout $\cI\in \bconj{G}$ et $\cD\in
\bconjc{G,\cI}$, on note $X^{(0)}_{_{G,\cI,\cD,n}}$ l'ensemble des points 
ferm\'es de $X/G$ de degr\'e $n$ ayant inertie $\cI$ et d\'ecomposition $\cD$.
\end{nota}
On a donc 
la relation
\begin{equation}\label{eq:lm:ecr:L:class}
\Lar(X,G,\rho,t)
=
\prod_{n\geq 1} \,
\prod_{
\substack{
\cI\in \bconj{G} \,
\\
\cD\in \bconjc{G,\cI}
}
}
\polcar_{\rho,\cI,\cD}(t^n)^{-\card{X^{(0)}_{_{G,\cI,\cD,n}}}}.
\end{equation}

\begin{nota}
Si $A$ est une $\bQ$-alg\`ebre, $P$ un \'el\'ement de $1+A[[t]]^+$ et $a$
un \'el\'ement de $A$
on pose
\begin{align}
P(t)^a&\eqdef 1+\sum_{n\geq 1}
\frac{a.(a-1)\dots(a-n+1)}{n!} (P(t)-1)^n
\\
&
=
\exp\left(a\,\log[P(t)]\right).
\end{align}
\end{nota}

\begin{defi}
Soit $k$ un corps de caract\'eristique z\'ero, $G$ un groupe
fini, $\rho$ une $\bQ$-repr\'esentation  de $G$, 
$X$ une $k$-$G$-vari\'et\'e quasi-projective et $Y=X/G$.
Motiv\'e par la relation \eqref{eq:lm:ecr:L:class}, 
on d\'efinit la \termin{fonction $L$ d'Artin motivique} associ\'ee \`a ces donn\'ees
comme \'etant l'\'el\'ement de $1+\kovchmkq{\bQ}[[t]]^+$
d\'efini par
\begin{equation}\label{eq:def:lm}
\Lm(X,G,\rho,t)
\eqdef 
\prod_{n\geq 1}\,\,
\prod_{\cI\in \bconj{G}}\, \,
\prod_{\cD\in \bconjc{G,\cI}}\,\,
\polcar_{\rho,\cI,\cD}(t^n)^{-\psin{X,G,\cI,\cD,n}}
\end{equation}
\end{defi}
\subsection{Propri\'et\'es}
Dans toute cette partie, le corps $k$ est suppos\'e de caract\'eristique z\'ero.
\subsubsection{Lien avec la fonction z\^eta motivique}
Dans le cas d'un corps de base fini, si $\rho$ est triviale,
$\Lar(X,G,\rho,t)$ est la fonction z\^eta de Hasse-Weil de $X/G$.
On a un r\'esultat similaire pour la fonction $L$ motivique.
\begin{lemme}\label{lm:ltriv:zm}
Soit $X$ une $k$-$G$-vari\'et\'e quasi-projective.
On a 
\begin{equation}
\Lm(X,G,\triv,t)=\Zm(X/G,t)
\end{equation} 
\end{lemme}
\begin{proof}
D'apr\`es la relation \eqref{eq:rel:psin:phixgicn}
et la d\'efinition de $\Lm(X,G,\triv,t)$, on a
\begin{equation}\label{eq:lmtriv}
\Lm(X,G,\triv,t)=\prod_{n\geq 1} \left(1-t^n\right)^{\,-\psi_n(X/G)}
\end{equation}
D'apr\`es la proposition 2.17 de \cite{Bou:prod:eul:mot}, le membre de droite de \eqref{eq:lmtriv}
coïncide avec $\Zm(X/G,t)$.
\end{proof}
\subsubsection{Somme directe, quotient, induction et restriction}
Nous \'enon\c cons et d\'emontrons \`a pr\'esent quelques propri\'et\'es \'el\'ementaires
des fonctions $L$ d'Artin motiviques, pendant naturel de propri\'et\'es
des fonctions $L$ d'Artin classiques : compatibilit\'e \`a la
somme directe, au quotient, \`a l'induction et \`a la restriction des
repr\'esentations. La compatibilit\'e \`a la somme directe et \`a l'induction nous  permettra de
montrer que notre fonction co\"\i ncide avec la fonction d\'efinie par Dhillon
et Minac (qui v\'erifie les propri\'et\'es analogues), et dans le cas d'un corps de nombres se sp\'ecialise sur la
fonction d'Artin classique en presque toute place.
\begin{lemme}\label{lm:compat:sommedir}
(Compatibilit\'e \`a la somme directe)
Si $\rho=\rho_1\oplus \rho_2$,
alors
\begin{equation}
\Lm(X,G,\rho,t)=\Lm(X,G,\rho_1,t)\,\Lm(X,G,\rho_2,t).
\end{equation}
\end{lemme}
\begin{proof}
On a en effet $\polcar_{\rho_1\oplus\rho_2}=\polcar_{\rho_1}\polcar_{\rho_2}$.
\end{proof}
\begin{lemme}\label{lm:compat:quot}
(Compatibilit\'e au quotient) 
Soit $H$ un sous-groupe distingu\'e de $G$, $\pi\,:\,G\to G/H$
le morphisme quotient, $\rho$ une $\bQ$-repr\'esentation de
$G/H$ et $X$ une $G$-vari\'et\'e quasi-projective. 
Alors on a 
\begin{equation}
\Lm(X/H,G/H,\rho,t)=\Lm(X,G,\rho\circ \pi,t),
\end{equation}
plus pr\'ecis\'ement on a pour tout $n\geq 1$
\begin{equation}
\prod_{\substack{
\cI\in \bconj{G}\\
\cD\in \bconjc{G,\cI}
}}
\!\!
\polcar_{\rho\circ \pi,\cI,\cD}(t^n)^{-\psin{X,G,\cI,\cD,n}}
=
\prod_{
\substack{
\cI\in \bconj{G/H}\\
\cD\in \bconjc{G/H,\cI}}} 
\!\!
\polcar_{\rho,\cI,\cD}(t^n)^{-\psin{X/H,G/H,\cI,\cD,n}}.
\end{equation}
\end{lemme}
\begin{lemme}\label{lm:compat:res}
(Compatibilit\'e \`a la restriction)
Soit $H$ un sous-groupe de $G$,  $\rho$ une $\bQ$-repr\'esentation de
$G$ et $X$ une $H$-vari\'et\'e. Soit $Y$ la $G$-vari\'et\'e form\'ee de l'union
disjointe de $G/H$-copies de $X$. 
Alors
\begin{equation}
\Lm(X,H,\rho_{|_{H}},t)=\Lm(Y,G,\rho,t),
\end{equation}
plus pr\'ecis\'ement on a pour tout $n\geq 1$
\begin{equation}
\prod_{\substack{
\cI\in \bconj{H}
\\
\cD\in \bconjc{H,\cI}
}}
\polcar_{\rho_{|_{H}},\cI,\cD}(t^n)^{-\psin{X,H,\cI,\cD,n}}
=
\prod_{\substack{
\cI\in \bconj{G}
\\
\cD\in \bconjc{G,\cI}}} 
\polcar_{\rho,\cI,\cD}(t^n)^{-\psin{Y,G,\cI,\cD,n}}.
\end{equation}
\end{lemme}
\begin{lemme}\label{lm:compat:ind}
(Compatibilit\'e \`a l'induction)
Soit $H$ un sous-groupe de $G$, $\rho$ une $\bQ$-repr\'esentation  de $H$
et $X$ une $G$-vari\'et\'e.
Alors
\begin{equation}
\Lm(X,H,\rho,t)=\Lm(X,G,\Ind_{\,H}^{\,G}\rho,t).
\end{equation}
\end{lemme}
\begin{proof}
Le principe de la d\'emonstration de ces trois r\'esultats 
est le suivant : on montre que ces relations sont <<v\'erifi\'ees>> pour
toutes $k$-extensions pseudo-finies, en <<copiant>> 
la d\'emonstration des propri\'et\'es analogues des fonctions $L$ d'Artin
classiques, puis on applique le th\'eor\`eme de Denef et Loeser.

Nous nous contentons de donner la d\'emonstration du lemme
\ref{lm:compat:ind}, 
et nous supposerons pour simplifier le rev\^etement $X\to X/G$ non ramifi\'e,
la preuve dans le cas ramifi\'e  \'etant analogue.
Rappelons tout d'abord que si
$\rho$ est une repr\'esentation d'un sous-groupe $D$ d'indice $f$ d'un
groupe cyclique $C$ on a 
\begin{equation}\label{eq:polcar:ind}
\polcar_{\Ind_{D}^C\rho,\{e\},C}(t)=\polcar_{\rho,\{e\},D}\left(t^f\right).
\end{equation}

Soit $\cC$ un \'el\'ement de $\bconjc{G}$, et $C$ un \'el\'ement de
$\cC$. Soient $\cC_1,\dots,\cC_r$ les $H$-classes de conjugaison de
l'ensemble $\{C\cap H,\quad C\in \cC\}$. Pour $i=1,\dots,r$, soit
\begin{equation}
\left(C\backslash G/H\right)_i\eqdef \{g\in C\backslash G/H,\quad
g\,C\,g^{-1}\,\cap\,H\in \cC_i\}
\end{equation}
Le cardinal de cet ensemble est donc (cf. notations \ref{nota:fICJD})
$d^{\{e\},\cC}_{\{e\},\cC_i}$, on le note ici $d_i$.
On note \'egalement $f_i$ l'entier $f^{\{e\},\cC}_{\{e\},\cC_i}$.

D'apr\`es \cite[\S 7.4,Proposition 15]{Ser:rlgf}, on a 
\begin{equation}
\polcar_{\Ind_{\,H}^{\,G}\rho,\{e\},\cC}(t)
=
\prod_{g\in C\backslash
  G/H}\polcar_{\Ind^{\,g\,C\,g^{-1}}_{\,g\,C\,g^{-1}\cap H}\rho,\{e\},g\,C\,g^{-1}}(t)
\end{equation}
soit, d'apr\`es la relation \eqref{eq:polcar:ind}
\begin{equation}
\polcar_{\Ind_{\,H}^{\,G}\rho,\{e\},\cC}(t)
=
\prod_{1\leq i\leq r}\polcar_{\rho,\{e\},\cC_i}\left(t^{f_i}\right)^{d_i}
\end{equation}
Appliquant le lemme \ref{lm:psiC} on obtient
\begin{align}
\polcar_{\Ind_{\,H}^{\,G}\rho,\{e\},\cC}(t^n)^{\,-\psin{X,G,\{e\},\cC,n}}
&
=
\prod_{1\leq i\leq r}\polcar_{\rho,\{e\},\cC_i}\left(t^{\,nf_i}\right)^{\,-\psi^{\{e\},\cC}_{_{X,H,\{e\},\cC_i,nf_i}}}.
\end{align}
On en d\'eduit
\begin{align}
\prod_{n\geq 1} \prod_{\cC\in \bconjc{G}} \polcar_{\Ind_{\,H}^{\,G}\rho,\{e\},\cC}(t^n)^{\,-\psin{X,G,\{e\},\cC,n}}
\hskip-0.3\textwidth&
\\
&
=
\prod_{n\geq 1} 
\,
\prod_{f\geq 1}
\,
\prod_{\cD\in \bconjc{H}}
\polcar_{\rho,\{e\},\cD}
\left(t^{n\,f}\right)^{\,-
\sumu{
\substack{\cD\subset \cC\cap H\\ f^{\cC}_{\cD}=f}
}
\psi^{\{e\},\cC}_{_{X,H,\{e\},\cD,nf}}}
\\
&
=
\prod_{n\geq 1} 
\,
\prod_{\cD\in \bconjc{H}}
\polcar_{\rho,\{e\},\cD}
\left(t^{n}\right)^{\,
-
\sumu{f|n}
\,\,\,
\sumu{
\substack{\cD\subset \cC\cap H\\f^{\cC}_{\cD}=f}
}
\psi^{\{e\},\cC}_{_{X,H,\{e\},\cD,n}}}
\\
&
=
\prod_{n\geq 1} 
\,
\prod_{\cD\in \bconjc{H}}
\,
\polcar_{\rho,\{e\},\cD}
\left(t^{n}\right)^{\,
-
\sumu{\cD\subset \cC\cap H}
\psi^{\{e\},\cC}_{_{X,H,\{e\},\cD,n}}}
\\
\text{(lemme \ref{lm:psi:eq:sum:psi:c})}
&
=
\prod_{n\geq 1} 
\,
\prod_{\cD\in \bconjc{H}}
\polcar_{\rho,\{e\},\cD}
\left(t^{n}\right)^{\,-
\psin{X,H,\{e\},\cD,n}}
\end{align}
d'o\`u le r\'esultat.

\end{proof}
\begin{cor}\label{cor:lmdm:lm}
\begin{enumerate}
\item
La fonction $\Lm(X,G,\rho,t)$ co\"\i ncide avec l'image dans
$1+\kochmk{\bQ}_{\bQ}[[t]]^+$ de la fonction 
$\LmDM(X,G,\rho,t)$.
\item
Supposons que $k$ soit un corps de nombres.
Soit $\ell$ un nombre premier. Alors pour presque toute place
$\mfp$ de $k$, 
on a la relation
\begin{equation}\label{eq:trp:lm:lar}
\Trp(\chil\left(\Lm(X,G,\rho,t)\right))=\Lar(X_{\mfp},G,\rho,t).
\end{equation}
\end{enumerate}
\end{cor}
\begin{proof}
D'apr\`es le th\'eor\`eme d'Artin (\cite[II-45, proposition 25]{Ser:rlgf}), 
il existe un entier positif $d$
tel qu'on puisse \'ecrire
\begin{equation}
d\,\chi_{\rho}=\sum_{\substack{C}}n_{C}\,\chi_{\Ind_{C}^G(\triv_C)}
\end{equation}
o\`u la somme porte sur des sous-groupes cycliques $C$ de $G$ et les
$n_{C}$ sont des entiers.
D'apr\`es les lemmes \ref{lm:ltriv:zm}, \ref{lm:compat:sommedir} et \ref{lm:compat:ind}, on a donc
\begin{equation}\label{eq:rel:lm:prod:c}
\Lm(X,G,\rho,t)^{d}
=
\prod_{\substack{C<G\\C\text{ cyclique}}}\Lm(X,C,\triv_C,t)^{n_{C}}
=
\prod_{\substack{C<G\\C\text{ cyclique}}}\Zm(X/C,t)^{n_{C}}
\end{equation}
De m\^eme, d'apr\`es les propositions 2.7 et 2.10  de
\cite{DhMi:motivic_chebotarev} et le lemme \ref{lm:lmdmtriv}, on a
\begin{equation}
\LmDM(X,G,\rho,t)^{d}
=
\prod_{\substack{C<G\\C\text{ cyclique}}}\Zm(X/C,t)^{n_{C}}.
\end{equation}
Ainsi  $\Lm(X,G,\rho,t)$ et $\LmDM(X,G,\rho,t)$ sont deux \'el\'ements de
$1+\kochmkq{\bQ}[[t]]^+$ dont une puissance co\" \i ncide. 
Ils sont donc \'egaux.

Pour montrer le deuxi\`eme point, on remarque que la relation 
\eqref{eq:rel:lm:prod:c} entra\^\i ne
\begin{equation}\label{eq:rel:lm:prod:c:chil}
\chil\left(\Lm(X,G,\rho,t)\right)^{d}
=
\prod_{\substack{C<G\\C\text{ cyclique}}}\chil\left(\Zm(X/C,t))\right)^{n_{C}}.
\end{equation}
D'apr\`es le corollaire \ref{cor:delbano}, 
il existe un ensemble fini $S$ de places tel que pour tout $\mfp\notin S$ 
et pour tout $C$
on a la relation 
\begin{equation}
\Trp(\chil\left(\Zm(X/C,t))\right))=\ZHW((X/C)_{\mfp},t).
\end{equation}
D'apr\`es la relation \eqref{eq:rel:lm:prod:c:chil}
on en d\'eduit
que pour presque tout $\mfp$
on a 
\begin{equation}
\Trp(\chil\left(\Lm(X,G,\rho,t)\right))^{d}
=
\prod_{\substack{C<G\\C\text{ cyclique}}}\ZHW((X/C)_{\mfp},t)^{n_{C}}.
\end{equation}
Comme les fonctions $L$ d'Artin classiques sont \'egalement compatibles \`a la
somme directe, la restriction et l'induction, on en d\'eduit le r\'esultat.

\end{proof}
\begin{rem}
Soit $k$ un corps fini et
$
\#_k\,:\,\kochmk{\bQ}\to \bC
$
le morphisme \og nombre de $k$-points\fg~obtenu en consid\'erant la trace du
Frobenius sur une r\'ealisation $\ell$-adique. 
On a en particulier, pour toute $k$-vari\'et\'e
projective et lisse,  $\#_k \Zm(X,t)=\ZHW(X,t)$.
En fait, un raisonnement analoge \`a celui de \cite[\S 5.2]{DhMi:motivic_chebotarev}
montre en outre que si $G$ est un groupe fini agissant sur $X$, on a $\#_k \Zm(h(X)^G,t)=\ZHW(X/G,t)$.
Comme dans la d\'emonstration pr\'ec\'edente, le th\'eor\`eme d'Artin
permet alors de montrer qu'on a
pour toute $k$-vari\'et\'e $X$ projective et lisse et toute
$\bQ$-repr\'esentation $\rho$ d'un groupe fini $G$
la relation
$
\Lar(X,G,\rho,t)=\#_k \LmDM(X,G,\rho,t)$.
Une relation analogue est montr\'ee dans 
\cite{DhMi:motivic_chebotarev} dans le cas de repr\'esentations 
d\'efinies sur un corps contenant toutes les racines de l'unit\'e.
\end{rem}
\begin{rem}
La relation \eqref{eq:trp:lm:lar}
pourrait aussi \^etre vue comme d\'ecoulant de la comparaison des expressions
\eqref{eq:lm:ecr:L:class} et \eqref{eq:def:lm} et du fait que pour
tout $n\geq 1$ et presque tout $\mfp$ le
nombre de points modulo $\mfp$ de $\psin{X,G,\cI,\cD,n}$ est \'egal
au cardinal de l'ensemble
$((X/G)_{\mfp})^{(0)}_{n,\cI,\cD}$ (proposition \ref{prop:spec:phixgcin}).
Cependant, pour rendre l'argument rigoureux, il faudrait 
 disposer d'une version uniforme en $n$ de la proposition
 \ref{prop:spec:phixgcin} 
(cf. la remarque \ref{rem:prop:spec:phixgcin}).
\end{rem}
\begin{cor}
Si $X$ est de dimension au plus $1$, $\Lm(X,G,\rho,t)$ est rationnelle. 
\end{cor}
\begin{proof}
Ceci d\'ecoule
de la proposition \ref{prop:lmdmcourbe}
et du corollaire \ref{cor:lmdm:lm}.
\end{proof}
Nous verrons \`a la section suivante que dans le cas d'un corps $k$ de
caract\'eristique non nulle et d'une vari\'et\'e $X$ de dimension au plus
$1$, on peut encore d\'efinir des motifs virtuels
$\psin{X,G,\cI,\cC,n}$ de sorte que $\LmDM$ v\'erifie la d\'ecomposition en produit
eulerien motivique \eqref{eq:def:lm}.

\subsection{Formules et motifs virtuels associ\'es aux symboles d'Artin}\label{sec:symb:artin}

Soit $k$ un corps, $G$ un groupe fini, $\courbe$ une $k$-$G$-courbe
projective et lisse sur $k$, $Y=\courbe/G$ et $E$ un corps de
caract\'eristique z\'ero contenant
toutes les racines de l'unit\'e.  
Les auteurs de \cite{DhMi:motivic_chebotarev} d\'efinissent pour tout
entier $n\geq 1$ et pour toute
classe de conjugaison $C$ de $G$ le <<motif virtuel des \'el\'ements
de degr\'e $n$ de symbole d'Artin $C$>> 
not\'e
$\Ar(\courbe,G,C,n))_{n\geq 1} \in \kochmk{E}\otimes{\bQ}$
ayant la propri\'et\'e suivante : si $k$ est un corps global, 
pour presque toute place finie $\mfp$ et pour tout $n\geq 1$, 
$\Trp(\chil(\Ar(\courbe,G,C,n)))$ est le nombre d'\'el\'ements de
$((\courbe/G)_{\mfp})_{\text{\'et}}$ de degr\'e $n$ et de symbole d'Artin $C$.
La d\'efinition pr\'ecise de ces motifs virtuels est rappel\'ee ci-dessous.

Le but de cette section est d'expliquer (lorsque $k$ est de
caract\'eristique z\'ero) comment ces motifs
virtuels sont reli\'es aux motifs $(\psin{\courbe,G,\{e\},\cC,n})$.
Au passage, nous expliquerons \'egalement, pour un corps $k$ de
caract\'eristique non nulle, comment on peut donner un sens aux
motifs virtuels $(\psin{X,G,\cI,\cC,n})$ si $X$ est une $k$-vari\'et\'e
projective lisse de dimension au plus $1$.

La remarque de d\'epart est la suivante. Supposons que $k$ soit un corps de nombres.
D'apr\`es la proposition \ref{prop:spec:phixgcin} et la propri\'et\'e
\'evoqu\'ee ci-dessus, pour tout $n\geq 1$, on a pour presque toute place $\mfp$ l'\'egalit\'e
\begin{equation}\label{eq:rel:varphi:ar}
\Trp(\chil(\psin{\courbe,G,\{e\},\cC,n}))
=
\sum_{C\leadsto \cC}
\Trp(\chil(\Ar(\courbe,G,C,n))),
\end{equation}
o\`u, pour toute classe de conjugaison $C$ et tout $\cC\in \bconjc{G}$ on note $C \leadsto \cC$
la propri\'et\'e : tout \'el\'ement de $C$ engendre un \'el\'ement de $\cC$.

Nous allons montrer que la relation \ref{eq:rel:varphi:ar} est d\'ej\`a vraie
au niveau des motifs virtuels, i.e. la proposition suivante.
\begin{prop}\label{prop:symbartmot}
Soit $k$ un corps de caract\'eristique z\'ero, $G$ un groupe fini, $\courbe$ une $k$-$G$-courbe
projective et lisse sur $k$ 
et $E$ un corps de caract\'eristique z\'ero contenant
toutes les racines de l'unit\'e. 
Pour tout $n\geq 1$ et tout $\cC\in \bconjc{G}$ on a dans $\kochmk{E}\otimes \bQ$ la relation
\begin{equation}
\psin{\courbe,G,\{e\},\cC,n}=\sum_{C\leadsto \cC}\Ar(\courbe,G,C,n)
\end{equation}
\end{prop}
Avant toute chose, rappelons la d\'efinition des  motifs virtuels $(\Ar(\courbe,G,C,n))$. 
Sous les hypoth\`eses de la proposition \ref{prop:symbartmot},
les auteurs de \cite{DhMi:motivic_chebotarev}
associent \`a toute $E$-repr\'esentation $\rho$ de $G$ une fonction $L$ non ramifi\'ee 
d\'efinie comme suit : le lieu de ramification $(\courbe/G)^{\text{ram}}$ 
de $\courbe\to \courbe/G$ est une union finie de points ferm\'es de $\courbe/G$. 
Pour tel point $y$, on note $\cI_y$ (respectivement $\cD_y$) l'ensemble de ses groupes de
d\'ecomposition (respectivement d'inertie),  
$x_y$ un point de la fibre au-dessus de $y$, $D_y$ (respectivement $I_y$) 
le groupe de d\'ecomposition (respectivement d'inertie) associ\'e, et
$\rho_y$ la $\bQ$-repr\'esentation de $D_y/I_y$ induite par $\rho$.
On  pose alors, suivant \cite{DhMi:motivic_chebotarev},
\begin{equation}
\LmDMnr(\courbe,G,\rho,t) 
=
\LmDM(\courbe,G,\rho,t)
\,
\prod_{y\in (\courbe/G)^{\text{ram}}} 
\,
\LmDM(\Spec(\kappa_{x_y}),D_y/I_y,\rho_y,t)^{-1}.
\end{equation}
\begin{rem}
Tout comme pour la d\'efinition de $\LmDM$, il est inutile dans la
d\'efinition de $\LmDMnr$ de supposer que le corps $E$ contient toutes
les racines de l'unit\'e. En particulier, la d\'efinition a un sens pour une
$\bQ$-repr\'esentation $\rho$.
\end{rem}
\begin{rem}\label{rem:compat:LmDMnr:sd}
La compatibilit\'e de $\LmDM$ aux sommes directes de repr\'esentations (\cite[Proposition
2.8]{DhMi:motivic_chebotarev}) entra\^\i ne aussit\^ot la propri\'et\'e
similaire pour  $\LmDMnr$. 
Par ailleurs la remarque \ref{rem:compat:scal} montre
que $\LmDMnr$ est compatible au changement de coefficients.
\end{rem}

Soit $C$ une classe de conjugaison de $G$. 
On note pour $n\geq 1$
\begin{equation}\label{eq:def:ecPnC}
\ecP_n(C)\eqdef\{D\in \Conj{G},\quad x\in D\imply x^n \in C\}.
\end{equation}
On note $\ind_C$ la fonction qui vaut $1$ sur $C$ et $0$ sinon.
Il existe donc des nombres rationnels $m_{\rho,C}$
tels que 
\begin{equation}\label{eq:indC:mrhoC}
\ind_C=\sum_{\rho\in \irrE{G}} m_{\rho,C}\,\chi_{\rho}
\end{equation}
Les motifs virtuels $\Ar(\courbe,G,C,n)$ sont alors d\'efinis par
r\'ecurrence sur $n\geq 1$
\`a l'aide de la formule suivante
\begin{equation}\label{eq:form:artin}
\sum_{n\geq 1} \left(\sum_{\substack{d|n\\ D\in \ecP_{\frac{n}{d}}(C)}} \Ar(\courbe,G,D,d)\right) t^n
=\sum_{\rho\in \irrq{G}} m_{\rho,C}\,t\,\frac{d\log}{dt} L^{\text{nr}}(\courbe,G,\rho,t).
\end{equation}

\begin{lemme}\label{lm:form:LmDMnr}
Pour toute $\bQ$-repr\'esentation $\rho$, on a dans $\kochmkq{\bQ}$ la relation
\begin{equation}
\LmDMnr(\courbe,G,\rho,t) 
=
\prod_{n\geq 1}\,\,
\prod_{\cD\in \bconjc{G}} 
\polcar_{\rho,\{e\},\cD}\left(t^n\right)^{\,-\psin{\courbe,G,\{e\},\cD,n}}.
\end{equation}
\end{lemme}
\begin{proof}
D'apr\`es le point 1 du corollaire \ref{cor:lmdm:lm} et les d\'efinitions
de  $\LmDMnr$ et $\Lm$ il suffit de montrer qu'on a la relation
\begin{multline}
\prod_{n\geq 1}\,\,\prod_{\cI\in \bconj{G}\setminus \{e\}} \,\,\prod_{\cD\in \bconjc{G,\cI}} 
\polcar_{\rho,\cI,\cD}(t^n)^{-\psin{\courbe,G,\cI,\cC,n}}
\\
=
\prod_{y\in (\courbe/G)^{\text{ram}}} \Lm(\Spec(\kappa_{x_y}),D_y/I_y,\rho_y,t).
\end{multline}
Or pour $y\in (\courbe/G)^{\text{ram}}$ on  a
\begin{multline}
\Lm(\Spec(\kappa_{x_y}),D_y/I_y,\rho_y,t)
\\
=
\prod_{1\leq n\leq [\kappa(y):k]}
\,\,
\prod_{\substack{
D\leq D_y/I_y
}
}
\polcar_{\rho_y,\{e\},D}\left(t^{n}\right)^{\,-\psin{\Spec(\kappa_{x_y}),D_y/I_y,\{e\},D,n}}.
\end{multline}
Par ailleurs pour $n\geq 1$, $\cI\in \bconj{G}\setminus \{e\}$ et
$\cD\in \bconjc{G,\cI}$, 
on a pour toute $k$-extension pseudo finie $K$
\begin{equation}
\psin{\courbe,G,\cI,\cD,n}(K)
=
\bigsqcup_{\substack{
y\in (\courbe/G)^{\text{ram}}
\\
[\kappa_y:k]\geq n
\\
\cI_y=\cI
\\
\cD\subset\cD_y
}
}
\psin{\Spec(\kappa_{x_y}),D_y/I_y,\{e\},\pi_y(\cD),n}(K).
\end{equation}
o\`u $\pi_y(\cD)$ est l'image dans $D_y/I_y$ d'un \'el\'ement de $D$ de
$\cD$ v\'erifiant $D\subset D_y$.
Ceci d\'ecoule en effet de la remarque \ref{rem:phixgidn}.
On en d\'eduit l'\'egalit\'e
\begin{equation}\label{eq:psin:courbe}
\psin{\courbe,G,\cI,\cD,n}
=
\sum_{\substack{
y\in (\courbe/G)^{\text{ram}}
\\
[\kappa_y:k]\geq n
\\
\cI_y=\cI
\\
\cD\subset\cD_y
}
}
\psin{\Spec(\kappa_{x_y}),D_y/I_y,\pi_y(\cD),n}.
\end{equation} 
Le r\'esultat cherch\'e en d\'ecoule ais\'ement.
\end{proof}

\begin{proof}[D\'emonstration de la proposition \ref{prop:symbartmot}]
Pour tout $\rho\in \irrE{G}$ et tout $\cC\in \bconjc{G}$ on pose (cf. \eqref{eq:indC:mrhoC}) 
$m_{\rho,\cC}\eqdef\sumu{C\leadsto \cC}m_{\rho,C}$.
On note pour $n\geq 1$
\begin{equation}
\ecP_n(\cC)
\eqdef
\{\cD\in \bconjc{G},\quad 
\forall D
,\quad
D\leadsto \cD\imply \exists C
,\quad C\leadsto \cC,\,\,D\in \ecP_n(C)\}.
\end{equation}

En sommant \eqref{eq:form:artin} sur toutes les classes de conjugaison
$C$ v\'erifiant $C\leadsto \cC$, on obtient la relation
\begin{equation}\label{eq:form:artin:bis}
\sum_{n\geq 1} 
\left[
\sum_{d|n} 
\sum_{\cD\in \ecP_{\frac{d}{n}}(\cC)}
\left(
\sum_{
\substack{
D\leadsto \cD
}
}
\Ar(\courbe,G,D,d)
\right)
\right] t^n
=\sum_{\rho\in \irrE{G}} m_{\rho,\cC}\,t\,\frac{d\log}{dt} \LmDMnr(\courbe,G,\rho,t) 
\end{equation}
On rappelle que l'on d\'esigne par $\theta_{\cC}$ la fonction
\begin{equation}
\theta_{\cC}\,:\,g\mapsto
\left\{
\begin{array}{cl}
1&\text{si le groupe engendr\'e par $g$ est dans }\cC\\
0&\text{sinon.}
\end{array}
\right.
\end{equation}
En particulier, $\theta_{\cC}$ est une fonction $\bQ$-centrale. Il
existe donc des rationnels $(\wt{m}_{\rho,\cC})_{\rho\in \irrq{G}}$ tels
que
\begin{equation}\label{eq:rel:thetaC}
\theta_{\cC}=\sum_{\rho\in \irrq{G}}\wt{m}_{\rho,\cC}\,\chi_{\rho}.
\end{equation}
Comme $\theta_{\cC}=\sumu{C\leadsto \cC}\ind_{C}$, on a alors  
\begin{equation}
\sum_{\rho\in \irrq{G}}\wt{m}_{\rho,\cC}\,\chi_{\rho}=\sum_{\rho\in \irrE{G}}m_{\rho,\cC}\,\chi_{\rho}.
\end{equation}
D'apr\`es la remarque \ref{rem:compat:LmDMnr:sd}, il vient
\begin{equation}\label{eq:form:artin:ter}
\sum_{\rho\in \irrE{G}} m_{\rho,\cC}\,t\,\frac{d\log}{dt}
\LmDMnr(\courbe,G,\rho,t)
=
\sum_{\rho\in \irrq{G}} \wt{m}_{\rho,\cC}\,t\,\frac{d\log}{dt}
\LmDMnr(\courbe,G,\rho,t).
\end{equation}
Pour $\cC\in \bconjc{G}$, on note
$x_{\cC}$ un g\'en\'erateur d'un \'el\'ement de $\cC$.
On a alors pour toute $\bQ$-repr\'esentation $\rho$
\begin{equation}
t\,\frac{d\log}{dt} \polcar_{\rho,\{e\},\cC}(t^n)
=
-n\,\sum_{k\geq 1} \Tr(\rho(x_{\cC}^k))\,t^{\,n\,k}.
\end{equation}
D'apr\`es le lemme \ref{lm:form:LmDMnr}, on a donc 
\begin{align}
&\phantom{=}\sum_{\rho\in \irrq{G}} \wt{m}_{\rho,\cC}\,t\,\frac{d\log}{dt} \LmDMnr(\courbe,G,\rho,t)\\
&=
\sum_{\rho\in \irrq{G}} \wt{m}_{\rho,\cC}\,
\sum_{\cD\in \Conjc{G}} 
\sum_{n\geq 1}
-\psin{\courbe,G,\{e\},\cD,n}\,
t\,
\frac{d\log}{dt} \polcar_{\rho,\{e\},\cD}(t^n)
\\
&
=
\sum_{\rho\in \irrq{G}} 
\wt{m}_{\rho,\cC}\,
\sum_{\cD\in \Conjc{G}} 
\sum_{n\geq 1}
n\,\psin{\courbe,G,\{e\},\cD,n} 
\sum_{k\geq 1} \Tr(\rho(x_{\cD}^k))\,t^{\,n\,k}
\\
&=\sum_{n\geq 1}  \sum_{\cD\in \Conjc{G}} 
n\,\psin{\courbe,G,\{e\},\cD,n} 
\sum_{k\geq 1}  \sum_{\rho\in \irrq{G}} \wt{m}_{\rho,\cC} \Tr(\rho(x_{\cD}^k))\,t^{\,n\,k}.
\end{align} 
Or on a d'apr\`es \eqref{eq:rel:thetaC}
\begin{equation}
\sum_{\rho\in \irrq{G}} \wt{m}_{\rho,\cC} \Tr(\rho(x_{\cD}^k))
=\theta_{\cC}(x_{\cD}^k)
=
\left\{
\begin{array}{cl}
1&\text{si }\langle x_{\cD}^k\rangle \in \cC\text{ i.e. si }\cD\in \ecP_k(\cC)\\
0&\text{sinon.}
\end{array}
\right.
\end{equation}
Finalement on obtient
\begin{align}
\sum_{\rho\in \irrq{G}}\!\!\!\!\wt{m}_{\rho,\cC}\,t\,\frac{d\log}{dt}
\LmDMnr(\courbe,G,\rho,t)
&=\sum_{n\geq 1}\,\,   
\sum_{k\geq 1}\,\,   
\sum_{
\substack{
\cD\in \ecP_k(\cC)}} 
n\,\psin{\courbe,G,\{e\},\cD,n} 
\,t^{\,n\,k}
\\
\label{eq:sumrhoinirrqG}
&=\sum_{n\geq 1}   
\left[
\sum_{d|n}   
\left(
\sum_{\substack{
\cD\in \ecP_{\frac{n}{d}}(\cC)}} 
d\,\psin{\courbe,G,\{e\},\cD,d}
\right)
\right]
\,t^{\,n}.
\end{align} 
D'apr\`es \eqref{eq:form:artin:bis}, 
\eqref{eq:form:artin:ter} 
et cette derni\`ere relation, 
on obtient pour tout $n\geq 1$ l'\'egalit\'e
\begin{equation}
\sum_{d|n} 
\,\,
\sum_{\cD\in \ecP_{\frac{d}{n}}(\cC)}
\left(
\sum_{
\substack{
D\leadsto \cD
}
}
\Ar(\courbe,G,D,d)
\right)
=
\sum_{d|n}   
\,\,
\sum_{\substack{
\cD\in \ecP_{\frac{n}{d}}(\cC)}} 
d\,\psin{\courbe,G,\{e\},\cD,d}
\end{equation}
d'o\`u le r\'esultat cherch\'e.
\end{proof}
\begin{rem}\label{rem:nonram:dimun}
En \'ecrivant les relations 
\eqref{eq:sumrhoinirrqG} pour tous les \'el\'ements $\cC$ de $\bconjc{G}$
et en utilisant le point 2 du corollaire \ref{cor:lmdm:lm} et les propri\'et\'es standards
des fonctions $L$ d'Artin classiques,
on montre par r\'ecurrence sur $n$ que si $k$ est un corps de nombres, il existe
un nombre fini de places $S$ tel que pour tout $\mfp\notin S$, pour
tout $\cC\in \Conjc{G}$ et tout $n\geq 1$
$\Trp(\chil(\psin{\courbe,G,\{e\},\cD,n}))$ 
est \'egal au nombre de points ferm\'es de degr\'e $n$ de $(X/G)_{\mfp,\text{\'et}}$
 ayant un groupe de d\'ecomposition dans $\cC$
(cf. remarque \ref{rem:prop:spec:phixgcin}).
\end{rem}
\begin{rem}
Comme la d\'efinition de $\LmDMnr(\courbe,G,\rho,t)$ a un sens en
caract\'eristique non nulle,
les relations \eqref{eq:sumrhoinirrqG} permettent de d\'efinir
$\psin{\courbe,G,\{e\},\cD,n}$ en caract\'eristique non nulle.
Par ailleurs, si $k$ est de caract\'eristique z\'ero et si $L$ est une $k$-extension finie s\'eparable munie d'une
action libre de $G$ on a \'egalement pour tout $\cC\in \bconjc{G}$
\begin{equation}\label{eq:sumrhoinirrqG:dim0}
\sum_{\rho\in \irrq{G}}\!\!\!\!\wt{m}_{\rho,\cC}\,t\,\frac{d\log}{dt}
\Lm(\Spec(L),G,\rho,t)
=\sum_{n\geq 1}   
\left[
\sum_{d|n}   
\left(
\sum_{\substack{
\cD\in \ecP_{\frac{n}{d}}(\cC)}} 
d\,\psin{\Spec(L),G,\{e\},\cD,d}
\right)
\right]
\,t^{\,n}.
\end{equation}
Ces relations permettent de d\'efinir $\psin{\Spec(L),G,\{e\},\cD,d}$ 
(et donc $\psin{\Spec(L),G,\cI,\cD,d}$ d'apr\`es la remarque
\ref{rem:ex:dim0}) si $k$ est de caract\'eristique non nulle.
Finalement, compte tenu de la relation \eqref{eq:psin:courbe},   
il est possible pour un
corps $k$ de caract\'eristique non nulle de donner un sens aux motifs virtuels $\psin{X,G,\cI,\cD,n}$
pour toute $k$-$G$-vari\'et\'e projective et lisse $X$ de dimension au plus $1$.
En utilisant les relations d'orthogonalit\'e
\begin{equation}\label{eq:orth}
\sum_{\cC\in \bconjc{G}} 
\chi_{\rho_0}(\cC)\,
\wt{m}_{\rho,\cC}
=\left\{
\begin{array}{ccc}
0&\text{si}&\rho\neq \rho_0\\
1&\text{si}&\rho=\rho_0
\end{array}
\right.
,
\end{equation}
on peut alors montrer en partant des relations \eqref{eq:sumrhoinirrqG} 
et \eqref{eq:sumrhoinirrqG:dim0}
qu'on a pour
tout $\rho$ la d\'ecomposition
\begin{equation}
\LmDM(X,G,\rho,t)
=
\prod_{n\geq 1}\,\,
\prod_{\cI\in \bconj{G}}\, \,
\prod_{\cD\in \bconjc{G,\cI}}\,\,
\polcar_{\rho,\cI,\cD}(t^n)^{-\psin{X,G,\cI,\cD,n}}.
\end{equation}
\end{rem}

\section{Le volume de Tamagawa motivique}

\subsection{Le volume de Tamagawa classique} 
Soit $k$ un corps fini de cardinal $q$, $\courbe$ une $k$-courbe projective, lisse et
g\'eom\'etriquement int\`egre, $K$ son corps de fonctions,
$\scX\to\courbe$ un morphisme projectif, lisse, \`a fibres
g\'eom\'etriquement int\`egres, dont la fibre g\'en\'erique $X$ 
 v\'erifie les hypoth\`eses suivantes :   
\begin{hyps}\label{hyp:sck}
\begin{enumerate}
\item
On a $H^1(X,\str{X})=H^2(X,\str{X})=0$.
\item
$\Pic(\sep{X})$ est un $\galabs{K}$-module discret libre de rang fini
qui co\" \i ncide avec $\Pic(\clo{X})$.
\item
Le rang de $\Pic(\sep{X})$ co\"\i ncide avec le deuxi\`eme nombre de Betti de $X$.
\end{enumerate}
\end{hyps}
Sous ces hypoth\`eses, nous rappelons la d\'efinition, d\^ue \`a Peyre, du volume de Tamagawa de la
famille $\scX/\courbe$. Lorsque $\scX/\courbe$ v\'erifie l'approximation faible, 
ce volume appara\^\i t conjecturalement dans 
l'estimation asymptotique du nombre de sections de $\scX\to \courbe$ de
degr\'e anticanonique born\'e (cf. \cite{Pey:var_drap} pour plus de
d\'etails, ainsi que la remarque \ref{rem:AF} et la section \ref{subsec:lien:conj} ci-dessous).

Soit $\ecH$ un sous-groupe d'indice fini de 
$\galabs{K}$ agissant trivialement sur  $\Pic(\clo{X})$,
$K'\eqdef (\sep{K})^{\ecH}$, $G\eqdef \Gal(K'/K)$ 
et $\scD\to\courbe$ le $k$-rev\^etement galoisien ramifi\'e de groupe $G$
correspondant \`a l'extension $K'/K$.
On note $\rhoNS$ la $\bQ$-repr\'esentation de $G$ induite par l'action de
$\galabs{K}$ sur $\Pic(\clo{X})$.

Le volume de Tamagawa de $\scX/\courbe$
peut alors \^etre d\'efini comme le produit eulerien
\begin{equation}\label{eq:def:tam:clas}
\Vol(\scX/\courbe)
\eqdef
q^{(1-g(\courbe))\,\dim(X)}
\prod_{y\in \courbe^{(0)}}L_y(\scD,G,\rhoNS,q^{-1})^{-1}
\frac{\card{\scX_y(\kappa_y)}}{\card{\kappa_y}^{\dim(X)}}.
\end{equation}
Gr\^ace \`a la compatibilit\'e des fonctions $L$ d'Artin 
au quotient, cette d\'efinition est ind\'ependante du choix de $\ecH$.
Dans le cas o\`u la famille $\scX\to \courbe$ est constante, 
nous donnerons dans la section suivante une version motivique 
du volume de Tamagawa. 
Rappelons d'abord succinctement les arguments qui permettent \`a Peyre de
montrer la convergence du produit eulerien \eqref{eq:def:tam:clas}.
La convergence du produit eulerien motivique d\'efinissant le volume de
Tamagawa motivique  sera d\'emontr\'ee par une adaptation motivique de
ces arguments. Les hypoth\`eses \ref{hyp:sck}  ont les cons\'equences suivantes :
\begin{enumerate}
\item
$b_1(X)=0$ et $b_1(\scX_y)=0$ 
pour tout $y\in \courbe^{(0)}$ ;
\item
pour presque tout $y\in \courbe^{(0)}$, on a un isomorphisme 
\begin{equation}\label{eq:picxbar:picxybar}
\Pic(\sep{X})\isom \Pic(\clo{\scX}_{\!y})
\end{equation}
compatible aux actions de $\galabs{K}$ et
$\galabs{\kappa_y}$ ;
\item
pour tout $\ell$ distinct de la
caract\'eristique de $k$ et
pour presque tout $y\in \courbe^{(0)}$ 
on a un isomorphisme de $\galabs{\kappa_y}$-$\Ql$-module
\begin{equation}\label{eq:picxbar:bis}
\Pic(\clo{\scX_y})\otimes \Ql \isom H^2_{\ell}(\scX_y)\otimes \Ql(1).
\end{equation}
\end{enumerate}
Soit $y\in \courbe^{(0)}$ un \'el\'ement non ramifi\'e, et $F_y$ un
frobenius associ\'e.
Son image dans $\galabs{\kappa_y}$ est donc $F_{\kappa_y}$.
On a l'estimation
\begin{equation}
L_y(\scD,G,\rhoNS,q^{-1})^{-1}=1-\Tr\left(F_y|\Pic(\clo{X})\right)\,q^{-\deg(y)}+\bigo{}{q^{-2\,\deg(y)}}
\end{equation}
soit d'apr\`es l'isomorphisme \eqref{eq:picxbar:picxybar}
\begin{equation}\label{eq:lyscD}
L_y(\scD,G,\rhoNS,q^{-1})^{-1}=1-\Tr\left(F_{\kappa_y}|\Pic(\clo{\scX_y})\right)\,q^{-\deg(y)}+\bigo{}{q^{-2\,\deg(y)}}
\end{equation}
Par ailleurs, on a d'apr\`es la formule des traces de Grothendieck-Lefschetz
\begin{equation}\label{eq:cardscXy}
\card{\scX_y(\kappa_y)}=\sum_{0\leq r\leq 2\,\dim(X)}(-1)^r\,\Tr(F_{\kappa_y}| H^r_{\ell}(\scX_y))
\end{equation}
En utilisant le th\'eor\`eme de Deligne sur les valeurs propres du Frobenius, la nullit\'e du premier nombre de
Betti, l'isomorphisme \eqref{eq:picxbar:bis} et la dualit\'e de Poincar\'e on d\'eduit de \eqref{eq:cardscXy} l'estimation
\begin{equation}\label{eq:cardscXy:bis}
\card{\scX_y(\kappa_y)}
=
q^{\,\dim(X)\,\deg(y)}
+\Tr\left(F_{\kappa_y}|\Pic(\clo{\scX_y})\right)\,q^{(\dim(X)-1)\,\deg(y)}+\bigo{}{q^{(\dim(X)-\frac{3}{2})\,\deg(y)}}
\end{equation}
De \eqref{eq:lyscD} et \eqref{eq:cardscXy:bis}, 
on d\'eduit finalement que le produit eulerien
\begin{equation}\label{eq:prodeul}
\prod_{y\in \courbe^{(0)}}L_y(\scD,G,\rhoNS,q^{-1})^{-1}
\frac{\card{\scX(\kappa_y)}}{\card{\kappa_y}^{\dim(X)}}
\end{equation}
est absolument convergent.

\begin{rem}\label{rem:AF}
Pour $x\in \courbe^{(0)}$, notons $K_x$ le compl\'et\'e de $K$ pour la
valuation d\'efinie par $x$.
Le volume de Tamagawa \eqref{eq:def:tam:clas}
correspond  au volume
de l'espace $\prod_{x\in\courbe^{(0)}}X(K_x)$ pour une certaine
mesure ad\'elique d\'efinie par Peyre (cette derni\`ere d\'efinition est plus
g\'en\'erale et vaut pour des $K$-vari\'et\'es n'ayant pas n\'ecessairement
bonne r\'eduction partout), et le nombre de Tamagawa de la
famille $\scX/\courbe$ est d\'efini comme \'etant
le volume de l'adh\'erence de $X(K)$ pour cette m\^eme mesure. 
Il co\" \i ncide donc avec le volume de Tamagawa si $X$ v\'erifie
l'approximation faible, i.e. si $X(K)$ est dense dans
$\prod_{x\in\courbe^{(0)}}X(K_x)$
(par exemple si $X$ est rationnelle).
\end{rem}
\begin{rem}\label{rem:norm}
Peyre normalise la mesure utilis\'ee par la partie
principale de la fonction $L$ d'Artin en $t=q^{-1}$, i.e. par la quantit\'e
\begin{equation}
\left[(1-q\,t)^{\rg(\Pic(\clo{X})^{G})}\,\Lar(\scD,G,\rhoNS,t)\right]_{t=q^{-1}}.
\end{equation}
Nous nous \'ecartons de cette d\'efinition notamment \`a cause du fait que
l'analogue motivique naturel de cette quantit\'e n'a pas n\'ecessairement
de sens dans l'anneau que nous consid\'erons (cf. la remarque \ref{rem:norm:bis}).
Bien entendu il faudra tenir compte de ce choix dans la formulation de
la version motivique de la conjecture de Manin, ce qui se fait \`a peu
de frais, gr\^ace la proposition \ref{prop:lmdmcourbe}
(cf. la section \ref{subsec:lien:conj}). Une question pertinente pour
une \og bonne \fg~formulation de la conjecture de Manin semble
d'ailleurs plut\^ot \^etre li\'ee aux p\^oles qui doivent appara\^\i tre sur le cercle de convergence
(cf. la remarque \ref{rem:enonce:question}).
\end{rem}
\begin{rem}
Supposons que la famille $\scX\to \courbe$ soit isotriviale, 
i.e. qu'il existe une $k$-vari\'et\'e $X$ telle que pour
tout $y\in \courbe^{(0)}$ on a $\scX_y\isom X\times_k \kappa_y$.
Le produit eulerien \eqref{eq:prodeul} peut se
r\'e\'ecrire
\begin{equation}\label{eq:lm:ecr:L:class:isotriv}
q^{\,(1-g(\courbe))\,\dim(X)}\,
\prod_{n\geq 1} \,
\prod_{
\substack{
\cI\in \bconj{G} \,
\\
\cD\in \bconjc{G,\cI}
}
}
\polcar_{\rhoNS,\cI,\cD}(q^{-n})^{\,\card{\scD^{(0)}_{G,\cI,\cD,n}}}
\left(
\frac{\card{X(k_n)}}{q^{\,n\,\dim(X)}}
\right)
^{\card{\courbe^{(0)}_{n}}}
\end{equation}
\end{rem}
\begin{rem}\label{rem:triv}
On se place \`a pr\'esent dans le cas d'une famille triviale $\scX=X\times \courbe$.
Remarquons que les trois premi\`eres hypoth\`eses de \ref{hyp:sck} \'equivalent alors aux
hypoth\`eses suivantes : $H^1(X,\ecO_X)=H^2(X,\ecO_X)=0$, 
et $\Pic(\clo{X})$ est libre de rang fini \'egal \`a  $b_2(X)$.
On a alors un isomorphisme $\Pic(\sep{X_K})=\Pic(\clo{X})$ compatible
aux actions de $\galabs{K}$ \`a gauche et $\galabs{k}$ \`a droite.
On a $\scD\isom\courbe\times_k k'$.
En particulier le rev\^etement
$\scD\to \courbe$ est non ramifi\'e et \eqref{eq:lm:ecr:L:class:isotriv}
se r\'e\'ecrit
\begin{equation}\label{eq:lm:ecr:L:class:triv:1}
q^{\,(1-g(\courbe))\,\dim(X)}\,
\prod_{n\geq 1} \,
\prod_{
\substack{
C<G
}
}
\polcar_{\rhoNS,\{e\},C}(q^{-n})^{\,\card{\scD^{(0)}_{G,\{e\},C,n}}}
\left(
\frac{\card{X(k_n)}}{q^{\,n\,\dim(X)}}
\right)
^{\card{\courbe^{(0)}_{n}}}
\end{equation}
Pour tout sous-groupe cyclique $C$ de $G$ et tout $n\geq 1$, 
on pose $\eta_n(C)=1$  si $\card{C}=\frac{[k':k]}{n\wedge [k':k]}$
et $\eta_n(C)=0$ sinon.
On a donc
\begin{equation}\label{eq:courbe:0:n:I:C}
\card{\scD^{(0)}_{G,\{e\},C,n}}
=
\eta_n(C)\,\card{\courbe^{(0)}_n}
\end{equation}
et \eqref{eq:lm:ecr:L:class:triv:1}
  peut se
r\'e\'ecrire
\begin{equation}\label{eq:lm:ecr:L:class:bis}
q^{\,(1-g(\courbe))\,\dim(X)}\,
\prod_{n\geq 1} \,
\prod_{C<G}
\left(
\polcar_{\rhoNS,\{e\},C}(q^{-n})^{\eta_n(C)} 
\frac{\card{X(k_n)}}{q^{\,n\,\dim(X)}}
\right)^{\card{\courbe^{(0)}_{n}}}.
\end{equation}
\end{rem}

\subsection{Vers un analogue motivique du volume de Tamagawa}

D\'esormais, {\em on consid\`ere uniquement le cas d'une famille constante}.
Dans le cas d'une famille isotriviale, il est imm\'ediat de concevoir
un analogue motivique de l'expression \ref{eq:lm:ecr:L:class:isotriv}, 
mais nous ne savons pas d\'emontrer la convergence du produit eulerien motivique en question.

\subsubsection{D\'efinitions}
Soit $k$ un corps
et 
$X$ une $k$-vari\'et\'e projective, lisse et g\'eom\'etriquement int\`egre, v\'erifiant
les hypoth\`eses suivantes :

\begin{hyp}\label{hyp:X}
$\Pic(\clo{X})$ est un 
$\bZ$-module
libre de rang fini
\'egal \`a $b_2(X)$, qui co\"\i ncide avec $\Pic(\sep{X})$.
\end{hyp}
\begin{hyp}\label{hyp:X:bis}
Les groupes $H^1(X,\str{X})$ et  $H^2(X,\str{X})$ sont nuls.
\end{hyp}
\begin{rem}
Si $k$ est de caract\'eristique z\'ero, l'hypoth\`ese $H^1(X,\str{X})=0$ est
superflue au vu de l'hypoth\`ese \ref{hyp:X}.
\end{rem}
\begin{lemme}\label{lm:csqces:hyp}
L'hypoth\`ese \ref{hyp:X} entra\^ \i ne la nullit\'e de $b_1(X)$ et 
 l'existence d'un isomorphisme de
$\galabs{k}$-$\Ql$-modules
\begin{equation}\label{eq:picxbarbis}
\Pic(\clo{X})\otimes \Ql \isom H^2_{\ell}(X)\otimes \Ql(1).
\end{equation}
Si $k$ est de caract\'eristique z\'ero et $X$ est une vari\'et\'e de Fano, 
les hypoth\`eses \ref{hyp:X} et \ref{hyp:X:bis} sont v\'erifi\'ees. 
\end{lemme}
\begin{proof}
Tous les arguments n\'ecessaires se trouvent dans \cite{Pey:duke}, nous les rappelons.
Les suites exactes de Kummer induisent des suites
exactes de $\galabs{k}$-modules
\begin{equation}
0\to H^1_{\text{\'et}}(\clo{X},\Zl(1))\to T_{\ell}(\Pic(\clo{X})\to 0
\end{equation}
et
\begin{equation}
0\to \Pic(\clo{X})\otimes \Zl\to H^2_{\text{\'et}}(\clo{X},\Zl(1))\to
T_{\ell}(\Br(\clo{X}))\to 0
\end{equation}
o\`u $T_{\ell}(M)$ d\'esigne le module de Tate de $M$. Sous l'hypoth\`ese \ref{hyp:X}, $T_{\ell}(\Pic(\clo{X}))$ est
nul, et donc $b_1(X)=0$.

D'apr\`es \cite[Corollaire 3.4]{Gro:BrauerII}, le corang $\ell$-adique de $\Br(\clo{X})$ est la
diff\'erence entre $b_2(X)$ et le rang de $\Pic(\clo{X})$. Sous
l'hypoth\`ese \ref{hyp:X} il est donc nul, d'o\`u
$T_{\ell}(\Br(\clo{X}))=0$, d'o\`u l'isomorphisme recherch\'e.

Si $k$ est de caract\'eristique z\'ero et $X$ est une vari\'et\'e de Fano, l'hypoth\`ese \ref{hyp:X:bis} d\'ecoule
du th\'eor\`eme d'annulation de Kodaira. 
D'apr\`es \cite[Lemme 1.2.1 et 
remarques 1.2.2 et 1.2.3]{Pey:duke}) $\Pic(\clo{X})$ est libre de rang
fini \'egal \`a $b_2(X)$.
\end{proof}
\begin{lemme}\label{lm:iso:pic:pic}
Supposons que $k$ soit un corps global. Alors, sous les hypoth\`eses
\ref{hyp:X} et \ref{hyp:X:bis}, pour presque toute place finie $\mfp$, il
existe 
un isomorphisme
\begin{equation}
\Pic(\clo{X})\isom \Pic(\clo{X_{\mfp}})
\end{equation}
compatible aux actions de $\galabs{k}$ et
$\galabs{\kappa_{\mfp}}$. 
\end{lemme}
\begin{proof}
Ceci d\'ecoule de la d\'emonstration du lemme 2.2.1 de \cite{Pey:duke}.
\end{proof}
Sous les hypoth\`eses ci-dessus nous allons d\'efinir, pour toute $k$-courbe $\courbe$ projective,
lisse et  g\'eom\'etriquement int\`egre le volume de Tamagawa motivique de la famille
constante $X\times_k \courbe\to \courbe$.
Il serait int\'eressant d'\'etendre ces d\'efinitions \`a une famille non
constante, par exemple \`a une famille isotriviale.
\begin{notas}\label{notas:ecD}
Soit $\ecH$ un sous-groupe d'indice fini de 
$\galabs{k}$ agissant trivialement sur  $\Pic(\clo{X})$,
 $k'\eqdef \clo{k}^{\ecH}$ et $G=\Gal(k'/k)$. 
On note $\rhoNS$ la $\bQ$-repr\'esentation de $G$ induite par l'action de
$\galabs{k}$ sur $\Pic(\clo{X})$. On pose $\scD\eqdef\courbe\times_{\Spec(k)} \Spec(k')$, de sorte que  
$\scD\to\courbe$ est un $k$-rev\^etement galoisien \'etale de groupe $G$.
\end{notas}

Un analogue motivique naturel du produit eulerien \eqref{eq:lm:ecr:L:class:triv:1} 
est alors donn\'e (formellement du moins dans un premier temps) par le produit eulerien motivique
\begin{equation}\label{eq:an:mot}
\bL^{(1-g(\courbe))\,\dim(X)}
\,
\prod_{n\geq 1} 
\,
\prod_{\cC\in \bconjc{G}}
\,
\polcar_{\rhoNS,\{e\},\cC}(\bL^{-n})^{\,\psin{\scD,G,\{e\},\cC,n}}
\,
\left[\frac{\Phi_n(X)}{\bL^{n\,\dim(X)}}
\right]^{\psinx{\courbe}}.
\end{equation}
Notons que, gr\^ace \`a la compatibilit\'e au quotient des fonctions $\bL$ d'Artin motiviques, pour tout $n$ la s\'erie formelle 
\begin{equation}
\prod_{\cC\in \bconjc{G}}
\,
\polcar_{\rhoNS,\{e\},\cC}(t^{n})^{\,\psin{\scD,G,\{e\},\cC,n}}
\end{equation}
est ind\'ependante du choix de $\ecH$ (i.e. de l'extension de $k$ trivialisant $\Pic(\clo{X})$).
\begin{rem}\label{rem:norm:bis}
Si l'on d\'esire adapter \`a ce cadre la normalisation utilis\'ee par Peyre (cf. remarque
\ref{rem:norm}), il faut en outre multiplier l'expression pr\'ec\'edente
par la valeur en $t=\bL^{-1}$ de la s\'erie
\begin{equation}\label{eq:1LTLm}
(1-\bL\,t)^{\rg(\Pic(\clo{X})^{G})}\,\Lm(\scD,G,\rhoNS,t).
\end{equation}
Mais d'apr\`es le point \ref{item:3:prop:lmdmcourbe} de la proposition
\ref{prop:lmdmcourbe}, on a 
\begin{equation}
\Zm(\Pic(\clo{X}),t)^{-1}
\,
\Zm(\Pic(\clo{X}),\bL\,t)^{-1}
\,
\Lm(\courbe,G,\rho,t)\in 1+\kochmk{\bQ}[t]^+.
\end{equation}
On voit alors que la s\'erie \eqref{eq:1LTLm} ne converge pas en $t=\bL^{-1}$
pour la topologie que l'on va utiliser  si $\Pic(\clo{X})$ n'est pas
un module galoisien trivial :
si $M$ 
est une  repr\'esentation irr\'eductible non triviale 
la s\'erie $\Zm(M,t)$ ne converge pas en $t=1$. Une
solution pourrait \^etre d'inverser $\Zm^{-1}(M,1)=\sum_{n\geq
  0}(-1)^n\symb{\Alt^n M}$,
mais outre que le morphisme de localisation correspondant n'est pas
injectif, on ne disposera plus ensuite dans le cas d'un corps global du morphisme de sp\'ecialisation en
presque toute place $\mfp$ (on a
$\Trp(\chil(\Zm^{-1}(M,1)))=\det(\Id-\Frp|M)$ et cette quantit\'e est
nulle si $M$ contient la
$\galabs{\kappa_{\mfp}}$ repr\'esentation triviale).
\end{rem}

Nous allons maintenant donner, dans le cas d'un corps de
caract\'eristique z\'ero $k$, un analogue motivique de
l'expression \eqref{eq:lm:ecr:L:class:bis}.
Notons $\bconjc{G}_{\cC,n}$
l'ensemble des \'el\'ements $\cD$ de $\bconjc{G}$ qui v\'erifient la
propri\'et\'e suivante~:
si  $D$ est un \'el\'ement de $\cD$, il existe un \'el\'ement $C$ de $\cC$
qui v\'erifie $C<D$ et $\card{C}=\frac{\card{D}}{\card{D}\wedge n}$
Soit
\begin{equation}\label{eq:def:eta}
\etak{k',G,\cC,n}
\eqdef
\veeu{
\cD\in \bconjc{G}_{\cC,n}
}
\frev{\Spec(k'),\Gop,\cD}.
\end{equation}
\begin{prop}\label{prop:trp:eta}
\begin{enumerate}
\item
Soit $K$ une extension pseudo-finie  de $k$, $C_K<G$ un groupe de d\'ecomposition de $K$ dans
l'extension $k\to k'$, et $n\geq 1$.  
Notons $K_n$ une extension de degr\'e $n$ de $K$.
Alors pour  tout $\cC\in \bconjc{G}$, on a
\begin{equation}
\etak{k',G,\cC,n}(K)
=
\left\{
\begin{array}{cl}
\{e\}&\text{s'il existe $C\in\cC$ tel que $C<C_K$ et }\card{C}=\frac{\card{C_K}}{n\wedge \card{C_K}}\\
\vide&\text{sinon.}
\end{array}
\right.
\end{equation}
En d'autres termes, $\etak{k',G,\cC,n}(K)=\{e\}$ si et seulement
si $K_n$ a pour d\'ecomposition $\cC$ dans l'extension $k'/k$. 
\item
\label{item:prop:trp:eta:spec}
Supposons que $k$ soit un corps de nombres. Il existe un ensemble fini
$S$ de places de $k$ tel que pour tout $\mfp\notin S$ et tout $n\geq 1$,
on ait, en notant $\cC_{\mfp}$ la classe des groupes de d\'ecomposition de $\mfp$
dans l'extension $k\to k'$,
\begin{equation}
\Trp\left[\chil\left(\etak{k',G,\cC,n}\right)\right]
=
\left\{
\begin{array}{cl}
1&\text{s'il existe }C\in \cC\text{ et }C_{\mfp}\in \cC_{\mfp}\text{
  tel que }C<C_{\mfp}\text{ et }\card{C}=\frac{\card{C_{\mfp}}}{n\wedge \card{C_{\mfp}}}\\
0&\text{sinon.}
\end{array}
\right.
\end{equation}
En d'autres termes, si on note $C_{\mfp}$ un \'el\'ement de $\cC_{\mfp}$
et $\kappa'_{\mfp}/\kappa_{\mfp}$ l'extension de corps r\'esiduels
correspondante, le membre de gauche de l'expression ci-dessus vaut
$1$ si l'extension $\kappa_{\mfp,n}/\kappa_{\mfp}$  a pour groupe de
d\'ecomposition dans $\kappa'_{\mfp}/\kappa_{\mfp}$ un \'el\'ement de $\cC$,
et vaut $0$ sinon.
\item
On a pour tout $\cC\in \bconjc{G}$ et tout $n\geq 1$
\begin{equation}\label{eq:psin=etak}
\psin{\scD,G,\{e\},\cC,n}
=\etak{k',G,\cC,n}\,.\,\psinx{\courbe}.
\end{equation}
\end{enumerate}
\end{prop}
\begin{proof}
Le premier point d\'ecoule  de la d\'efinition de $\etak{k',G,\cC,n}$ et des propri\'et\'es
classiques des groupes de d\'ecomposition.
Le deuxi\`eme point d\'ecoule de ces m\^emes d\'efinitions et propri\'et\'es,
ainsi que du point \ref{item:nbre:pts:mod:p}
du th\'eor\`eme \ref{thm:denefloeser}.

Pour montrer le troisi\`eme point, on se ram\`ene au cas o\`u $\courbe$ est
affine.
Alors, d'apr\`es le premier point, les formules $\psin{\scD,G,\cC,n}$
et $\psinx{\courbe}\wedge \etak{k',G,\cC,n}$ sont deux formules de
$\Sym^n(\courbe)$
dont les $K$-points co\"\i ncident pour toute $k$-extension
pseudo-finie $K$, d'o\`u le r\'esultat.
\end{proof}

D'apr\`es 
\eqref{eq:psin=etak}
le produit eulerien \eqref{eq:an:mot} se r\'e\'ecrit (formellement du moins)
\begin{equation}\label{eq:an:mot:bis}
\bL^{(1-g(\courbe))\,\dim(X)}
\,
\prod_{n\geq 1} 
\,
\prod_{\cC\in \bconjc{G}}
\,
\left[
\polcar_{\rhoNS,\{e\},\cC}(\bL^{-n})^{\,\etak{k',G,\cC,n}}
\,
\frac{\Phi_n(X)}{\bL^{n\,\dim(X)}}
\right]^{\psinx{\courbe}}.
\end{equation}
C'est sous cette forme que nous allons d\'efinir le volume de Tamagawa motivique en caract\'eristique non nulle.
Dans ce cas, il n'est malheureusement pas clair que cette derni\`ere forme soit
\'equivalente \`a \eqref{eq:an:mot}, i.e. que la relation
\eqref{eq:psin=etak} soit v\'erifi\'ee.
Si $k$ est un corps global de caract\'eristique nulle, l'analogue du point \ref{item:prop:trp:eta:spec}
de la proposition \ref{prop:trp:eta} est encore valide, d'apr\`es la remarque \ref{rem:car:non:nulle:bis}.
Si $k$ est fini, toujours d'apr\`es la remarque \ref{rem:car:non:nulle:bis}, 
l'analogue suivant du point \ref{item:prop:trp:eta:spec}
de la proposition \ref{prop:trp:eta}.
\begin{prop}\label{prop:trp:eta:spec:fin}
Supposons que  $k$ soit fini.
Alors pour tout $n\geq 1$, tout $m\geq 1$ et tout sous-groupe $C$
de $G$,
$\Tr(F_k^m|\chil(\etak{k',G,C,n}))$ vaut $1$
si $k_{m\,n}/k_m$ a pour groupe de d\'ecomposition $C$ dans
l'extension $k'\otimes_k k_m/k_m$ et $0$ sinon.
\end{prop}
Nous allons montrer que le produit eulerien motivique \eqref{eq:an:mot:bis}
converge effectivement dans une certaine compl\'etion de l'anneau
$\kochmkq{\bQ}$ et, 
dans le cas o\`u $k$ est un corps global, se
sp\'ecialise sur le volume de Tamagawa de $X_{\mfp}\times \courbe_{\mfp}/\courbe_{\mfp}$ pour presque tout $\mfp$.

\subsection{Topologie utilis\'ee} 
Soit $A$ et $B$ des anneaux et 
$
\varphi\,:\,A\longto B[u,u^{-1}]
$
un morphisme d'anneaux.
On d\'efinit une filtration d\'ecroissante de $A$ par des sous-groupes en posant pour $i\in \bZ$
$
\fil^{i}A=\{a\in A,\,\deg(\varphi(a))\leq -i\}.
$

\begin{nota}\label{nota:hatAphi}
On pose 
$
\hat{A}^{\varphi} 
\eqdef 
\underset{\longlto}{\lim}\,\,A /\fil^{i}A.
$
\end{nota}
On a 
$
\fil^{i}A\,.\,\fil^{j}A\subset \fil^{i+j}A
$
ce qui permet de munir 
$\hat{A}^{\varphi}$
d'une structure d'anneau. Le morphisme $\varphi$ s'\'etend alors en un morphisme d'anneaux
de $\hat{A}^{\varphi}$ vers l'anneau des s\'eries de Laurent \`a
coefficients dans $B$ i.e $\varphi\,:\,\hat{A}^{\varphi}\to
B((u^{-1}))$.

On a imm\'ediatemment le crit\`ere de convergence suivants.
\begin{lemme}\label{lm:crit:convergence}
On suppose que $A$ est une $\bQ$-alg\`ebre.
Soit $(a_n)_{n\geq 0}$ et $(b_n)_{n\geq 0}$ deux suites d'\'el\'ements
de $A$. 
On suppose qu'on a pour tout $n$
$\deg(\varphi(b_n))\geq 0$,
et
$\deg(\varphi(a_n))+\deg(\varphi(b_n))<0$,
et qu'on a 
$
\deg(\varphi(a_n))+\deg(\varphi(b_n))\longto -\infty$.
Alors le produit
$
\prod_{n\geq 0} (1+a_n)^{b_n}
$
converge dans $\hat{A}^{\varphi}$, et on a 
\begin{equation}
\varphi\left(\prod_{n\geq 0} (1+a_n)^{b_n}\right)
=
\prod_{n\geq 0} (1+\varphi(a_n))^{\varphi(b_n)}.
\end{equation}
\end{lemme}

Nous allons d\'efinir le volume de Tamagawa motivique  comme
un \'el\'ement de $\hml$ o\`u $\Poincl$ est le polyn\^ome de Poincar\'e virtuel
$\ell$-adique.
L'int\'er\^et d'utiliser la r\'ealisation $\ell$-adique est de pouvoir
sp\'ecialiser le r\'esultat dans le cas d'un corps global ou d'un corps
fini. On montrera que cette sp\'ecialisation donne bien le r\'esultat
attendu, \`a savoir le volume de Tamagawa classique. Ceci \'etant, on
aurait pu utiliser une autre cohomologie de Weil. Dans le cas d'une
surface, on peut m\^eme utiliser un polyn\^ome de Poincar\'e absolu, cf. la
section \ref{subsec:vrai:motif}.

\subsection{\'Enonc\'e du r\'esultat}

\begin{thm}\label{thm:princ}
Soit $k$ un corps, $\courbe$ une $k$-courbe projective, lisse et
g\'eom\'etriquement int\`egre et $X$ une $k$-vari\'et\'e projective, lisse,
g\'eom\'etriquement int\`egre et v\'erifiant les hypoth\`eses \ref{hyp:X}
et \ref{hyp:X:bis}.
Le produit eulerien motivique
\begin{equation}\label{eq:def:tam:mot}
\bL^{(1-g(\courbe))\,\dim(X)}
\,
\prod_{n\geq 1} 
\left[
\prod_{\cC\in \bconjc{G}}
\polcar_{\rhoNS,\{e\},\cC}(\bL^{-n})^{\etak{k',G,\cC,n}}\,\frac{\Phi_n(X)}{\bL^{n\,\dim(X)}}
\right]^{\psinx{\courbe}}
\end{equation}
converge dans $\hml$. 
On l'appelle \termin{volume de Tamagawa motivique  de $X\times \courbe$}, 
et on le note $\Volm(X\times \courbe/\courbe)$.
\end{thm}
Ce th\'eor\`eme sera d\'emontr\'e \`a la sous-section \ref{subsec:dem:thm:princ}.
\begin{rem}
Si $k$ est un corps quelconque, on peut d\'efinir un polyn\^ome de
Poincar\'e virtuel $\ell$-adique $\Poincl\,:\,\kovark\to \kogkql[u]$ qui
en caract\'eristique z\'ero se factorise par $\chiv$ (si $k$ est de type
fini, ceci provient de l'existence d'une
filtration par le poids sur les groupes de cohomologie $\ell$-adique \`a
support compact ; pour le cas g\'en\'eral cf. \cite{Eke:class}).
On peut alors se poser la question de 
la convergence de \eqref{eq:def:tam:mot} dans
$\wh{\cM_k\otimes \bQ}^{\Poincl}$ o\`u $\cM_k\eqdef\kovark[\bL^{-1}]$
(cf. les d\'efinitions \ref{def:phinv}, \ref{def:eq:relpsinphin}
et \ref{rem:car:non:nulle:bis}). 
En caract\'eristique z\'ero, compte tenu du fait que la situation 
\og se
factorise \fg~\`a travers $\chiv$, la r\'eponse \`a
cette question est positive. En caract\'eristique non nulle, en
l'absence d'analogue de $\chiv$, nous ne connaissons pas la r\'eponse. L'un
des probl\`emes qui se posent est qu'on n'est plus a priori assur\'e de
l'\'egalit\'e $\Poincl(\Zm(X,t))=\Poincl(\Zv(X,t))$, et donc de la validit\'e
de la relation \eqref{eq:poinch:phin} pour $\phinvar{X}$.

On peut \'egalement consid\'erer la convergence dans le compl\'et\'e
$\wh{\cM_k\otimes \bQ}^{\text{dim}}$
de $\cM_k\otimes \bQ$ pour la filtration dimensionnelle, utilis\'ee dans la th\'eorie
de l'int\'egration motivique (cf. \cite[\S 3.2]{DeLo:germs}). Comme une
vari\'et\'e de dimension $n$ a un polyn\^ome de Poincar\'e virtuel 
de
degr\'e $2\,n$, on a un morphisme continu naturel ${\wh{\cM_k\otimes \bQ}}^{\text{dim}}
\to {\wh{\cM_k\otimes \bQ}}^{\Poincl}$, mais nous ne savons pas s'il est injectif, et nous ignorons
si  \eqref{eq:def:tam:mot} converge dans ${\wh{\cM_k\otimes \bQ}}^{\text{dim}}$.
\end{rem}

\begin{defi}
Soit $k$ un corps global, et $\mfp$ une place finie
de $k$.
Un \'el\'ement de $\kogkqlq((u^{-1}))$
est dit 
\termin{pur en $\mfp$}
s'il s'\'ecrit
$
\sum_{n\leq n_0} a_n\,u^n
$
o\`u,  pour tout $n\leq n_0$, $a_n$ est une combinaison lin\'eaire d'\'el\'ements
de la forme $\symb{V}$ o\`u
les valeurs propres de $\Fr_{\mfp}$ agissant sur $V^{I_{\mfp}}$ 
sont des nombres alg\'ebriques dont tous les conjugu\'es complexes ont pour module $N(\mfp)^{\frac{n}{2}}$.
Il est dit \termin{pur} s'il est pur en presque tout $\mfp$ 
(en particulier l'image par $\Poincl $ d'un \'el\'ement de $\kochmk{\bQ}$ 
est pur).
\end{defi}

\begin{thm}\label{thm:princ:spe}
On conserve les hypoth\`ese du th\'eor\`eme \ref{thm:princ}
et on suppose en outre que $k$ est un corps global. 
Alors, pour presque toute place finie $\mfp$,  
$\Poincl(\Volm(X\times \courbe/\courbe))$ est pur en $\mfp$,
et la s\'erie
\begin{equation}
\Trp\left[\Poincl (\Volm(X\times \courbe/\courbe))\right]\in \bC[[u^{-1}]]
\end{equation}
converge absolument en $u=-1$ vers le volume de Tamagawa  
$\Vol(X_{\mfp}\times \courbe_{\mfp}/\courbe_{\mfp})$
d\'efini par le produit eulerien \eqref{eq:def:tam:clas}.
\end{thm}
Ce th\'eor\`eme sera d\'emontr\'e \`a la sous-section \ref{subsec:dem:thm:princ:spe}.
\begin{thm}\label{thm:princ:spe:fin}
On conserve les hypoth\`ese du th\'eor\`eme \ref{thm:princ}
et on suppose en outre que $k$ est un corps fini. 
Alors pour tout entier $m$ v\'erifiant
\begin{equation}\label{eq:cond:bete}
m>\frac{1}{2}\log_q(1+\Sup_i b_i(X))
\end{equation}
la s\'erie
\begin{equation}
\Tr\left[F_k^m|\Poincl(\Volm(X\times \courbe/\courbe))\right]\in \bC[[u^{-1}]]
\end{equation}
converge absolument en $u=-1$ vers le volume de Tamagawa  
$\Vol(X_{k_m}\times \courbe_{k_m}/\courbe_{k_m})$
d\'efini par le produit eulerien \eqref{eq:def:tam:clas}.
\end{thm}
Ce th\'eor\`eme sera d\'emontr\'e \`a la sous-section \ref{subsec:dem:thm:princ:spe:fin}.
\begin{rem}\label{rem:princ:spe:fin}
La condition \eqref{eq:cond:bete} semble bien entendu artificielle, et
on aimerait  s'en d\'ebarrasser. Elle serait superflue si par
exemple $\Poincl(\Volm(X\times \courbe/\courbe))$ \'etait \`a croissance
polyn\^omiale born\'ee au sens de \cite[\S 2]{Eke:class}, mais nous ne
savons pas si cette propri\'et\'e est v\'erifi\'ee.
\end{rem}

\subsection{Quelques lemmes}
\begin{lemme}\label{lm:chieqLchirho}
Soit $k'$ une extension galoisienne finie de $k$, de groupe $G$.
Soit $\rho$ une $\bQ$-repr\'esentation de $G$.
Alors $\chieq(\Spec(k'),\chi_{\rhop})$ (cf. le th\'eor\`eme \ref{thm:chiveq} 
en caract\'eristique nulle et la
d\'efinition \ref{rem:car:non:nulle} en caract\'eristique non nulle) 
co\" \i ncide avec la classe de $V_{\rho}$ dans $\koAMk{\bQ}$.
\end{lemme}
\begin{proof}
D'apr\`es le point \ref{item:thm:chiveq:projlisse} du th\'eor\`eme
\ref{thm:chiveq} en caract\'eristique nulle et la remarque
\ref{rem:car:non:nulle}
en caract\'eristique non nulle, 
via l'\'equivalence de cat\'egories \eqref{eq:cat:artin},
$\chieq(\Spec(k'),\chi_{\rhop})$ s'identifie \`a la classe dans
$\koAMk{\bQ}$
de l'image du projecteur de $V_{\rho}\otimes \bQ[G]$
donn\'e par
\begin{equation}
\frac{1}{\card{G}}\sum_{g\in \Gop}\rhop(g^{-1})\otimes g^{\ast}
=
\frac{1}{\card{G}}\sum_{g\in G}\rho(g)\otimes \rho_d(g).
\end{equation}
D'apr\`es le lemme \ref{lm:well:known}, cette image est isomorphe \`a $V_{\rho}$.
\end{proof}

\begin{lemme}\label{lm:chiltrrho}
Soit $k'$ une extension galoisienne finie de $k$, de groupe $G$.
Soit $\rho_0$ une $\bQ$-repr\'esentation de $G$.
On a dans $\kochmk{\bQ}_{\bQ}$ la relation 
\begin{equation}
\sum_{\cC\in \bconjc{G}}
\chi_{\rho_0}(\cC)\,\frev{\Spec(k'),\Gop,\cC}
=\symb{V_{\rho_0}}.
\end{equation}
\end{lemme}
\begin{proof}
On peut supposer $\rho_0$ irr\'eductible. Soit $\cC\in \bconjc{G}$. 
On rappelle que $\theta_\cC$ est la fonction qui \`a $g\in G$ associe
$1$ si le groupe engendr\'e par $g$ est dans $\cC$ et $0$ sinon.
C'est une fonction $\bQ$-centrale \`a valeurs dans $\bQ$.
Il existe donc des \'el\'ements $m_{\rho,\cC}\in \bQ$
tels que  
$
\theta_\cC=\sum_{\rho\in \irrq{\Gop}} m_{\rho,\cC}\,\chi_{\rho}$. 
On a alors d'apr\`es les th\'eor\`emes \ref{thm:chiveq} et
\ref{thm:denefloeser} en caract\'eristique nulle et les d\'efinition \ref{def:car:non:nulle}
 et \ref{def:car:non:nulle:bis} en
caract\'eristique non nulle
\begin{equation}\label{eq:phiLGC:sumrho}
\frev{\Spec(k'),\Gop,\cC}=\sum_{\rho\in \irrq{\Gop}} m_{\rho,\cC}\,\chieq(\Spec(k'),\chi_{\rho}).
\end{equation}
soit d'apr\`es le lemme \ref{lm:chieqLchirho}
\begin{equation}\label{eq:phiLGC:sumrho:bis}
\frev{\Spec(k'),\Gop,\cC}=\sum_{\rho\in \irrq{\Gop}} m_{\rho,\cC}\,\symb{V_{\rhop}}.
\end{equation}
Le lemme d\'ecoule alors  des relations d'orthogonalit\'e
\eqref{eq:orth}. 
\end{proof}

\begin{lemme}\label{lm:eq:rel:trace:rho}
Soit $k'$ une extension galoisienne finie de $k$, de groupe $G$.
Pour tout $n\geq 1$ et toute $\bQ$-repr\'esentation $\rho$ de $G$ ,
on a dans $\koAMk{\bQ}\otimes \bQ$ la relation
\begin{equation}\label{eq:rel:trace:rho}
\sum_{\cC\in \bconjc{G}}\chi_{\rho}(\cC)\,\chil(\etak{k',G,\cC,n})=
P_{\dim(\rho),n}\left(\symb{\wedgeo{j} V_{\rho}}\right)_{1\leq
  j\leq \dim(\rho)}
\end{equation}
\end{lemme}
\begin{proof}
Pour $\cC\in \bconjc{G}$ on note $\theta_{\cC,n}$ la fonction
$\bQ$-centrale qui \`a $g\in G$ associe $1$ si le groupe engendr\'e par $g^n$ est dans $\cC$
et $0$ sinon. En particulier, on a les relations
$
\sumu{\cC\in \bconjc{G}}
\chi_{\rho}(\cC)\,\theta_{\cC,n}=\chi_{\rho}^n
$
et
$
\theta_{\cC,n}
=\sumu{
\cD\in \bconjc{G}_{\cC,n}
}
\theta_\cD.
$
On a donc
\begin{align}
\sum_{\cC\in \bconjc{G}}\chi_{\rho}(\cC)\,\etak{k',G,\cC,n}
&
=
\sum_{\cC\in \bconjc{G}}\chi_{\rho}(\cC)\,
\sum_{
\cD\in \bconjc{G}_{\cC,n}
}
\frev{\Spec(k'),G^{\text{op}},\cD}
\\
&
=
\sum_{\cC\in \bconjc{G}}\chi_{\rho}(\cC)\,
\sum_{
\cD\in \bconjc{G}_{\cC,n}
}
\chieq(\Spec(k'),\theta_{\cD})
\\
&
=
\sum_{\cC\in \bconjc{G}}\chi_{\rho}(\cC)\,
\chieq(\Spec(k'),\theta_{\cC,n})
\\
&
=
\chieq(\Spec(k'),\chi_{\rhop}^n).
\end{align}
Par ailleurs, il d\'ecoule de la remarque
\ref{rem:trfpdimvm:trfn}
que l'\'el\'ement de $K_0\left(\GQvect\right)$ 
donn\'e par $P_{\dim(\rho),n}\left(\symb{\wedgeo{j} V_\rho}\right)$
a pour caract\`ere $\chi_{\rho}^n$. D'apr\`es le lemme \ref{lm:chieqLchirho},
il est donc \'egal \`a 
$
\chieq(\Spec(k'),\chi_{\rho}^n)
$.

\end{proof}

\subsection{D\'emonstration du th\'eor\`eme \ref{thm:princ}}
\label{subsec:dem:thm:princ}
Notons $\ell$ la classe de $\Ql(-1)$ dans $\kogkql$.
Pour $n\geq 1$, soit $P_{X,\ell,n}(u)$
l'\'el\'ement de $1+\kogkql[[u^{-1}]]^+$
d\'efini par
\begin{align}
P_{X,\ell,n}(u)
&
\eqdef
\Poincl 
\left(
\prod_{\cC\in \bconjc{G}}
\polcar_{\rhoNS,\{e\},\cC}(\bL^{-n})^{\,\etak{k',G,\cC,n}}\,\frac{\Phi_n(X)}{\bL^{n\,\dim(X)}}
\right)
\\
&=
\prod_{\cC\in \bconjc{G}}
\polcar_{\rhoNS,\{e\},\cC}(\ell^{-n}\,u^{-2\,n})^{\,\chil(\etak{k',G,\cC,n})}\,
\frac{\Poincl(\Phi_n(X))}
{\ell^{n\,\dim(X)}\,u^{\,2\,n\,\dim(X)}}.
\end{align}
Il d\'ecoule de la relation \eqref{eq:poinch:phin} qu'on a
$
\deg(\Poinc_H(\Phi_n(\courbe)))=2\,n.
$
Ainsi, d'apr\`es la relation \eqref{eq:relpsinphin}
et la remarque \ref{rm:eq:relpsinphin}, on a pour tout $n$  
\begin{equation}
0 \leq \deg(\Poincl (\psinx{\courbe})\leq 2\,n.
\end{equation}
En vertu du lemme \ref{lm:crit:convergence}, il suffit 
donc
pour \'etablir la convergence dans $\hml$ du produit
eulerien motivique \eqref{eq:def:tam:mot} 
de montrer qu'on a 
\begin{equation}\label{eq:degpnlx}
\deg_u(P_{X,\ell,n}(u)-1)\leq -3\,n.
\end{equation}
D'apr\`es la proposition \ref{prop:poinchPhi_d} et le fait que $b_1(X)$
est nul, 
il existe un polyn\^ome $Q_0$ \`a coefficients dans $\kogkql$ v\'erifiant
\begin{align}
\Poincl (\Phi_n(X))&=\ell^{\,n\,\dim X}\,u^{\,2\,n\,\dim(X)}
+
P_{b_2(X),n}\left(\classe{\wedgeo{j} H_{\ell}^{2\,\dim(X)-2}(X)}\right)\,u^{\,n\,(2\,\dim(X)-2)}
\notag\\
&\phantom{=}+u^{n\,(2\,\dim(X)-3)}\,Q_0(u^{-1}).
\end{align}
D'apr\`es \eqref{eq:picxbarbis} et la 
dualit\'e de Poincar\'e, on a l'\'egalit\'e
\begin{equation}
\classe{H_{\ell}^{2\,\dim(X)-2}(X)}=
\classe{\Pic(X)^{\vee}}\,\ell^{\,\dim(X)-1}.
\end{equation}
Compte tenu du fait qu'une $\bQ$-repr\'esentation est isomorphe \`a sa duale,
on en tire pour tout $j$ la relation
\begin{equation}
\classe{\wedgeo{j} H_{\ell}^{2\,\dim(X)-2}(X)}
= \ell^{\,j\,(\dim(X)-1)}\,\classe{\wedgeo{j} \Pic(X)}
\end{equation}
et finalement
\begin{equation}
P_{b_2(X),n}\left(\wedgeo{j} H_{\ell}^{2\,\dim(X)-2}(X)\right)
=\ell^{\,n\,(\dim(X)-1)}\,P_{b_2(X),n}\left(\classe{\wedgeo{j} \Pic(X)}\right).
\end{equation}

Ainsi, on a
\begin{align}
\frac{\Poincl (\Phi_n(X))}
{\ell^{n\,\dim(X)}\,u^{\,2\,n\,\dim(X)}}
&
= 
1
+
\ell^{\,-n}\,
P_{b_2(X),n}
\left(\classe{\wedgeo{j} \Pic(X)}
\right)\,u^{\,-2\,n}
\notag\\
&
\phantom{=}+u^{-3\,n}\,\ell^{\,-n\,\dim(X)}\,Q_0(u^{-1}).
\label{eq:poincellphin}
\end{align}

Par ailleurs, il existe  un \'el\'ement $Q_1$ de $\kogkqlq[[u]]$
tel qu'on ait
\begin{multline}
\prod_{\cC\in \bconjc{G}}
\polcar_{\rhoNS,\{e\},\cC}(\ell^{-n}\,u^{-2\,n})^{\,\chil(\etak{k',G,\cC,n})}\,
\\
=
1-\ell^{-n}\,\sum_{\cC\in \bconjc{G}}\chi_{\rho}(\cC)\,\chil(\etak{k',G,\cC,n})\,u^{-2\,n}
+u^{-4\,n}\,Q_1(u^{-1}).
\label{eq:prod:polcar}
\end{multline}
D'apr\`es la relation \eqref{eq:rel:trace:rho} appliqu\'ee \`a $\rhoNS$, on a pour tout
$n$ la relation
\begin{equation}\label{eq:rel:trace}
\sum_{\cC\in \bconjc{G}}\chi_{\rhoNS}(\cC)\,\chil(\etak{k',G,C,n})=
P_{b_2(X),n}
\left(\symb{\wedgeo{j} \Pic(X)}
\right)_{1\leq
  j\leq b_2(X)}.
\end{equation}
De \eqref{eq:poincellphin}, \eqref{eq:prod:polcar}
et \eqref{eq:rel:trace} on d\'eduit l'in\'egalit\'e \ref{eq:degpnlx},
et donc le th\'eor\`eme \ref{thm:princ}.
\subsection{D\'emonstration du th\'eor\`eme \ref{thm:princ:spe}}
\label{subsec:dem:thm:princ:spe}
\begin{notas}
Pour toute place $\mfp$ de $k$ non ramifi\'ee dans l'extension $k'/k$,
notons $\cC_{\mfp}$ la classe des groupes de d\'ecomposition de $\mfp$ dans
l'extension $k'/k$, et pour $n\geq 1$, $\cC_{\mfp,n}$ l'unique
classe de $\bconjc{G}$ tel qu'il existe $C\in \cC_{\mfp,n}$ et
$C_{\mfp}\in \cC_{\mfp}$ 
v\'erifiant
$C<C_{\mfp}$ et $\card{C}=\frac{\card{C_{\mfp}}}{n\wedge \card{C_{\mfp}}}$.
\end{notas}
\begin{lemme}\label{lm:thm:princ:spe}
Soit $b(X)$ le plus grand nombre de Betti de $X$.
Il existe un ensemble fini
$S$ de places finies de $k$ (d\'ependant de $X$ et de $\courbe$),
tel que pour tout $\mfp\notin S$ 
on a les
propri\'et\'es suivantes :
\begin{enumerate}
\item
Pour tout $n\geq 1$,
$\Poincl (\Phi_n(X))$, 
$\Poincl (\psinx{\courbe})$ 
  et $P_{X,\ell,n}(u)$ sont purs en  $\mfp$.
\item
$\Poincl(\Volm(X\times \courbe))$ est pur en $\mfp$.
\item
Pour tout $n\geq 1$, il existe des nombres alg\'ebriques
$(\alpha_{\mfp,n,r})_{r\geq 3}$ v\'erifiant
\begin{equation}\label{eq:trplogpnl:maj}
\forall r\geq 3,\quad \abs{\alpha_{\mfp,n,r}}\leq \frac{2\,(\dim(X)+b_2(X))\,(1+b(X))^r}{r}
\end{equation}
et
\begin{equation}\label{eq:trplogpnl}
\Trp\left(\log(P_{X,\ell,n}(u))\right)
=
\sum_{r\geq 3} \alpha_{\mfp,n,r} N(\mfp)^{-\frac{n\,r}{2}}\,u^{\,-n\,r}.
\end{equation}
\item
\label{item:lm:thm:princ:spe:psin:courbe}
Pour tout $n\geq 1$, il existe des nombres alg\'ebriques
$(\beta_{\mfp,n,r})_{0\leq r\leq 2\,n}$
v\'erifiant
\begin{equation}
\abs{\beta_{\mfp,n,r}}\leq \frac{b_1(\courbe)+1}{n}
\end{equation}
et 
\begin{equation}
\Trp\left(\Poincl (\psinx{\courbe})\right)
=
\sum_{0\leq r \leq 2\,n} 
\beta_{\mfp,n,r}\,N(\mfp)^{\,\frac{r}{2}}\,u^r.
\end{equation}
\item
Il existe des nombres alg\'ebriques $(\gamma_{\mfp,r})_{r\geq 1}$
v\'erifiant
\begin{equation}
\abs{\gamma_{\mfp,r}}\leq 6\,r\,(b_1(\courbe)+1)\,(\dim(X)+b_2(X))\,(1+b(X))^r
\end{equation}
et
\begin{equation}
\Trp\left[\sum_{n\geq 1} \Poincl (\psinx{\courbe})\,\log(P_{X,\ell,n})\right]
=1+\sum_{r \geq 1} \gamma_{\mfp,r}\,N(\mfp)^{-\frac{r}{2}}\,u^{-r}.
\end{equation}
\item
La s\'erie enti\`ere
$
\Trp(\Poincl[\Volm(X\times\courbe)])\in \bC[[u^{-1}]]
$
converge absolument pour tout $u=z\in \bC$ v\'erifiant
\begin{equation}
\abs{z}> (1+b(X))\,N(\mfp)^{-\frac{1}{2}}
\end{equation}
et sa somme vaut alors
\begin{equation}
\exp\left[\sum_{n\geq 1}\Trp[\Poincl(\psinx{\courbe})]_{u=z}\,
\log
\left(
\polcar_{\rhoNS,\cC_{\mfp,n}}(N(\mfp)^{-n}\,z^{-2\,n})
\frac{\Trp[\Poincl(\Phi_n(X))]_{u=z}}
{N(\mfp)^{n\,\dim(X)}\,z^{\,2\,n\,\dim(X)}}
\right)
\right]
.
\end{equation}
\end{enumerate}
\end{lemme}
\begin{proof}
Notons $S$ la r\'eunion des places finies $\mfp$ qui v\'erifient l'une des conditions
suivantes :
\begin{enumerate}
\item
$\mfp$ appartient \`a l'ensemble fini du point \ref{item:prop:trp:eta:spec}
de la proposition \ref{prop:trp:eta} ;
\item
il existe  $\cC\in \bconjc{G}$ tel que  
$\chil\left(\frev{\Spec(k'),\Gop,\cC}\right)=\Poincl \left(\frev{\Spec(k'),\Gop,\cC}\right)$
 n'est pas pur en $\mfp$ ;
\item
$\Poincl (X)$ n'est pas pur en $\mfp$ ;
\item
$\Poincl (\courbe)$ n'est pas pur en $\mfp$ ;
\item
$\mfp$ est ramifi\'e dans l'extension $k'/k$ ;
\item
$\mfp$ divise $\ell$.
\end{enumerate}
Pour $\mfp\notin S$ et $n\geq 1$, la relation \eqref{eq:def:eta}
montre  que $\chil(\etak{k',G,\cC,n})$ est pur en $\mfp$.
Par ailleurs, pour $\mfp\notin S$ et $n\geq 1$, 
la proposition \ref{prop:poinchPhi_d} montre que 
$\Poincl(\Phi_n(X))$  et
$\Poincl(\Phi_n(\courbe))$ sont pur en $\mfp$.
La relation \eqref{eq:relpsinphin} montre alors que 
$\Poincl(\psinx{\courbe})$ est pur en $\mfp$.
On en d\'eduit que $P_{X,\ell,n}(u)$ est pur en $\mfp$, puis que 
$\Poincl(\Volm(X\times \courbe))$ est pur en $\mfp$.

Soit $\mfp\notin S$. 
Posons $v=N(\mfp)^{-\frac{1}{2}}\,u^{-1}$.
Soit $n\geq 1$.  
Comme $\Poincl (\Phi_n(X))$ est pur en $\mfp$, on a 
\begin{align}
\Trp\left(\frac{\Poincl (\Phi_n(X))}{\ell^{\,n\,\dim(X)}\,u^{\,2\,n\,\dim(X)}}
\right)\hskip-0.3\textwidth&
\\
&=
\sum_{0\leq r\leq 2\,\dim(X)}
\Tr(\Fr_{\mfp}^n|H^{2\,\dim(X)-r}(X)) \,N(\mfp)^{\,-n\,\frac{2\,\dim(X)-r}{2}}\,v^{\,n\,r}
\\
&=
\sum_{0\leq r\leq 2\,\dim(X)}
a_{\mfp,n,r}\,v^{\,n\,r}\label{eq:poinclphin:sum:apnr}
\end{align}
o\`u les~$a_{\mfp,n,r}$ sont des  nombres alg\'ebriques v\'erifiant
$a_{\mfp,n,r}=a_{\mfp,n,2\,\dim(X)-r}$
et 
\begin{equation}
\forall \,0\leq r\leq 2\,\dim(X), \quad 
\abs{\alpha_{\mfp,n,r}}\leq b(X).
\end{equation}

Il d\'ecoule de la proposition \ref{prop:trp:eta}
et de \eqref{eq:poinclphin:sum:apnr}
que l'on a 
\begin{equation}
\Trp(P_{X,\ell,n}(u))
=
\polcar_{\rhoNS,\cC_{\mfp,n}}(N(\mfp)^{-n}\,u^{-2\,n})\,
\left(
\sum_{0\leq r\leq 2\,\dim(X)}
a_{\mfp,n,r}\,v^{\,n\,r}
\right).
\end{equation}
On voit ainsi que
$
\Trp(P_{X,\ell,n}(u))
$
s'\'ecrit 
$
Q_{1,\mfp,n}(v^n)\,Q_{2,\mfp,n}(v^n)
$
o\`u $Q_{1,\mfp,n}$ est un polyn\^ome de degr\'e $2\,b_2(X)$ dont les
racines sont de module $1$ 
et $Q_{2,\mfp,n}$ est un polyn\^ome r\'eciproque de degr\'e $2\,\dim(X)$ dont les
racines sont de module inf\'erieur \`a $1+b(X)$.
On en d\'eduit \eqref{eq:trplogpnl} et \eqref{eq:trplogpnl:maj}.

Montrons le point \ref{item:lm:thm:princ:spe:psin:courbe}.
D'apr\`es la proposition \ref{prop:poinchPhi_d:mot} et la remarque
\ref{rem:trfpdimvm:trfn}, 
on a pour tout $d\geq 1$
\begin{equation}
\Tr\left(\Fr_{\mfp}|\Poincl (\Phi_d(\courbe))\right)=1+(-1)^{d+1}\,\Tr(\Fr_{\mfp}^d|\Hl^1(\courbe))\,u^d+N(\mfp)^{d}\,u^{2\,d}.
\end{equation}
Par ailleurs, comme $\Poincl (\Phi_d(\courbe))$ est pur en $\mfp$  on a 
\begin{equation}
\forall d\geq 1,\quad \abs{\Tr(\Fr_{\mfp}^n|\Hl^1(\courbe))}\,N(\mfp)^{-\frac{d}{2}}\leq b_1(\courbe).
\end{equation}
D'apr\`es la relation \eqref{eq:relpsinphin}, on a
\begin{equation}
\forall n\geq 1,\quad
\psinx{\courbe}=\frac{1}{n} \sum_{d|n} \mu\left(\frac{d}{n}\right)\,\Phi_d(\courbe),
\end{equation}
o\`u $\mu\,:\,\bN\to \{0,1,-1\}$ est la fonction de M\"obius.
On a donc pour tout $n\geq 1$
\begin{align}
\Trp\left(\Poincl (\psinx{\courbe})\right)
\hskip-0.2\textwidth&\notag\\
&=\frac{1}{n}
\sum_{\substack{d|n \\ d\text{ impair}}} 
\,(-1)^{d+1}\,\mu\left(\frac{n}{d}\right) \Tr(\Fr_{\mfp}^d|\Hl^1(\courbe))\,N(\mfp)^{-\frac{d}{2}}\,v^d
\notag\\
&\phantom{=}
+
\frac{1}{n}
\sum_{\substack{d|n \\ d\text{ pair}}}
\,\left[(-1)^{d+1}\,\mu\left(\frac{d}{n}\right)
  \Tr(\Fr_{\mfp}^d|\Hl^1(\courbe))\,N(\mfp)^{-\frac{d}{2}}
+\mu\left(\frac{2\,n}{d}\right)
\right]
\,v^d
\end{align}
d'o\`u le r\'esultat annonc\'e. Les deux derniers points en d\'ecoulent ais\'ement.
\end{proof}

Montrons \`a pr\'esent le th\'eor\`eme \ref{thm:princ:spe}.
Tout d'abord, d'apr\`es le lemme \ref{lm:iso:pic:pic}, il existe un
ensemble fini $S'$ de places de $k$ tel que pour  toute place
$\mfp\notin S'$, on a un isomorpisme
\begin{equation}
\Pic(\clo{X})\isom \Pic(\clo{X_{\mfp}})
\end{equation}
compatible aux actions de $\galabs{k}$ \`a gauche et
$\galabs{\kappa_{\mfp}}$ \`a droite.
Pour $\mfp\notin S'$, soit $G_{\mfp}$ un groupe de d\'ecomposition de
$\mfp$ dans l'extension $k'/k$, $\rhoNSp$ la $\bQ$-repr\'esentation de
$G_{\mfp}$ induite par l'action de $\galabs{\kappa_{\mfp}}$ sur
$\Pic(\clo{X})$, $\kappa'_p$ l'extension galoisienne de
groupe $G_{\mfp}$ de $\kappa_{\mfp}$ et $\scD(\mfp)\eqdef
\courbe_{\mfp}\times_{\kappa_p}\kappa'_p$.
Le volume de Tamagawa  de la famille
$X_{\mfp}\times \courbe_{\mfp}/\courbe_{\mfp}$ est donc \'egal \`a
\begin{equation}\label{eq:vol:tam:naif:1}
N(\mfp)^{\,(1-g(\courbe))\,\dim(X)}\,
\prod_{n\geq 1} \,
\prod_{
\substack{
C<G_{\mfp}
}
}
\polcar_{\rhoNSp,\{e\},C}(N(\mfp)^{-n})^{\,\card{\scD(\mfp)^{(0)}_{G_{\mfp},\{e\},C,n}}}
\left(
\frac{\card{X(\kappa_{\mfp,n})}}{N(\mfp)^{\,n\,\dim(X)}}
\right)
^{\card{(\courbe_{\mfp})^{(0)}_{n}}}
\end{equation}
Comme $\scD_{\mfp}=(\courbe\times_k k')_{\mfp}$ est isomorphe \`a la r\'eunion disjointe de
$G/G_{\mfp}$ copies de $\scD(\mfp)$, par compatibilit\'e \`a la
restriction des fonctions $L$ d'Artin classique 
on a l'\'egalit\'e
\begin{equation}
\Vol(X_{\mfp}\times \courbe_{\mfp}/\courbe_{\mfp})
=
N(\mfp)^{\,(1-g(\courbe))\,\dim(X)}\,
\prod_{n\geq 1} \,
\prod_{
\substack{
C<G
}
}
\polcar_{\rhoNS,\{e\},C}(N(\mfp)^{-n})^{\,\card{(\scD_{\mfp})^{(0)}_{G,\{e\},C,n}}}
\left(
\frac{\card{X(\kappa_{\mfp,_n})}}{N(\mfp)^{\,n\,\dim(X)}}
\right)
^{\card{(\courbe_{\mfp})^{(0)}_{n}}}
\end{equation}
D'apr\`es la remarque \ref{rem:triv}, on a donc
\begin{equation}\label{eq:vol:tam:naif:0}
\Vol(X_{\mfp}\times \courbe_{\mfp}/\courbe_{\mfp})
=
N(\mfp)^{\,(1-g(\courbe))\,\dim(X)}\,
\prod_{n\geq 1} \,
\left[\polcar_{\rhoNS,\cC_{\mfp,n}}(N(\mfp)^{-n})
\frac{\card{X(\kappa_{\mfp,_n})}}{N(\mfp)^{\,n\,\dim(X)}}
\right]
^{\card{(\courbe_{\mfp})^{(0)}_{n}}}.
\end{equation}

Soit $S''$ l'ensemble fini de places de $k$ constitu\'e de la r\'eunion de
l'ensemble $S$ du lemme \ref{lm:thm:princ:spe}, des places $\mfp$
v\'erifiant $N(\mfp)\leq (1+b(X))^2$
et de l'ensemble $S'$ introduit ci-dessus.
Il d\'ecoule alors  du lemme  \ref{lm:thm:princ:spe} que pour tout
$\mfp\notin S''$ la s\'erie $\Poincl\left(\Volm(X\times \courbe)\right)$
est pur en $\mfp$ et que la s\'erie enti\`ere
\begin{equation}
N(\mfp)^{\,(g(\courbe-1))\,\dim(X)}\,
\Trp\left[\Poincl \left(\Volm(X\times
    \courbe)\right)\right]\in\bC[[u^{-1}]]
\end{equation}
converge absolument en $u=-1$  vers
\begin{multline}\label{eq:somme}
\exp\left[\sum_{n\geq 1}\Trp[\Poincl(\psinx{\courbe})]_{u=-1}\,
\log
\left(
\polcar_{\rhoNS,\{e\},\cC_{\mfp\!,\!n}}(N(\mfp)^{-n})
\frac{\Trp[\Poincl(\Phi_n(X))]_{u=-1}}
{N(\mfp)^{n\,\dim(X)}}
\right)
\right]
\\
=
\exp\left[\sum_{n\geq 1}\Trp[\chil(\psinx{\courbe})]\,
\log
\left(
\polcar_{\rhoNS,\{e\},\cC_{\mfp\!,\!n}}(N(\mfp)^{-n})
\frac{\Trp[\chil(\Phi_n(X))]}
{N(\mfp)^{n\,\dim(X)}}
\right)
\right]
.
\end{multline}

Mais d'apr\`es le corollaire \ref{cor:prop:spec:psin}, on a 
$
\Trp[\chil\left(\psi_n(\courbe)\right)]=\card{\left(\courbe_{\mfp}\right)_n^{(0)}}
$
et
$\Trp[\chil\left(\Phi_n(X)\right)]=\card{X_{\mfp}(\kappa_{\mfp,n})}$.
On en d\'eduit que l'expression \eqref{eq:somme}  co\"\i ncide bien 
avec
$N(\mfp)^{\,(g(\courbe)-1)\,\dim(X)}\,\Vol(X_{\mfp}\times \courbe_{\mfp}/\courbe_{\mfp})$.

\subsection{D\'emonstration du th\'eor\`eme \ref{thm:princ:spe:fin}}
\label{subsec:dem:thm:princ:spe:fin}
\begin{notas}
Pour tout entier $m\geq 1$, 
notons $C_{m}$ le groupe de d\'ecomposition de $k_m/k$ dans
l'extension $k'/k$, et pour $n\geq 1$, $C_{m,n}$ le
sous-groupe de $C_{m}$ 
v\'erifiant $\card{C_{m,n}}=\frac{\card{C_{m}}}{n\wedge \card{C_{m}}}$.
\end{notas}
\begin{lemme}\label{lm:thm:princ:spe:fin}
Soit $b(X)$ le plus grand nombre de Betti de $X$.
On a alors :
\begin{enumerate}
\item
Pour tout $n\geq 1$, il existe des nombres alg\'ebriques
$(\alpha_{n,r})_{r\geq 3}$ v\'erifiant
\begin{equation}\label{eq:trplogpnl:maj:fin}
\forall r\geq 3,\quad \abs{\alpha_{n,r}}\leq \frac{2\,(\dim(X)+b_2(X))\,(1+b(X))^r}{r}
\end{equation}
et pour tout $m\geq 1$
\begin{equation}\label{eq:trplogpnl:fin}
\Tr\left(F_k^{m}|\log(P_{X,\ell,n}(u))\right)
=
\sum_{r\geq 3} \alpha_{n,r} q^{-\frac{n\,m\,r}{2}}\,u^{\,-n\,r}.
\end{equation}
\item
\label{item:lm:thm:princ:spe:psin:courbe:fin}
Pour tout $n\geq 1$, il existe des nombres alg\'ebriques
$(\beta_{n,r})_{0\leq r\leq 2\,n}$
v\'erifiant
\begin{equation}
\abs{\beta_{n,r}}\leq \frac{b_1(\courbe)+1}{n}
\end{equation}
et pour tout $m\geq 1$
\begin{equation}
\Tr\left(F_k^m|\Poincl (\psinx{\courbe})\right)
=
\sum_{0\leq r \leq 2\,n} 
\beta_{n,r}\,q^{\,\frac{m\,r}{2}}\,u^r.
\end{equation}
\item
Il existe des nombres alg\'ebriques $(\gamma_{r})_{r\geq 1}$
v\'erifiant
\begin{equation}
\abs{\gamma_{r}}\leq 6\,r\,(b_1(\courbe)+1)\,(\dim(X)+b_2(X))\,(1+b(X))^r
\end{equation}
et pour tout $m\geq 1$
\begin{equation}
\Tr\left[F^m_k|\sum_{n\geq 1} \Poincl (\psinx{\courbe})\,\log(P_{X,\ell,n})\right]
=1+\sum_{r \geq 1} \gamma_{r}\,q^{-\frac{m\,r}{2}}\,u^{-r}.
\end{equation}
\item
Pour $m\geq 1$, la s\'erie enti\`ere
$
\Tr(\Poincl[F_k^m|\Volm(X\times\courbe)])\in \bC[[u^{-1}]]
$
converge absolument pour tout $u=z\in \bC$ v\'erifiant
\begin{equation}
\abs{z}> (1+b(X))\,q^{-\frac{m}{2}}
\end{equation}
et sa somme vaut alors
\begin{equation}
\exp\left[\sum_{n\geq 1}\Tr[F_k^m|\Poincl(\psinx{\courbe})]_{u=z}\,
\log
\left(
\polcar_{\rhoNS,\{e\},C_{m,n}}(q^{-n\,m}\,z^{-2\,n})
\frac{\Tr[F_k^m|\Poincl(\Phi_n(X))]_{u=z}}
{q^{n\,m\,\dim(X)}\,z^{\,2\,n\,\dim(X)}}
\right)
\right]
.
\end{equation}
\end{enumerate}
\end{lemme}
La d\'emonstration de ce lemme est tr\`es similaire \`a celle du lemme
\ref{lm:thm:princ:spe}.
Par un raisonnement analogue \`a celui de la section
\ref{subsec:dem:thm:princ:spe}, on en d\'eduit le th\'eor\`eme \ref{thm:princ:spe:fin}.

\subsection{Lien conjectural avec la fonction z\^eta des hauteurs antica\-noniques}\label{subsec:lien:conj}
On se place sous les hypoth\`eses du th\'eor\`eme \ref{thm:princ}.
On suppose en outre que le c\^one effectif de $\clo{X}$ est finiment engendr\'e
et que la classe du faisceau anticanonique de $X$ est
\`a l'int\'erieur du c\^one effectif.
Ces hypoth\`eses permettent de d\'efinir un invariant rationnel $\alpha^{\ast}(X)$
(cf. \cite[\S 3.1]{Pey:var_drap}).
On d\'efinit par ailleurs l'invariant $\beta(X)$ comme \'etant le cardinal du groupe $H^1(k,\Pic(\clo{X}))$.
Si $K$ d\'esigne le corps des fonctions de $\courbe$, on suppose
en outre que l'ensemble $X(K)$ est Zariski dense.
Par souci de simplification, on supposera \'egalement qu'on a
\begin{equation}
\Max \{d\in \bN_{>0},\quad \frac{1}{d}\classe{\omega_X^{-1}}\in \Pic(X)\}=1.
\end{equation}

Supposons tout d'abord que $k$ soit un corps fini de cardinal $q$.
Pour $U$ ouvert de Zariski de $X$ assez petit, on peut alors
consid\'erer la fonction z\^eta des hauteurs
anticanonique $\ZH(X\times \courbe/\courbe,U,t)$ : c'est la s\'erie
g\'en\'eratrice qui compte le nombre de morphismes de $\courbe$ vers $X$
de degr\'e anticanonique donn\'e dont l'image n'est pas incluse dans le
compl\'ementaire de $U$.
Voici une version de la conjecture de Manin dans ce cadre.
\begin{question}\label{ques:manin}
On suppose que $X\times \courbe/\courbe$ v\'erifie l'approximation
faible (par exemple que $X$ est rationnelle).
Est-il vrai que pour un ouvert $U$ assez petit, la s\'erie enti\`ere $\ZH(X\times \courbe/\courbe,U,t)$
a un rayon de convergence \'egal \`a $q^{-1}$ et que pour un certain
$\eps>0$ sa somme se prolonge en une fonction m\'eromorphe sur
$\{\abs{t}<q^{-1+\eps}\}$, admettant un p\^ole d'ordre $\rg(\Pic(X))$ en
$t=q^{-1}$  tel que
\begin{multline}
\lim_{t\to q^{-1}}(1-qt)^{\rg(\Pic(X))}\ZH(X\times\courbe/\courbe,U,t)
\\
=\alpha^{\ast}(X)\,\beta(X)\,\left[(1-q\,t)^{\rg(\Pic(X))}\,\Lar(\scD,G,\rhoNS,t)\right]_{t=q^{-1}}\,\Vol(X\times \courbe/\courbe).
\end{multline}
\end{question}
Nous donnons ci-dessous un pendant motivique de la version affaiblie
suivante de la question \ref{ques:manin}.
\begin{question}\label{ques:manin:bis}
On suppose que $X\times \courbe/\courbe$ v\'erifie l'approximation faible.
Est-il vrai que pour un ouvert $U$ assez petit, la s\'erie 
\begin{equation}\label{eq:dettimeszh}
\det(\Id-F_k\,t|\Pic(\clo{X}))\,\det(\Id-q\,F_k\,t|\Pic(\clo{X}))\,\ZH(X\times \courbe/\courbe,t)
\end{equation}
converge en $t=q^{-1}$ vers 
\begin{equation}
\alpha^{\ast}(X)\,\beta(X)\,\left[\det(\Id-F_k\,t|\Pic(\clo{X}))\,\det(\Id-q\,F_k\,t|\Pic(\clo{X}))\,\Lar(\scD,G,\rhoNS,t)\right]_{t=q^{-1}}\,
\Vol(X\times \courbe/\courbe)\quad ?
\end{equation}
\end{question}
Revenons au cas d'un corps $k$ quelconque.
Pour $U$ ouvert de Zariski de $X$ assez petit, on peut alors
consid\'erer la fonction z\^eta des hauteurs
anticanonique g\'eom\'etrique $\ZHv(X\times \courbe/\courbe,U,t)$ : c'est une s\'erie formelle
dont les coefficients sont les classes dans $\kovark$ des espaces de
modules param\'etrant les morphismes de $\courbe$ vers $X$
de degr\'e anticanonique donn\'e dont l'image n'est pas incluse dans le
compl\'ementaire de $U$. Lorsque $k$ est un corps fini, le morphisme 
\og nombre de points \fg~envoie $\ZHv$ sur $\ZH$.
Si $k$ est de caract\'eristique z\'ero, on peut consid\'erer la fonction
z\^eta des hauteurs motiviques $\ZHm\eqdef \chiv(\ZHv)$.
Par analogie avec la question \ref{ques:manin:bis}, on peut alors poser la question suivante.
\begin{question}\label{ques:manin:mot}
Supposons $k$ de caract\'eristique z\'ero.
Est-il vrai que pour un ouvert $U$ assez petit la s\'erie
\begin{equation}
\Zm(\Pic(\clo{X}),t)^{-1}\,\Zm(\Pic(\clo{X}),t)^{-1}\,\ZHm(X\times \courbe/\courbe,U,t)
\end{equation}
converge dans $\hml$ en $t=\bL^{-1}$ vers 
\begin{equation}
\alpha(X)\,\beta(X)\,
\left[
\Zm(\Pic(\clo{X}),t)^{-1}\,\Zm(\Pic(\clo{X}),\bL\,t)^{-1}
\Lm(\scD,G,\rhoNS,t)\right]_{t=\bL^{-1}}\,
\Volm(X\times \courbe/\courbe)\quad ?
\end{equation}
\end{question}
Les arguments d\'evelopp\'es dans \cite{Bou:prod:eul:mot}
montrent que la r\'eponse \`a cette question est positive dans le cas o\`u
$X$ est une vari\'et\'e torique d\'eploy\'ee et $\courbe=\bP^1$.
\begin{rem}\label{rem:enonce:question}
On pourrait imaginer renforcer la question \ref{ques:manin:bis} en demandant
en outre la convergence de la s\'erie \eqref{eq:dettimeszh} en $t=q^{-1+\eps}$
pour $\eps>0$ assez petit. Ceci aurait deux avantages :
d'une part une telle convergence impliquerait une r\'eponse positive \`a
la question \ref{ques:manin}, d'autre part l'adaptation au cadre
motivique serait ais\'ee (quitte \`a introduire formellement des racines
de $\bL$ dans l'anneau de Grothendieck des motifs).
Cependant une r\'eponse positive \`a la question ainsi reformul\'ee
entra\^\i nerait en outre que les p\^oles
de la fonction z\^eta des hauteurs sur le cercle de rayon $q^{-1}$ 
sont inclus dans l'ensemble $\{\alpha^{-1}\,q^{-1}\}$, $\alpha$
d\'ecrivant les valeurs propres de $F_k$ sur $\Pic(\clo{X})$.
Ceci n'est pas v\'erifi\'e par exemple dans le cas du plan projectif
\'eclat\'e en un point (o\`u $q^{-1}$ n'est pas l'unique p\^ole du prolongement m\'eromorphe sur le cercle de rayon
$q^{-1}$).
La question de la nature
des p\^oles qui doivent appara\^\i tre sur le cercle de rayon $q^{-1}$ reste
\`a \'etudier.
\end{rem}

Supposons \`a pr\'esent que $k$ soit un corps de nombres
et indiquons comment les questions \ref{ques:manin:bis} et
\ref{ques:manin:mot} pourraient \^etre reli\'ees.
On suppose qu'il existe  un ouvert $U$ tel que la r\'eponse \`a la question \ref{ques:manin:mot} soit positive.
Notons, pour all\'eger l'\'ecriture,
\begin{equation}
\wt{\ZHm}(t)\eqdef\Zm(\Pic(\clo{X}),t)^{-1}\,\Zm(\Pic(\clo{X}),\bL\,t)^{-1}\,
\ZHm(X\times \courbe/\courbe,U,t)\in \mk[[t]]
\end{equation}
et
\begin{equation}
\wt{\Lm}(t)\eqdef\Zm(\Pic(\clo{X}),t)^{-1}\,\Zm(\Pic(\clo{X}),\bL\,t)^{-1}\,
\Lm(\scD,G,\rhoNS,t)\in \mk[[t]].
\end{equation}
Ainsi, d'apr\`es le point \ref{item:3:prop:lmdmcourbe} de la proposition 
\ref{prop:lmdmcourbe}, $\wt{\Lm}(t)$ est un polyn\^ome.

Supposons en outre que pour presque tout $\mfp$
on ait
$\Trp(\chil(\wt{\ZHm}(t))=\wt{\ZHp}(t)$ o\`u
\begin{equation} 
\wt{\ZHp}(t)\eqdef
\det(\Id-\Fr_{\mfp}\,t|\Pic(\clo{X_{\mfp}}))\,\det(\Id-N(\mfp)\,\Fr_{\mfp}\,t|\Pic(\clo{X_{\mfp}}))
\,
\ZH(X_{\mfp}\times
\courbe_{\mfp}/\courbe_{\mfp},U_{\mfp},t)\in \bC[[t]].
\end{equation}
D'apr\`es le th\'eor\`eme \ref{thm:princ:spe}, on est dans la situation
suivante :
\begin{equation}
\xymatrix{
\underset{\in\kogkqlq[u,u^{-1}][[t]]}{\Poincl(\wt{\ZHm}(t))}
\ar@{|->}[rrr]^>>>>>>>>>>>>>>>{t=\ell^{-1}u^{-2}}
\ar@{|->}[d]^{\Trp}
&
&
&
\underset{
\in \kogkqlq((u^{-1}))
}
{
\alpha(X)\beta(X)
\wt{\Lm}(\ell^{-1}\,u^{-2})
\Volm(X\times \courbe/\courbe)
}
\ar@{|->}[d]^{\Trp}
\\
\underset{\in\bC[u,u^{-1}][[t]]}{\Trp(\,.\,)}
\ar@{|->}[rrr]^>>>>>>>>>>>>>>>>>>>>>>>>>>>>>>>>{t=N(\mfp)^{-1}u^{-2}}\ar@{|->}[d]^{u=-1}
&
&
&
\underset{\in\bC((u^{-1}))}{\Trp(\,.\,)}
\ar@{|->}[d]^{u=-1}
\\
\underset{\in \bC[[t]]}
{\wt{\ZHp}(t)}\ar@{|.>}[rrr]^>>>>>>>>>>>>>>>{t=N(\mfp)^{-1}\,?}
&
&
&
\alpha^{\ast}(X)\beta(X)
\wt{\Lar}(N(\mfp)^{-1})
\Vol(X_{\mfp}\times \courbe_{\mfp}/\courbe_{\mfp})
}
\end{equation}
Si on arrive \`a montrer que la fl\`eche horizontale inf\'erieure est bien d\'efinie et fait 
\og commuter \fg le carr\'e inf\'erieur, 
on obtient que la  r\'eponse \`a la question \ref{ques:manin:bis} est
positive en $\mfp$. 
Concr\`etement, on est ramen\'e \`a un probl\`eme d'interversion de s\'erie.
Il s'agirait alors de d\'egager les propri\'et\'es de $\wt{\ZHm}$ assurant
que cette interversion est licite.

\subsection{Une vraie version motivique}\label{subsec:vrai:motif}
L'appelation \og motivique \fg pour le volume de Tamagawa dont
l'existence est montr\'ee par le th\'eor\`eme \ref{thm:princ} 
est abusive au vu de la compl\'etion utilis\'ee.
Il devrait plut\^ot \^etre qualifi\'e de \og cohomologique\fg.
Si on admet la conjecture que tout motif de Chow admet une
d\'ecomposition de Chow-K\"unneth (cf. \cite{Mur:conj}), on peut d\'efinir un polyn\^ome de Poincar\'e
virtuel \og absolu\fg
\begin{equation}
\Poincabs\,:\,\map{\kochmk{\bQ}}{\kochmk{\bQ}[u,u^{-1}]}{\symb{M}}{\sumu{i\in \bZ} \symb{h^i(M)}\,u^i}
\end{equation}
et on peut se demander si la convergence du nombre de Tamagawa
motivique a lieu dans~$\hmabs$ (et pas seulement
dans une compl\'etion li\'ee \`a une r\'ealisation cohomologique). En fait on peut montrer un
r\'esultat de convergence inconditionnel pour les surfaces : soit $\mathscr{S}_k$ la
sous-cat\'egorie pleine de~$\chmk{\bQ}$ dont les objets sont les motifs
d\'ecoup\'es sur les vari\'et\'es de dimension au plus $2$, leurs sommes et
leurs duaux.
Comme les vari\'et\'es de dimension au plus $2$ admettent des
d\'ecomposition de Chow-K\"unneth (cf. \cite{Mur:surf}), on peut
d\'efinir un polyn\^ome de Poincar\'e virtuel absolu
\begin{equation}
\Poincabs\,:\,\map{K_0(\mathscr{S}_k)}{K_0(\mathscr{S}_k)[u,u^{-1}]}{\symb{M}}{\sumu{i\in \bZ} \symb{h^i(M)}\,u^i}
\end{equation}
\begin{thm}\label{thm:princ:mot}
Soit $k$ un corps de caract\'eristique z\'ero.
Soit $\courbe$ une courbe projective, lisse et g\'eom\'etriquement int\`egre.
Soit $S$ une surface projective, lisse et g\'eom\'etriquement int\`egre
v\'erifiant les hypoth\`eses \ref{hyp:X}. 
On suppose en outre que $A_0(S_{k(S)})$ est nul.
On reprend les notations \ref{notas:ecD}.
Le produit eulerien motivique
\begin{equation}\label{eq:def:tam:mot:surf}
\bL^{\,2\,(1-g(\courbe))}
\,
\prod_{n\geq 1} 
\left[
\prod_{\cC\in \bconjc{G}}
\polcar_{\rhoNS,\{e\},\cC}(\bL^{-n})^{\etak{k',G,\cC,n}}\,\frac{\Phi_n(S)}{\bL^{2\,n}}
\right]^{\psinx{\courbe}}
\end{equation}
converge dans $\wh{K_0(\mathscr{S}_k)\otimes \bQ}^{\,\,\Poincabs}$ (cf. la notation \ref{nota:hatAphi}).  
\end{thm}
\begin{rem}
L'hypoth\`ese que $A_0(S_{k(S)})$ est nul est v\'erifi\'ee d\`es que
$A_0(S_{\clo{k(S)}})$ est nulle. Ceci vaut en particulier si $S$ est
$\clo{k(S)}$-rationnellement connexe, et donc si $S$ est une surface
de Fano.
\end{rem}
\begin{proof}
Comme $H^1(S,\ecO_S)=0$, la vari\'et\'e d'Albanese de $S$ est triviale.
D'apr\`es \cite[Propositions 14.2.1, 14.2.3 et Corollary 14.4.9 (a)]{KMP:trans}, on a une d\'ecomposition 
\begin{equation}
h(S)=\ind\oplus \Pic(\clo{S})(-1) \oplus t^2(S) \oplus \ind(-2)
\end{equation}
o\`u $t^2(S)$ est un motif de poids $2$ qui est nul si et seulement si
$A_0(S_{k(S)})$ est nul.
Ainsi on a 
\begin{equation}
\Zm(S)=\Zm(\ind,t)\,\Zm(\Pic(\clo{S}),\bL t)\,\Zm(\ind,\bL^2\,t)
\end{equation}
d'o\`u
\begin{equation}
\Zm(S)^{-1}=(1-t)(\sum_{n\geq 0} (-1)^n \symb{\Alt^n(\Pic(\clo{S}))}\,\bL^n t^n)\,(1-\bL^2,t)).
\end{equation}
On en d\'eduit l'analogue pour le polyn\^ome de Poincar\'e virtuel absolu
de la proposition \ref{prop:poinchPhi_d:mot} : pour tout $n\geq 1$, on a
\begin{equation}
\Poincabs(\Phi_n(S))
=1+ P_{b_2(S),n}\left(\symb{\wedgeo{j} \Pic(\clo{S})}\right)_{1\leq j\leq
  b_2(S)} \,u^{2\,n}+\bL^{2\,n}\,u^{4\,n}
\end{equation}
\`A partir de l\`a, il est facile de v\'erifier que toutes les \'egalit\'es de la d\'emonstration 
du th\'eor\`eme  \ref{thm:princ} ont lieu dans $K_0(\mathscr{S}_k)\otimes \bQ$ (et pas
seulement dans $\kogkqlq$).
\end{proof}

\backmatter

\bibliographystyle{smfalpha}
\providecommand{\bysame}{\leavevmode ---\ }
\providecommand{\og}{``}
\providecommand{\fg}{''}
\providecommand{\smfandname}{\&}
\providecommand{\smfedsname}{\'eds.}
\providecommand{\smfedname}{\'ed.}
\providecommand{\smfmastersthesisname}{M\'emoire}
\providecommand{\smfphdthesisname}{Th\`ese}

\end{document}